\documentclass[11pt]{amsart}
\usepackage{mathrsfs}
\usepackage{amssymb}
\usepackage{amsthm}
\usepackage{titletoc}
\usepackage{amssymb}
\usepackage{amsxtra}
\usepackage{mathrsfs}
\usepackage{hyperref}
\newcommand{\diam}{\textrm{diam}}
\newcommand{\dlog}{\textrm{dlog}}
\newcommand{\var}{\textrm{var}}
\newtheorem{theorem}{Theorem}
\newtheorem{corollary}{Corollary}

\newtheorem{lemma}{Lemma}
\newtheorem{remark}{Remark}
\newtheorem{definition}{Definition}
\newtheorem{question}{Question}
\newtheorem{example}{Example}
\newtheorem{conjecture}{Conjecture}

\renewcommand{\thethmx}{\Alph{thmx}}
\title[Meromorphic functions sharing four values]{On a question of Gary G. Gundersen concerning meromorphic functions sharing three distinct values IM and a fourth value CM }
\author{Xiao-Min Li$^1$*, Qing-Fei Zhai$^2$, Hong-Xun Yi$^3$}
\address{$^{1,2}$ Department of Mathematics, Ocean University of China, Qingdao, Shandong 266100, P. R. China
\vskip 2mm \hspace{1.5mm} Email:{\sf lixiaomin@ouc.edu.cn(X.-M. Li), 1476157537@qq.com(Q.-F. Zhai).}
\vskip 2mm $^3$ Department of Mathematics, Shandong University, Jinan, Shandong 250199, P. R. China
\vskip 2mm \hspace{1.5mm} Email:{\sf hxyi@sdu.edu.cn(H.-X. Yi).}
}
\thanks{{\sf Corresponding author: Xiao-Min Li}}
\thanks{Project supported in part by the NSF of Shandong Province, China (No. ZR2019MA029) and
the FRFCU(No.3016000/842464005).}
\begin{document}
\begin{abstract}
In 1992, Gundersen \cite {Gundersen1992} proposed the following famous open question: if two non-constant meromorphic functions share three values
IM and share a fourth value CM, then do the functions necessarily share all four values CM? The open question is a long-standing question in the studies of   the Nevanlinna$'$s value distribution theory of meromorphic functions, and has not been completely resolved by now. In this paper, we prove the following result: suppose that $f$ and $g$ are two distinct non-constant meromorphic functions, and one of $f$ and $g$ has finite order. If $f$ and $g$ share $a_1,$ $a_2,$ $a_3$ IM and $a_4$ CM, where $a_1,$ $a_2,$ $a_3,$ $a_4$ are four distinct complex values in the extended complex plane, then $f$ and $g$ share $a_1,$ $a_2,$ $a_3$ and $a_4$ CM. Applying the main result obtained in this paper, we resolve a question proposed in Gundersen \cite[p.458]{Gundersen1979}  concerning the nonexistence of two distinct non-constant meromorphic functions sharing three distinct values DM and a fourth value CM for finite order meromorphic functions. The obtained result also extends the corresponding result in Mues\cite[pp.109-117]{Mues1989F1} concerning the nonexistence of two distinct and non-constant entire functions that share three distinct finite values DM.  Examples are provided to show that the main results obtained in this paper, in a sense, are best possible.
\end{abstract}
\vskip 2mm
\par
\maketitle
\section{Introduction and main results}
\vskip 2mm
\par In this paper, by meromorphic functions we will always mean meromorphic functions in the complex plane. We adopt the standard notations in the Nevanlinna theory of meromorphic functions as explained in \cite{Hayman1964, Laine1993, YangYi2003}. Throughout this paper, we denote by $E\subset[0, +\infty)$ a set of finite linear measure. For a meromorphic function $f$ in the complex plane we denote by $S(r,f)$ any quantity satisfying $S(r,f)=o(T(r,f)),$ as $r \not\in E$ and $r\rightarrow\infty.$ Following Halburd-Korhonen \cite[p.466]{HalburdKorhonen2006}, we define the notion of the small function of a non-constant meromorphic function in the complex plane as follows: let $f$ be a non-constant meromorphic function. Next we denote by $S(f)$ the set of all meromorphic functions $\alpha$ in the complex plane such that $T(r, \alpha) = o(T(r,f))$ for all $r\in [0, +\infty)$ possibly outside
a set $E\subset [0, +\infty)$ of finite linear measure. Functions in the set $S(f)$ are called small compared to $f,$ or slowly moving with respect to $f,$  or a small function of $f$ for short.
\vskip 2mm
\par
Let $f$ and $g$ be two non-constant meromorphic functions in the complex plane, and let $a$ be a value in the extended plane. We say that $f$ and $g$ share the value $a$ CM, provided that $f$ and $g$ have the same set of $a$-points, and each common $a$-point of $f$ and $g$  has the same multiplicities as the $a$-point of $f$ and $g.$ We say that $f$ and $g$ share the value $a$ IM, provided that $f$ and $g$ have the same set of $a$-points ignoring multiplicities. We say that $f$ and $g$ share the value $a$ DM, provided that $f$ and $g$ share $a$ IM, and at least one common $a$-point of $f$ and $g$ has different multiplicities related to $f$ and $g$ (cf.\cite{Gundersen1979,Gundersen1983}). Here we mention that Mues \cite {Mues1989F1} and Yang-Yi \cite{YangYi2003} defined the notion of the DM shared value as follows: we say that $f$ and $g$ share the value $a$ DM, provided that  $f$ and $g$ share $a$ IM, and that $f-a$ and $g-a$ have no common zeros of the same multiplicities. For a non-constant meromorphic function $f,$ we denote by $\mu(f)$ and $\rho(f)$ respectively the lower
order and the order of $f,$ the relevant definitions can be found, for example, in \cite{Hayman1964, Laine1993, YangYi2003}. For convenience, we recall them below:
\begin{definition}\label{definition1.1} \rm{} For a non-constant meromorphic function $f,$ the lower
order and the order of $f,$ denoted as $\mu(f)$ and $\rho(f)$ respectively are defined as
\begin{equation}\label{definition1.1}
\mu(f)=\liminf\limits_{r\rightarrow\infty}\frac{\log T(r,f)}{\log r} \, \, \text{and} \, \, \rho(f)=\limsup\limits_{r\rightarrow\infty}\frac{\log T(r,f)}{\log r} \, \, \text{ respectively}.
\end{equation}
\end{definition}
Following Edrei-Fuchs \cite[p.293]{EdreiFuchs1959} , we say that a meromorphic function $f$ in the complex plane is of regular growth, provided $\mu(f)=\rho(f).$ Here we mention that both of $\mu(f)$ and $\rho(f)$ may be $+\infty.$
\vskip 2mm
\par In 1926, Nevanlinna \cite{Nevanlinna1926} proved the following theorems that are the famous five-value theorem and the four-value theorem respectively:
\vskip 2mm
\par
\begin{theorem}\rm{}(\cite[p.109]{Nevanlinna1926}) \label{TheoremA}
If $f$ and $g$ are non-constant meromorphic functions that share five values IM, then $f=g$.
\end{theorem}
\vskip 2mm
\par
\begin{theorem}\rm{}(\cite[p.122]{Nevanlinna1926})\label{TheoremB}
Let $f$ and $g$ be distinct and non-constant meromorphic functions, and let $a_1,$ $a_2,$ $a_3,$ $a_4$ be four distinct complex values in the extended complex plane. If $f$ and $g$ share $a_1,$ $a_2,$ $a_3,$ $a_4$ CM, then $f$ is a M\"{o}bius transformation of $g$, two of the shared values, say $a_1$ and $a_2$, are Picard exceptional values, and the cross ratio $\left(a_1, a_2, a_3, a_4\right)=-1.$
\end{theorem}
\vskip 2mm
\par In 1976, L. Rubel posed the following question(cf.\cite{Gundersen1979}): whether do the four distinct CM shared values be replaced with four distinct IM shared values or not in Theorem \ref{TheoremB}? In this direction, Gundersen \cite{Gundersen1979} and  Gundersen \cite{Gundersen1983} proved the following results respectively, where Theorem \ref{TheoremD} improved Theorems \ref{TheoremB} and \ref{TheoremC}, and Theorem \ref{TheoremC} improved Theorem \ref{TheoremB}:
\vskip 2mm
\par
\begin{theorem}\rm{}(\cite[Theorem 1]{Gundersen1979})\label{TheoremC}
If $f$ and $g$ share three values CM and share a fourth value IM, then they share all four values CM.
\end{theorem}
\begin{theorem}\rm{}(\cite[Theorem 1]{Gundersen1983})\label{TheoremD}
If two non-constant meromorphic functions share two values CM and share two other values IM, then the functions share all four values CM.
\end{theorem}
\vskip 2mm
\par Gundersen \cite{Gundersen1979} gave the following famous example that shows that the L.Rubel's question above is negative:
\vskip 2mm
\par
\begin{example}\rm{}(\cite[pp.458-459]{Gundersen1979})\label{Example1.1}
Let $f=\frac{e^h+1}{(e^h-1)^2}$ and $g=\frac{(e^h+1)^2}{8(e^h-1)},$ where $h$ is a non-constant entire function, we can easily verify that $f$ and $g$ share $0,$ $1,$ $\infty,$ $-1/8$ DM, and $f$ is not a M\"{o}bius transformation of $g.$
\end{example}
\vskip 2mm
\par Based upon Example \ref{Example1.1} and Theorem \ref{TheoremD}, Gundersen (\cite{Gundersen1979, Gundersen1983, Gundersen1992}) proposed  the following famous question:
\vskip 2mm
\par
\begin{question}\rm{}(\cite{Gundersen1979, Gundersen1983, Gundersen1992})\label{Question1.1}
 If two non-constant meromorphic functions share three values IM and share a fourth value CM, then do the functions necessarily share all four values CM?
\end{question}
\vskip 2mm
\par Question \ref{Question1.1} is the main open question in the theory of meromorphic functions that share four values.  This open question appears to be difficult and has not been completely resolved by now. In the past several decades, there are many partial results on Question \ref{Question1.1}. For the ease of the following narrations, we introduce the following notation: let $f$ and $g$ be two non-constant meromorphic functions, and let $a$ be a value in the extended complex plane. Suppose that $f$ and $g$ share $a$ IM. Next we denote by $\overline {N}_E(r,a;f,g)$ the counting function of those common $a$-points of $f$ and $g$ in $|z|<r$ where each point in $\overline {N}_E(r,a;f,g)$ has the same multiplicities regard to $f$ and $g,$ and each point in $\overline {N}_E(r,a;f,g)$ is counted only once. We define
\begin{equation}\nonumber
\tau(a)=\liminf\limits_{r\rightarrow\infty}\frac{\overline {N}_E(r,a;f,g)}{\overline{N}\left(r,\frac{1}{f-a}\right)} \, \, \text{if} \, \, \overline{N}\left(r,\frac{1}{f-a}\right)\neq 0
\end{equation}
for the large positive number $r,$ where $\overline{N}\left(r,\frac{1}{f-a}\right)=\overline{N}(r,f)$ if $a=\infty.$ From the definition we see that $0\leq \tau(a)\leq 1.$ We mention that   $\tau(a)$ is defined as $\tau(a)=1,$ when $f$ and $g$ share $a$ CM. In 1989, Mues \cite[Theorem 1]{Mues1989} proved that if $f$ and $g$ are two distinct and non-constant meromorphic functions in the complex plane that share four distinct values in the extended complex plane, and one of the four distinct values is shared by $f$ and $g$ CM, and $\tau(b)>\frac{2}{3}$ for another value $b$ of the four distinct shared values, then $f$ and $g$ share all the four distinct values CM. In 1993, Wang \cite{Wang1993} improved \cite[Theorem 1]{Mues1989} and proved that if $f$ and $g$ are two distinct and non-constant meromorphic functions in the complex plane that share four distinct values in the extended complex plane, and two of the four distinct values, say $a$ and $b$ satisfy  $\tau(a)>\frac{4}{5}$ and $\tau(b)>\frac{4}{5},$ then $f$ and $g$ share all the four distinct values CM. As far as we know, there are other research works concerning Question \ref{Question1.1} and the uniqueness question of meromorphic functions that share four distinct values in the complex plane, such as Mues \cite{Mues1989F1,  Mues1995},  Gundersen \cite{Gundersen1992}, Reinders \cite{Reinders1993}, Ueda \cite{Ueda1988},  Wang\cite {Wang2001}, Steinmetz \cite{Steinmetz1988, Steinmetz2012}, Huang\cite{Huang2004} and Huang-Du\cite{HuangDu2004}, etc. In addition, Ishizaki\cite{Ishizaki2000}, Li \cite{Li2001} and Yao \cite{Yao2003} studied the uniqueness question of non-constant meromorphic functions sharing four distinct small functions. We recall the following  partial results in Gundersen \cite[Theorem 1]{Gundersen1992} and Li-Yi \cite{LiYi2007} on the open question respectively:
\vskip 2mm
\par
\begin{theorem}\rm{}(\cite[Theorem 1]{Gundersen1992})\label{TheoremE}
Let $f$ and $g$ be two non-constant meromorphic functions that share $a_l,$ $a_2,$ $a_3$
IM and $a_4$ CM, where $a_1,$ $a_2,$ $a_3$ and $a_4$ are four distinct complex values in the extended complex plane. Suppose that there exist some real constant $\lambda >4/5$ and some set $I \subset(0, +\infty)$ that has infinite linear measure such that $N\left(r,\frac{1}{f-a_4}\right)/T(r,f)\geq \lambda$ for all $r \in I.$ Then $f$ and $g$ share all four values CM. Here $N\left(r,\frac{1}{f-a_4}\right)=N(r,f)$ when $a_4=\infty.$
\end{theorem}
\vskip 2mm
\par
\begin{theorem}\rm{}(\cite[Theorem 1]{Gundersen1992})\label{TheoremF}
 Let $f$ and $g$ be two distinct and non-constant meromorphic functions, and let $a_1,$ $a_2,$  $a_3,$  $a_4$ be four
distinct values in the extended complex plane. If $f$ and $g$ share $a_1,$ $a_2,$ $a_3$ IM and
$a_4$ CM, then  $f$  and $g$ are functions of normal growth, $f$
and $g$ have the same order, and the order of $f$  and $g$ is a
positive integer or infinite.
\end{theorem}
\vskip 2mm
\par Next we consider the following special case of Question \ref{Question1.1}:
\vskip 2mm
\par
\begin{question}\rm{}\label{Question1.2}
If two non-constant meromorphic functions share three values IM and share a fourth value CM, where one of the two non-constant meromorphic functions has finite order, then do the two functions necessarily share all four values CM?
\end{question}
\vskip 2mm
\par In this paper, we give an affirmative answer to Question \ref{Question1.2} for two distinct and non-constant meromorphic functions satisfying the assumptions in Question \ref{Question1.2}. Indeed, we prove the following theorem in this paper:
\vskip 2mm
\par
\begin{theorem}\rm{}\label{Theorem1.1}
Suppose that $f$ and $g$ are two distinct and non-constant meromorphic functions, and one of $f$ and $g$ has finite order. If $f$ and $g$ share $a_1,$ $a_2,$ $a_3$ IM and $a_4$ CM, where $a_1,$ $a_2,$ $a_3,$ $a_4$ are four distinct complex values in the extended complex plane, then $f$ and $g$ share $a_1,$ $a_2,$ $a_3$ and $a_4$ CM.
\end{theorem}
\vskip 2mm
\par
\begin{remark}\rm{}\label{Remark1.1}
Theorem \ref{Theorem1.1} gives an affirmative answer to Question \ref{Question1.1} for non-constant meromorphic functions $f$ and $g$ of finite order sharing three values IM and a fourth value CM.
\end{remark}
\vskip 2mm
\par From Theorem \ref{Theorem1.1} we get the following result:
\vskip 2mm
\par
\begin{corollary}\rm{}\label{corollary1.1}
Suppose that $f$ and $g$ are two distinct and non-constant entire functions, and one of $f$ and $g$ has finite order. If $f$ and $g$ share $a_1,$ $a_2,$ $a_3$ IM, where $a_1,$ $a_2,$ $a_3$ are three distinct finite complex values, then $f$ and $g$ share $a_1,$ $a_2$ and $a_3$  CM.
\end{corollary}
\vskip 2mm
\par
\maketitle
\section{Preliminaries}
\vskip 2mm
\par In this section, we will introduce some results that play an important role in proving the main result in this paper.
First of all, we introduce the following result due to Adams-Straus\cite{Adams1971}:
\vskip 2mm
\par
\begin{lemma}\rm{}(\cite[Theorem 3]{Adams1971}). \label{lemma2.1}
\, Let $f$ and $g$ be two non-constant rational functions. If $f$ and $g$ share four distinct
values $a_1,$ $a_2,$ $a_3,$ $a_4$ IM, then $f=g.$
\end{lemma}
\vskip 2mm
\par The following result was proved in Nevanlinna \cite{Nevanlinna1926} originally:
\vskip 2mm
\par
\begin{lemma}\rm{}(\cite[p.373, Satz 3]{Nevanlinna1926} or\cite[Theorem 2]{Gundersen1979}) \label{Lemma2.2}
Let $f$ and $g$ be two distinct and non-constant meromorphic functions, and let $a_1,$ $a_2,$ $a_3$ be three distinct and finite complex values in the complex plane. If $f$ and $g$ share $a_1,$ $a_2,$ $a_3$ and $\infty$ IM, then
\vskip 2mm
\par (i) $T(r,f)=T(r,g)+S(r,f);$ (ii) $2T(r,f)=\overline{N}\left(r,f\right)+\sum\limits_{j=1}^{3} \overline{N}\left(r,\frac{1}{f-a_{j}}\right)+S(r,f).$
\end{lemma}
\vskip 2mm
\par
\begin{remark}\rm{}\label{Remark2.1} A proof of Lemma \ref{Lemma2.2} can be in Gundersen \cite[p.101, proof of Lemma 1]{Gundersen1992}. From \cite[Theorem 2.1]{Hayman1964} and the proof of  \cite[Lemma 1]{Gundersen1992} we see that the error term $S(r,f)$ in Lemma \ref{Lemma2.2} is expressed as
\begin{equation}\label{eq2.1}
S(r,f)=m\left(r,\frac{f'}{f}\right)+m\left(r,\sum\limits_{j=1}^3\frac{f'}{f-a_j}\right)+O(1).
\end{equation}
\end{remark}
\vskip 2mm
\par The following result and its proof can be found in Gundersen \cite{Gundersen1983}:
\vskip 2mm
\par
\begin{lemma}\rm{}(\cite[p.550, Lemma 2]{Gundersen1983} and \cite[p.549, Corollary 1(II)]{Gundersen1983}). \label{Lemma2.3}
Let $f$ and $g$ be two diatinct and non-constant meromorphic functions that share four distinct values $a_1,$ $a_2,$ $a_3$ and $a_4$ IM, where $a_4=\infty.$ Then the following statement holds:
\vskip 2mm
\par (i) $N_0\left(r,\frac{1}{f'}\right)+N_0\left(r,\frac{1}{g'}\right)=S(r,f),$ where $N_0\left(r,\frac{1}{f'}\right)$ and $N_0\left(r,\frac{1}{g'}\right)$ ``count" respectively only those points in $N_0\left(r,\frac{1}{f'}\right)$ and $N_0\left(r,\frac{1}{g'}\right)$ which do not occur when $f(z)=g(z)=a_{j}$ for some $j=1,2,3,4.$
\vskip 2mm
\par (ii) $\sum\limits_{j=1}^{4}N^{\ast}(r,a_{j},f,g)=S(r,f),$ where $N^{\ast}(r,a_{j},f,g)$ with $1\leq j\leq 4$ and $j\in\Bbb{Z}$ denotes the counting function of those common $a_j$-points of $f$ and $g$ in $|z|<r$ that are multiple for both $f$ and $g,$ and each such point in $N^{\ast}(r,a_{j},f,g)$ is counted according to the number of the times of the smaller of the two multiplicities.
\end{lemma}
\vskip 2mm
\par
\begin{remark}\rm{}\label{Remark2.2}
A proof of Lemma \ref{Lemma2.3} can be found in the proof of Lemma 3 in \cite[p.102]{Gundersen1992}. Following the expression of the Mues$'$s function $\Psi(f)$ in \cite[pp.176-177]{FrankHua1999} and the proof of Lemma 3 in \cite[p.102]{Gundersen1992}, we can see that the error term $S(r,f)$ in Lemma \ref{Lemma2.3} is expressed as
\begin{equation}\label{eq2.2}
S(r,f)=m(r,\phi)+O(1)= m\left(r,\sum\limits_{s=1}^3\sum\limits_{t=1}^3\frac{c_{st}f'}{f-a_s}\frac{g'}{g-a_t}\right)+O(1),
      \end{equation}
   where $c_{st}$ with $s,t\in\{1,2,3\}$ is a finite complex constant that depends only on $a_1,$ $a_2$ and $a_3,$  and $\phi$ is called the E.Mues's function that is an entire function defined as
\begin{equation}\nonumber
\phi=\frac{f'g'(f-g)^2}{(f-a_1)(f-a_2)(f-a_3)(g-a_1)(g-a_2)(g-a_3)}
\end{equation}
when $f$ and $g$ share $a_1,$ $a_2,$ $a_3,$ $\infty$ IM. This can be found in Mues \cite[p.171, proof of Lemma 1]{Mues1989}. Here we mention that based upon the assumption that two distinct and non-constant meromorphic functions $f$ and $g$ share $0,$ $1,$ $c$ and $\infty$ IM, where $c$ is a finite complex number such that $c\not\in\{0,1,\infty\},$ the Mues$'$s function $\phi$ in Mues \cite[p.171, proof of Lemma 1]{Mues1989} is written into
\begin{equation}\nonumber
\phi=\frac{f'g'(f-g)^2}{(f)(f-1)(f-c)(g)(g-1)(g-c)}.
\end{equation}
\end{remark}
The following result is due to Yang \cite{Yang1993}:
\vskip 2mm
\par
\begin{lemma}\rm{}(\cite[Theorem 1.6]{Yang1993})\label{lemma2.4}. Suppose that
$f$ is a non-constant meromorphic function of finite order, and let $k$ be a positive integer. Then $m\left(r,\frac{f^{(k)}}{f}\right)=O(\log r), \, \, \text{as}  \, \, r\rightarrow\infty.$
\end{lemma}
\vskip 2mm
\par
\begin{remark}\rm{}\label{Remark2.3} Based upon the assumptions of Lemma \ref{Lemma2.2} or the assumptions of Lemma \ref{Lemma2.3}, we additionally suppose that $f$ has finite order. Then, from Lemma \ref{lemma2.4}, the formula \eqref{eq2.1} in Remark \ref{Remark2.1}, the formula \eqref{eq2.2} in Remark \ref{Remark2.2} we deduce that the error term $S(r,f)$ in Lemma \ref{Lemma2.2} and the error term $S(r,f)$ in Lemma \ref{Lemma2.3} can be estimated as $S(r,f)=O(\log r),$ as $r\rightarrow\infty.$
\end{remark}
\vskip 2mm
\par The following result is from Yang-Yi\cite{YangYi2003}:
\begin{lemma}\rm{}(\cite[p.11, Theorem 1.5]{YangYi2003}). \label{Lemma2.5}
If $f$ is a transcendental meromorphic function in the complex plane, then $\lim\limits_{r\rightarrow\infty}\frac{T(r,f)}{\log r}=\infty.$
\end{lemma}
\vskip 2mm
\par  From Lemma \ref{Lemma2.2}(i), Lemma \ref{Lemma2.3}(ii), Lemma \ref{lemma2.4}, Remark \ref{Remark2.1}-Remark \ref{Remark2.3} we get the following result:
\vskip 2mm
\par
\begin{lemma}\rm{}(\cite[proof of Lemma 9]{LiYi2007}). \label{Lemma2.6} Let $f$ and $g$ be two distinct and non-constant meromorphic
functions,  and one of $f$ and $g$ has finite order. Suppose that $f$ and $g$ share $a_1,$ $a_2,$ $a_3,$ $a_4$ IM,
where $a_1,$ $a_2,$ $a_3$ are three distinct finite values and $a_4=\infty.$ Then
\begin{equation}\nonumber
\frac {1}{7} T(r,f)\leq\sum\limits_{j=1}^{3}\sum\limits_{l=1}^{6}\overline{N}_{(1,l)}\left(r,a_j;f,g\right)
+\sum\limits_{j=1}^{3}\sum\limits_{m=2}^{6}\overline{N}_{(m,1)}\left(r,a_j;f,g\right)
+O(\log r),
\end{equation}
 as $r\rightarrow\infty.$ Here and in what follows, $\overline{N}_{(1,l)}\left (r,a_j;f,g \right )$ denotes the reduced counting function of those common zeros of $f-a_j$ and $g-a_j$ in $|z|<r$ that are simple zeros of $f-a_j$  and are zeros of $g-a_j$ of multiplicity $l$ for $1\leq l\leq 6$ with $l\in\Bbb{Z},$
 and $1\leq j\leq 3$ with $j\in\Bbb{Z},$ while $\overline{N}_{(m,1)}\left(r,a_j;f,g\right)$ denotes the reduced counting function of those common zeros of $f-a_j$ and $g-a_j$ in $|z|<r$ that are zeros of $f-a_j$ of multiplicity $m$ and are simple zeros of $g-a_j$ for $m\geq 2$ and $m\in\Bbb{Z}.$
 \end{lemma}
\vskip 2mm
\par
The following result due to Zhang \cite{Zhang1999} improved Lemma 7 from Li-Yang \cite{LiYang1998}:
\vskip 2mm
\par
\begin{lemma}\rm{}(\cite[Lemma 6]{Zhang1999}). \label{lemma2.7}
Let $f_1$ and $f_2$ be two non-constant meromorphic functions such that
\begin{equation} \nonumber
\overline{N}(r,f_j)+\overline{N}\left(r,\frac{1}{f_j}\right)=S(r) \, \, \text{for} \, \,  j\in\{1,2\}.
\end{equation}
Then, either $\overline{N}_0(r,1;f_1,f_2)=S(r),$ or there exist two integers $s$ and $t$ with $|s|+|t|>0,$ such that $f^s_1f^t_2=1.$ Here and in what follows, $\overline{N}_0(r,1;f_1,f_2)$ denotes the reduced counting function of the common $1$-points in $|z|<r,$ and $S(r)$ is any quantity such that $S(r)=o(T(r)),$ as $r\not\in E$ and $r\rightarrow\infty,$ where $T(r)=T(r,f_1)+T(r,f_2),$ and $E\subset \Bbb{R}^{+}$ denotes a set of finite linear measure.
\end{lemma}
\vskip 2mm
\par
The following result is due to Markushevich \cite{Markushevich1965}:
\begin{lemma}\rm{}(\cite{Markushevich1965}). \label{lemma2.8}
Let $Q(z)=q_nz^n+q_{n-1}z^{n-1}+\cdots+q_1z+q_0,$ where $n$ is a positive integer and $q_n=|q_n|e^{i\theta_n}$ with $|q_n|>0$ and $\theta_n \in [0,2\pi).$ For any given positive number $\varepsilon$ satisfying $0<\varepsilon<\frac{\pi}{4n},$ we consider $2n$ angles:
\begin{equation}\nonumber
S_j: -\frac{\theta_{n}}{n}+(2j-1)\frac{\pi}{2n}+\varepsilon<\theta<-\frac{\theta_{n}}{n}+(2j+1)\frac{\pi}{2n}-\varepsilon,
\end{equation}
where $j$ is an integer satisfying $0\leq j\leq 2n-1.$ Then, there exists a positive number $R=R(\varepsilon)$ such that  $\text{Re}(Q(z))>|q_n|(1-\varepsilon)r^n\sin(n\varepsilon)$ for $z\in S_j$ with $|z|=r>R,$ where $j$ is even, and $\text{Re}(Q(z))<-|q_{n}|(1-\varepsilon)r^n\sin(n\varepsilon)$ for $z\in S_j$ with $|z|=r>R,$ where $j$ is odd.
\end{lemma}
Next we introduce the Nevanlinna theory of meromorphic functions in one angular domain that will play a key role in the proof of the main result: let $f$ be a meromorphic function on the angular domain $\overline{\Omega}(\alpha,\beta)=\{z\in\Bbb{C}:\alpha\leq \arg z\leq \beta\},$ where $0<\beta-\alpha\leq 2\pi.$ Following Goldberg-Ostrovskii \cite[p.25]{Goldberg1970}, we define
\begin{equation}\nonumber
\begin{aligned}
&A_{\alpha,\beta}(r,f)=\frac{\omega}{\pi}\int_{1}^{r}\left(\frac{1}{t^{\omega}}-\frac{t^{\omega}}{r^{2\omega}}\right)\{\log ^{+}|f(te^{i\alpha})|+\log ^{+}|f(te^{i\beta})|\}\frac{dt}{t},\\
&B_{\alpha,\beta}(r,f)=\frac{2\omega}{\pi r^{\omega}}\int_{\alpha}^{\beta}\log ^{+}|f(re^{i\theta})|\sin\omega(\theta-\alpha)d\theta,\\
&C_{\alpha,\beta}(r,f)=2\sum\limits_{1<|b_{m}|<r}\left(\frac{1}{|b_{m}|^{\omega}}-\frac{|b_{m}|^{\omega}}{r^{2\omega}}\right)\sin\omega(\theta_{m}-\alpha),
\end{aligned}
\end{equation}
where $\omega=\frac{\pi}{\beta-\alpha}$ and $\{b_m\} _{m=1}^{+\infty} \subset \overline{\Omega}(\alpha,\beta)$  with $b_m=|b_m|e^{i\theta_{m}}$ and $m\in\Bbb{Z}^{+}$ is the sequence of poles of $f$ on $\overline{\Omega}(\alpha,\beta),$ and each distinct point in the sequence $\{b_m\} _{m=1}^{+\infty}$ is repeated as many times as its multiplicity of a pole of $f$ on $\overline{\Omega}(\alpha,\beta).$ We denote by $\overline{C}_{\alpha,\beta}(r,f)$ the reduced form of $C_{\alpha,\beta}(r,f).$ In other word, if each point in the above expression of $C_{\alpha,\beta}(r,f)$ is counted only once for each distinct point in the sequence $\{b_m\} _{m=1}^{+\infty},$ we denote $\overline{C}_{\alpha,\beta}(r,f)$ instead of $C_{\alpha,\beta}(r,f).$
We call $C_{\alpha,\beta}(r,f)$ the Nevanlinna$'$s angular counting function of poles of $f$ on $\overline{\Omega}(\alpha,\beta).$
The Nevanlinna$'$s angular characteristic function is defined as $S_{\alpha,\beta}(r,f)=A_{\alpha,\beta}(r,f)+B_{\alpha,\beta}(r,f)+C_{\alpha,\beta}(r,f).$ We recall the following two results are due to Zheng \cite{Zheng2003}:
\vskip 2mm
\par
\begin{lemma}\rm{}(\cite[Lemma 3]{Zheng2003}). \label{lemma2.9}
Let $f$ be meromorphic on $\overline{\Omega}(\alpha,\beta).$ Then for arbitrary complex number $a\in\Bbb{C},$ we have
\begin{equation}\nonumber
S_{\alpha,\beta}\left(r,\frac{1}{f-a}\right)=S_{\alpha,\beta}(r,f)+O(1).
\end{equation}
\end{lemma}
\vskip 2mm
\par
\begin{lemma}\rm{}(\cite[Lemma 4]{Zheng2003}). \label{lemma2.10}
Let $f$ be meromorphic on $\overline{\Omega}(\alpha,\beta).$ Then for arbitrary $q$ distinct $a_{1},a_{2},\cdots,a_{q}$ in $\Bbb{C}\cup\{\infty\},$ we have
\begin{equation}\nonumber
(q-2)S_{\alpha,\beta}(r,f)\leq \sum\limits_{j=1}^{q}\overline{C}_{\alpha,\beta}\left(r,\frac{1}{f-a_{j}}\right)+R_{\alpha,\beta}(r,f),
\end{equation}
where the term $\overline{C}_{\alpha,\beta}\left(r,\frac{1}{f-a_{j}}\right)$ will be replaced with $\overline{C}_{\alpha,\beta}\left(r,f\right)$ when $a_{j}=\infty$ for some positive integer $j$ satisfying $1\leq j\leq q.$
\end{lemma}
\vskip 2mm
\par From Zheng \cite[p.81, Lemma 2.5.3]{Zheng2009} and the formula
 \begin{equation}\nonumber
 m_{\alpha,\beta}(r,f)=\frac{1}{2\pi}\int_{\alpha}^{\beta}\log^+|f(re^{\theta})|d\theta \, \, \text{with} \, \, 0\leq \alpha<\beta\leq 2\pi
 \end{equation}
in Zheng \cite[p.56]{Zheng2009}, we get the following result:
\begin{lemma}\rm{}(\cite[Lemma 2.5.3]{Zheng2009}). \label{lemma2.11}
Let $f$ be a meromorphic function in the complex plane. For any positive number $r$ such that $0<r<R,$ we have
\begin{equation}\label{eq2.3}
\begin{aligned}
&\quad R_{\alpha,\beta}(r,f)\\
&\leq K\left(\left(\frac{R}{r}\right)^{\omega}\int_{1}^{R}\frac{\log T(t,f)}{t^{1+\omega}}dt+\log \frac{r}{R-r}+\log \frac{R}{r}+\frac{1}{r^{\omega}}m\left(r,\frac{f'}{f}\right)+1\right),
\end{aligned}
\end{equation}
where $\omega=\frac{\pi}{\beta-\alpha}$ with $0\leq \alpha<\beta\leq 2\pi,$ and $K$ is a constant independent of $r$ and $R.$
\end{lemma}
\vskip 2mm
\par From Remark \ref{Remark2.3}, Lemma \ref{lemma2.4} and Lemma \ref{lemma2.11} with $r\in (0, +\infty)$ and $R=2r,$ we deduce the following result:
\vskip 2mm
\par
\begin{lemma}\rm{}\label{Lemma2.12}
Based upon the assumptions of Lemma \ref{lemma2.11}, we have the following conclusion: suppose that $f$ is a non-constant meromorphic function of finite order in the complex plane. Then, we have $R_{\alpha,\beta}(r,f)=O(1),$ when $r\in (0, +\infty),$ $R=2r$ and $r\rightarrow\infty.$
\end{lemma}
 \begin{proof} Suppose that $f$ is a non-constant meromorphic function of finite order in the complex plane. Then, from Lemma \ref{lemma2.4} we see that there exists a large positive number $t_0$ and a positive constant $M_1$ such that $t_0\geq 1$ and $\log T(t,f)\leq M_1\log t\leq t^{\delta_1}$ when $1\leq t_0\leq t\leq R$ with $R=2r,$ where $\delta_1$ is a positive constant such that $1+\omega-\delta_1>1.$ Combining this with Lemma \ref{lemma2.4}, the formula \eqref{eq2.3} in Lemma \ref{lemma2.11} for $r\in (0, +\infty)$ and $R=2r,$ we have
\begin{equation}\nonumber
\begin{aligned}
&\quad R_{\alpha,\beta}(r,f)\leq K\left(2^{\omega}\int_{1}^{2r}\frac{\log T(t,f)}{t^{1+\omega}}dt+\frac{M_2}{r^{\omega}}\log r +\log 2+1\right)\\
&=K\left(2^{\omega}\int_{1}^{t_0}\frac{\log T(t,f)}{t^{1+\omega}}dt+2^{\omega}\int_{t_0}^{2r}\frac{\log T(t,f)}{t^{1+\omega}}dt+\frac{M_2}{r^{\omega}} \log r+\log 2 +1\right)\\
&\leq  K\left(2^{\omega}\int_{t_0}^{2r}\frac{t^{\delta_1}}{t^{1+\omega}}dt+\frac{M_2}{r^{\omega}} \log r+\log 2 +1+O(1)\right)\\
&= K\left(2^{\omega}\int_{t_0}^{2r}\frac{1}{t^{1+\omega-\delta_1}}dt+\frac{M_2}{r^{\omega}} \log r+\log 2 +1+O(1)\right)=O(1),\\
\end{aligned}
\end{equation}
as $r\rightarrow\infty.$ Here and in what follows, $M_2$ is a positive constant such that $M_2\geq M_1.$ This reveals the conclusion of Lemma \ref{Lemma2.12}.
\end{proof}
\vskip 2mm
\par The following result was originally proved in Goldberg-Ostrovskii\cite{Goldberg1970}:
\vskip 2mm
\par
\begin{lemma}\rm{}(\cite[p.291, Theorem 2.7]{Goldberg1970})\label{Lemma2.13}
Let $f(z)$ be a meromorphic function on the angular domain $\overline{\Omega}(\alpha,\beta)=\{z\in\Bbb{C}:\alpha\leq \arg z\leq \beta\},$ where $0<\beta-\alpha\leq 2\pi.$ If $S_{\alpha,\beta}(r,f)=O(1),$ then there exists a set $F\subset (0, +\infty)$ of finite logarithmic measure, such that when $r\not\in F$ and $r\rightarrow\infty,$ the relation
\begin{equation}\nonumber
\log |f(re^{i\theta})|=cr^{\omega}\sin(\omega(\theta-\alpha))+o(r^{\omega})
\end{equation}
holds uniformly in $\theta$ with $\alpha\leq\theta\leq \beta.$ Here $c$ is a real constant and $\omega=\frac{\pi}{\beta-\alpha}.$
\end{lemma}
\vskip 2mm
\par The following result was proved in Edrei \cite{Edrei1965} and Yang \cite{Yang1979}:
\vskip 2mm
\par
\begin{lemma}\rm{}(\cite[p.85]{Edrei1965}, \cite[Lemma 1]{Yang1979} and \cite[p.5, Theorem 1.1.3]{Zheng2009})\label{Lemma2.14}
 Let $f$ be a transcendental and meromorphic meromorphic function in the complex plane with the lower order $0\leq \mu<\infty$ and
 the order $0<\rho\leq \infty.$ Then, for arbitrary positive number $\sigma$ satisfying $\mu\leq \sigma\leq \rho$ and a set $F\subset (0, +\infty)$ with finite logarithmic measure, there exist a sequence of positive numbers $\{r_n\}$ such that
\vskip 2mm
\par (1) \rm{} $r_n\not\in F$ and $\lim\limits_{n\rightarrow\infty}\frac{r}{n}=\infty;$ \quad (2) \rm{} $\liminf\limits_{n\rightarrow\infty} \frac{\log T(r,f)}{\log r}\geq\sigma;$
\vskip 2mm
\par (3) \rm{}  $T(t,f) <\left(1+ o(l)\right)\left(\frac{t}{r_n}\right)^{\sigma}T(r_n,,f)$ for $t \in \left[\frac{r_n}{n},nr_n\right]$ with $n\in\Bbb{Z}^{+}.$
\end{lemma}
\vskip 2mm
\par A sequence $\{r_n\}$ satisfying (I), (2) and (3) in Lemma \ref{Lemma2.14} is called P\'{o}1ya peaks of order $\sigma$
outside $F\subset (0, +\infty)$ with finite logarithmic measure.  For $r>0$ and $a\in\Bbb{C},$ we define
\begin{equation}\nonumber
D(r,a)=:\left\{\theta\in [-\pi,\pi): \log^{+}\frac{1}{|f(re^{i\theta})-a|}>\frac{T(r,f)}{\log r}\right\}
\end{equation}
and
\begin{equation}\nonumber 
D(r, \infty)=:\left\{\theta\in [-\pi,\pi): \log^{+}|f(re^{i\theta})|>\frac{T(r,f)}{\log r}\right\}.
\end{equation}
\vskip 2mm
\par The following result is a special version of the main result of Baernstein \cite{BaernsteinII1973}:
\vskip 2mm
\par
\begin{lemma} \rm{}(\cite[p.419]{BaernsteinII1973})\label{Lemma2.15}
 Let $f$ be a transcendental and meromorphic function in the complex plane with the finite lower order $\mu$
and the order $0<\rho<\infty,$ and for some point $a\in\Bbb{C}\cup \{\infty\},$ we have $\delta=:\delta(a,f)> 0.$ Then, for arbitrary P\'{o}1ya
peaks $\{r_n\}\subset (0, +\infty)\setminus F$ of order $\sigma > 0$ satisfying $\mu\leq\sigma\leq \rho,$ we have
\begin{equation}\nonumber
\liminf\limits_{r_n\rightarrow\infty}\text{mes}D(r_n, a)\geq \min\left\{2\pi,\frac{4}{\sigma}\arcsin\sqrt{\frac{\delta}{2}}\right\}.
\end{equation}
Here and in what follows, $F\subset (0, +\infty)$ is a set of finite logarithmic measure.
\end{lemma}
\vskip 2mm
\par Next we introduce the Nevanlinna theory in a simply connected domain on $\Bbb{C}$ (cf.\cite[pp.26-38]{Zheng2009}): let $D\subset\Bbb{C}$ be a simply connected domain surrounded by finitely many piecewise analytic curves. Then for any $a\in D,$ there exists a Green function, denoted as $G_D(z,a),$ for $D$ with singularity at $a\in D$ which is uniquely determined by the following conditions: (1) $G_{D}(z,a)$ is harmonic in $D\backslash\{a\};$ (2) in a neighborhood of $a,$ $G_{D}(z,a)=\log\frac{1}{|z-a|}+\omega_D(z,a)$ for some function $\omega_D(z,a)$ harmonic in $D;$
(3) $G_{D}(z,a)\equiv0$ on the boundary of $D.$
\vskip 2mm
\par Next we denote by $\Gamma=\partial{D}$ the positive boundary of $D$ and $\mathfrak{n}$ the inner normal of $\Gamma$ with respect to $D.$ Since for $z\in D,$ $G_{D}(z,a)>0$ and for $z\in \Gamma,$ $G_{D}(z,a)=0,$ from the definition of directional derivative it follows that the directional derivative of $G_{D}(z,a)$ on $\Gamma$ in the inner normal is non-negative, that is to say, $\frac{\partial G}{\partial {\mathfrak{n}}}\geq 0$ with $ G=G_{D}(z,a).$ From the Green formula, in view of the Green function, we can establish the following formula, which is an extension of the Poisson formula for a disk. For a generalization of the formula, the reader is referred to  \cite[p.4, Theorem 1.1]{Goldberg1970}. Following Zheng\cite[p.28]{Zheng2009}, we define the counting function of $f$ with the center at $a$ for $D$ as follows:
\begin{equation}\label{eq2.4}
N(D,a,f)=\sum\limits_{b_{n}\in D}G_{D}(b_{n},a)+n(0,a,f)\omega_{D}(a,a),
\end{equation}
where $a$ is a point such that $a\in D,$ and $b_{n}$ with $n\in\Bbb{Z}^{+}$ is a pole of $f$ appearing often according to its multiplicities of a pole of $f(z),$ and $n(0,a,f)$ is the multiplicity of a pole of $f(z)$ at $a,$ while $\overline{N}(D,a,f)$ is the sum in \eqref{eq2.4} counting all distinct $b_{n}\in D$ with $n\in\Bbb{Z}^{+}$ and with $n(0,a,f)$ replaced by $1$ when $f(a)=\infty.$  Following Zheng \cite[p.28]{Zheng2009}, we also define
\begin{equation}\label{eq2.5}
m(D,a,f)=\frac{1}{2\pi}\int_{\Gamma}\log^{+}|f(\zeta)|\frac{\partial G_{D}(\zeta,a)}{\partial \mathfrak{n}}ds
\end{equation}
and
\begin{equation}\label{eq2.6}
T(D,a,f)=m(D,a,f)+N(D,a,f),
\end{equation}
 which are called the proximity function and the Nevanlinna characteristic function of $f$ with the center at $a$ for the simply connected domain $D$ respectively. For $p$ meromorphic functions $f_1,$ $f_2,$ $\ldots,$ $f_p$ on a simply connected domain $D$ surrounded by finitely many piecewise analytic curves, where $p$ is a positive integer such that $p\geq 2,$  the basic properties concerning the proximity functions, the Nevanlinna counting functions and the Nevanlinna characteristic functions of meromorphic functions on $\overline {D}=:D\cup \partial{D}$ with the center at $a$ for the simply connected domain $D$ are similar to the corresponding properties concerning the classical proximity functions, the classical Nevanlinna counting functions and the classical Nevanlinna characteristic functions of meromorphic functions on the closed disc $|z|\leq r$ with the center at the origin point $0$ for the closed disk $|z|\leq r,$ this can be found, for example, in Zheng \cite[pp.28-29]{Zheng2009}. We recall the following result that is called the Nevanlinna second fundamental theorem with the center at $a$ for a simply connected domain $D$ (cf.\cite[p.32]{Zheng2009}):
\vskip 2mm
\par
\begin{lemma}\rm{}(cf.\cite[p.32, Theorem 2.1.4]{Zheng2009})\label{Lemma2.16}
Let $f$ be a meromorphic function on $D\cup \partial D,$ where $D$ is a simply connected domain such that $D\subset\Bbb{C},$ and let $a_1,$ $a_2,$ $\ldots,$ $a_q$ be $q$ distinct finite complex numbers, where $q$ is a positive integer such that $q\geq 2,$ Then for $a\in D$ such that $f(a)\not\in\{ 0,\infty\}$ and  $f(a)\neq a_j$ with $1\leq j\leq q $ and $j\in\Bbb{Z},$ we have
\begin{equation}\nonumber 
\begin{aligned}
(q-1)T(D,a,f)&\leq N(D,a,f)+\sum\limits_{j=1}^{q}N\left(D,a,\frac{1}{f-a_{j}}\right)-N_{1}(D,a,f)\\
&\quad+S(D,a,f),
\end{aligned}
\end{equation}
where
\begin{equation}\nonumber
\begin{aligned}
S(D,a,f)&=m\left(D,a,\frac{f'}{f}\right)+\sum\limits_{j=1}^{q}m\left(D,a,\frac{f'}{f-a_{j}}\right) \\
&\quad +q\left(\log^{+}\frac{2q}{\delta}+\log^{+}\frac{\delta}{2q}+\log2\right)+\log q-\log|f'(a)| \\
&\quad+\sum\limits_{j=1}^{q}\left(\log|f(a)-a_{j}|+\varepsilon(a_j,D)\right)
\end{aligned}
\end{equation}
and
\begin{equation}\nonumber
N_{1}(D,a,f)=2N(D,a,f)-N(D,a,f')+N\left(D,a,\frac{1}{f'}\right).
\end{equation}
Here $\delta=\min\limits_{1\leq j< k \leq q}|a_j-a_k|$ and  $\varepsilon(a_j,D)\leq \log^{+}|a_j|+\log 2$ for $1\leq j\leq q $ with $j\in\Bbb{Z}.$
\end{lemma}
\vskip 2mm
\par We recall the following result that depicts the relationship between two Green's functions for two domains of which is mapped to another one by a univalent analytic function:
 \begin{lemma}\rm{}(cf.\cite[p.276, 5.3 Theorem]{Conway1978})\label{Lemma2.17} Let $D_1$ and $D_2$ be regions on the complex plane such that there is a one-to-one analytic function $\phi$ of $D_1$ onto $D_2,$ and let $a\in D_1$ and $\alpha=\phi(a).$ If $G_a$ and $G_{\alpha}$ are the Green's functions for $D_1$ and $D_2$  with singularities $a$ and $\alpha$ respectively, then $G_a(z)=G_{\alpha}(\phi(z))$ for each $z\in D_1.$
\end{lemma}
\vskip 2mm
\par
\begin{remark}\rm{}\label{Remark2.5X2} Following Zheng \cite[pp.35-37]{Zheng2009} and Goldberg-Ostrovskii\cite[p.2]{Goldberg1970}, we have a discussion as follows: let $f$ be a non-constant meromorphic function on $\overline{D}=:D\cup\partial {D},$ where $D\subset\Bbb{C}$ is a simple connected domain that is surrounded by finitely many piecewise analytic curves, in other words, the boundary $\Gamma =\partial{D}$ of the simple connected domain $D$ consists of finitely many piecewise analytic curves. Then, for a point $a\in D$ such that $f(a)\neq 0, \infty,$ by the Riemann mapping theorem (cf.\cite[p.230,Theorem 1]{Ahlfors1979})
there exists a conformal mapping $\phi_a(z)$ of $D$ onto the unit disc $\{w\in\Bbb{C}:|w|<1\}$ in such a way that $\phi_a(a)=0.$ By the theorem on the boundary behavior of conformal mappings, the function $\phi_a(z)$ is continuous on $\overline{D}=D\cup\partial{D},$ and $|\phi_a(\zeta)|=1$ for $\zeta\in \Gamma.$ It is easy for us to see that $G_D(z,a)=-\log |\phi_a(z)|$ with $z\in \overline{D}=D\cup\partial{D}$ satisfies (1) and (2).  Along the positive boundary $\Gamma=\partial{D}$ of $D,$ we have
\begin{equation}\label{eq2.7}
\frac{\phi'_a(\zeta)}{\phi_a(\zeta)}d\zeta=i\frac{\partial G_D(\zeta,a)}{\partial \mathfrak{n}}ds \, \, \text{for each} \, \, \zeta\in \Gamma.
\end{equation}
when $\phi_a(z)$ is analytic in the boundary $\partial{D}$ of $D.$ Indeed, from the Cauchy-Riemann condition that the tangent vector $\mathfrak{s}$ at any point along the positive boundary $\Gamma=\partial{D}$ of $D$ becomes the inner normal $\mathfrak{n}$ of $\Gamma$ at the corresponding point on $\Gamma$ after $\mathfrak{s}$ is rotated $\pi/2$ anticlockwise. Therefore, for each $\zeta\in \Gamma$ we have
\begin{equation}\label{eq2.8}
\frac{\partial u}{\partial{\mathfrak{s}} }=\frac{\partial v}{\partial{\mathfrak{n}} } \, \, \text{and} \, \, \frac{\partial u}{\partial{\mathfrak{n}} }=-\frac{\partial v}{\partial{\mathfrak{s}} }
\end{equation}
when $u(z)+iv(z)=:H(z)$ with $u(z)=\text{Re}(H(z))$ and $v(z)=\text{Im}(H(z))$ is an analytic function on $\overline{D}=D+\partial D.$ Therefore, it follows from \eqref{eq2.8} and $\log \phi_a(z)=\log |\phi_a(z)|+i\arg \phi_a(z)$ with $z\in D\setminus\{a\}$ that
\begin{equation}\label{eq2.9}
\begin{split}
d\log \phi_a(\zeta)&=\frac{\partial{\log \phi_a(\zeta)}}{\partial{\mathfrak{s}}}ds=i\frac{\partial{\arg \phi_a(\zeta)}}{\partial{\mathfrak{s}}}ds=-i\frac{\partial\log|\phi_a(\zeta)|}{\partial{\mathfrak{n}}}ds\\
&=i\frac{\partial{G_D(\zeta,a)}}{\partial{\mathfrak{n}}}ds
\end{split}
\end{equation}
for $\zeta\in \Gamma$ when the function $\phi_a(z)$ mentioned above is analytic in the boundary $\partial{D}$ of $D.$ Here and in what follows, $d\mathfrak{s}$
denotes the tangent vector along the positive boundary $\Gamma$ of $D$ at any given point $\zeta\in \Gamma,$ and $\mathfrak{n}$ denotes the corresponding inner normal of the positive boundary $\Gamma$ of $D$ at the corresponding point $\zeta\in \Gamma.$ From \eqref{eq2.9} we get \eqref{eq2.7}. In this way, we can obtain the Green functions for some special simple connected domains on $\Bbb{C}$ as follows:
\vskip 2mm
\par ({\bf{\ref{Remark2.5X2}.1}}) For $D=\{z\in\Bbb{C}:|z|<R, \, \, \text{Im}(z)>0\}$ and a point $a\in D,$ where $R$ is some positive number, we follow Zheng \cite[p.36]{Zheng2009} and Goldberg-Ostrovskii\cite[p.2, Example 2.]{Goldberg1970} to have
\begin{equation}\label{eq2.10}
\phi_a(z)=\frac{R(z-a)}{R^2-\overline{a}z}\cdot\frac{R^2-az}{R(z-\overline{a})}
\end{equation}
and
\begin{equation}\label{eq2.11}
 G_D(z,a)=-\log |\phi_a(z)|=\log \left| \frac{R^2-\overline{a}z}{R(z-a)}\cdot\frac{R(z-\overline{a})}{R^2-az}\right|
\end{equation}
for $z\in D\cup\partial D.$ Moreover, from \eqref{eq2.7}, \eqref{eq2.10}, \eqref{eq2.11} and Goldberg-Ostrovskii\cite[p.306]{Goldberg1970} we have for each point $\zeta=Re^{i\theta}$ with $0\leq \theta\leq \pi$ on the semi-circle $\{\zeta\in\Bbb{C}: \zeta=Re^{i\theta} \, \, \text{with} \, \, 0<\theta<\pi\}$ and $a=|a|e^{\varphi}$ with $0<|a|<R$ and $0<\varphi< \pi$ that
\begin{equation}\label{eq2.12}
\begin{aligned}
&\quad \frac{G_D(\zeta,a)}{\partial{\mathfrak{n}}}ds=-i\left(\log \phi_a(\zeta)\right)'ds=-i\left(\log\frac{R^2-\overline{a}\zeta}{R(\zeta-a)}
 \cdot\frac{R(\zeta-\overline{a})}{R^2-a\zeta} \right)'ds\\
 &=\left(\frac{R^2-|a|^2}{|\zeta-a|^2}-\frac{R^2-|a|^2}{|\zeta-\overline{a}|^2}\right)d\theta\\
  \end{aligned}
\end{equation}
\begin{equation*}
\begin{aligned}
 &=\frac{4(R^2-|a|^2)R|a|\sin\varphi \sin\theta}{(R^2+|a|^2-2R|a|\cos(\varphi-\theta))((R^2+|a|^2-2R|a|\cos(\varphi+\theta))}d\theta,
 \end{aligned}
 \end{equation*}
while for each point $\zeta=t$ in the interval $\{\zeta\in\Bbb{R}: -R<t<R\}$ and $a=|a|e^{\varphi}$ with $0<|a|<R$ and $0<\varphi< \pi,$ we have
\begin{equation}\nonumber
\frac{G_D(\zeta,a)}{\partial{\mathfrak{n}}}ds=2\left(\frac{|a|\sin\varphi}{|a-t|^2}-\frac{R^2|a|\sin\varphi}{|R^2-at|^2}\right).
  \end{equation}
\vskip 2mm
\par  ({\bf{\ref{Remark2.5X2}.2}}) For $D=\{z\in\Bbb{C}:|z|<R\}$ and the point $a\in D,$ where $R$ is any given positive number, we follow Zheng \cite[p.36]{Zheng2009} to write the proximity function $m(D,a,f),$ the Nevanlinna counting function $N(D,a,f),$ the Nevanlinna characteristic function $T(D,a,f)$ of $f$ with the center at $a$ for $D=\{z\in\Bbb{C}:|z|<R\}$ into $m(R,a,f),$  $N(R,a,f)$ and $T(R,a,f)$ respectively. In particular, the proximity function $m(R,0,f),$ the Nevanlinna counting function $N(R,0,f),$ the Nevanlinna characteristic function $T(R,0,f)$ of $f$ with the center at $0$ for $D=\{z\in\Bbb{C}:|z|<R\}$ are reduced into $m(R, f),$ $N(R,f)$ and $T(R,f)$ respectively. Indeed, we substitute \eqref{eq2.10} into \eqref{eq2.7}, and then we have
\begin{equation}\label{eq2.13}
\frac{\partial{G_D(\zeta,z)}}{\partial{\mathfrak{n}}}ds=\frac{R^2-|z|^2}{|\zeta-z|^2}d\theta \, \, \text {and} \, \, \omega_D(z,a)=\log \left|\frac{R^2-\overline{a}z}{R}\right|
\end{equation}
for any complex number $\zeta=Re^{i\theta}$ with $\theta\in [0,2\pi)$ and any given complex number $z$ such that $|z|<R$ and a point $a$ such that $|a|<R.$ In particular, we have from \eqref{eq2.13} that
\begin{equation}\label{eq2.14}
\frac{\partial{G_D(\zeta,0)}}{\partial{\mathfrak{n}}}ds=d\theta \, \, \text {and} \, \, \omega_D(z,0)=\log R
\end{equation}
for any complex number $\zeta=Re^{i\theta}$ with $\theta\in [0,2\pi).$ Therefore, it follows from \eqref{eq2.4}-\eqref{eq2.6} and \eqref{eq2.14} that the proximity function $m(D,0,f),$ the Nevanlinna counting function $N(D,0,f),$ the Nevanlinna characteristic function $T(D,0,f)$ of $f$ with the center at $0$ for $D=\{z\in\Bbb{C}:|z|<R\}$ are reduced into
\begin{equation}\label{eq2.15}
m(R,f)=\frac{1}{2\pi}\int_0^{2\pi}\log^{+}|f(Re^{i\theta})|d\theta,
\end{equation}
\begin{equation}\label{eq2.16}
N(R, f)=\sum\limits_{0<|b_{n}|<R}\log \frac{R}{|b_n|}+n(0,f)\log R
\end{equation}
and
\begin{equation}\label{eq2.17}
T(R,f)=m(R,f)+N(R,f)
\end{equation}
respectively.
\vskip 2mm
\par  ({\bf{\ref{Remark2.5X2}.3}})Based upon the assumptions of Lemma \ref{Lemma2.16}, we follow Zheng \cite[p.35]{Zheng2009} to see that $N(D,a,f)$ and $T(D,a,f)$ are non-negative for the case of $f(a)\neq \infty,$ while for the case of  $f(a)=\infty,$ $T(D,a,f)$ may be negative. For example, we consider the function $f(z)=\frac{1}{2z^p}$ with $D=\left\{z\in\Bbb{C}:|z|<\frac{1}{2}\right\},$ where $p$ is a positive integer. In view of \eqref{eq2.4}, \eqref{eq2.15}-\eqref{eq2.17} and Theorem 2.16 in Zheng \cite[p.34]{Zheng2009} we deduce $T(D,0,f)=T\left(\frac{1}{2},f\right)=-\log 2$ and  $N\left(D,0,\frac{1}{f-e^{i\theta}}\right)=N\left(\frac{1}{2},\frac{1}{f-e^{i\theta}}\right)=0$ with $\theta\in[0,2\pi]$ for $D=\{z\in\Bbb{C}:|z|<\frac{1}{2}\}.$ It is obvious that $N(D,a,f)\geq 0$ and so $T(D,a,f)\geq 0,$ when $\omega_D(a,a)\geq 0$ for $a\in D=\{z\in\Bbb{C}:|z|<\frac{1}{2}\}.$
\end{remark}
\vskip 2mm
\par For convenience in stating the following result, we shall use the following notation. We
shall let $(f,H)$ denote a pair that consists of a transcendental meromorphic function $f$ and a finite set $H=\{(k_1,j_1), (k_2,j_2),\ldots,(k_q,j_q)\}$ of distinct pairs of integers that satisfy $k_l>j_l$ for $l=1,2\ldots,q.$ We recall the following result due to Gundersen \cite{Gundersen1988}:
\vskip 2mm
\par
\begin{lemma}\rm{}(\cite[p.90, Corollary 2]{Gundersen1988})\label{Lemma2.18}
Let $(f,H)$ be a given pair where $f$ has finite order $\rho,$ and let $\varepsilon$ be a given constant. Then, there exists a set $F\subset (1, +\infty)$ that has finite logarithmic measure, such that if $|z|\not\in F\cup [0,1]$ and for all $(k,j)\in H,$ we have
\begin{equation}\nonumber
\left|\frac{f^{(k)}(z)}{f^{(j)}(z)}\right|\leq |z|^{(k-j)(\rho-1+\varepsilon)}.
\end{equation}
\end{lemma}
\vskip 2mm
\par Let $f$ be a transcendental and meromorphic function in the complex plane. An important
 role in complex dynamics \cite{Bergweiler1993} is played by the singular values of the inverse
function $f^{-1}:$ these are the critical values of $f$ and the asymptotic values, that
is, values $a\in\Bbb{C}$ such that $f(z)$ tends to $a$ as $z$ tends to infinity along a path $\gamma_a.$ We denote by $B$ the class of all transcendental and meromorphic functions $f$ in the complex plane for which the set of finite singular values of $f^{-1}$ is bounded, and by $S$ we denote the
subclass of $B$ consisting of those $f$ for which $f^{-1}$ has finitely many singular values.
Both classes $S$ and $B$ have been studied extensively in iteration theory and value distribution theory
 (cf.\cite{Bergweiler1993, Bergweiler1995, Eremenko1992, Langley1998, Rippon1999}).  We recall the following result from Bergweiler \cite{Bergweiler1993} and  Rippon \cite{Rippon1999} that was proved first in Eremenko \cite{Eremenko1992} for entire functions:
\vskip 2mm
\par
\begin{lemma}\rm{}(\cite{Bergweiler1993, Rippon1999})\label{Lemma2.19}
Let $f$ be a transcendental and meromorphic function in the complex plane plane such that the set of finite singular values of the inverse function $f^{-1}$  is bounded. Then there exist $L > 0$ and $M > 0$ such that if $|z| > L$ and $|f(z)| > M$ then
\begin{equation}\nonumber
\left|z\frac{f'(z)}{f(z)}\right|\geq \frac{\log \left|f(z)/M\right|}{C},
\end{equation}
where  $C$ is a positive absolute constant, in particular independent of $f,$ $L$ and $M.$
\end{lemma}
\vskip 2mm
\par The following result is due to Bergweiler-Eremenko \cite{Bergweiler1995D2}:
\vskip 2mm
\par
\begin{lemma}\rm{}(\cite[p.90, Corollary 3]{Bergweiler1995D2})\label{Lemma2.20}
If a meromorphic function of finite order $\rho$ has only
finitely many critical values, then it has at most $2\rho$ asymptotic values.
\end{lemma}
\vskip 2mm
\par
Next we introduce the notion of the spherical derivative of a meromorphic function in the complex plane (cf.\cite{Hayman1964,Yang1993}): let $f$ be a non-constant meromorphic function. The spherical derivative of $f$ at $z\in \Bbb{C}$ is given as $f^{\#}(z)=\frac{|f'(z)|}{1+|f(z)|^2}.$ We recall the following result from Chang-Zalcman \cite{chang2008}:
\vskip 2mm
\par
\begin{lemma}\rm{}(\cite[Lemma 1]{chang2008})\label{Lemma2.21}. Let $f$ be a meromorphic function in
$\Bbb{C}.$ If $f$ has bounded spherical derivative in $\Bbb{C},$ then $f$ is of order at most $2.$ If, in addition, $f$ is an entire function, then
the order of $f$ is at most $1.$
\end{lemma}
\vskip 2mm
\par
\begin{remark}\rm{}\label{Remark2.5} Following He-Xiao \cite[pp.53-55]{HeXiao1988}, we introduce the Ahlfors-Shimizu's characteristic function of a meromorphic function: suppose that $f$ is a meromorphic function in the complex plane. Then, the Ahlfors-Shimizu$'$s characteristic function of $f,$ denoted as $\mathring{T}(r,f),$ is defined as
\begin{equation}\nonumber
\mathring{T}(r,f)=\int_0^r\frac{1}{t}\left(\frac{1}{\pi}\iint \limits_{|z|\leq t}(f^{\#}(z))^2dxdy\right)dt,
\end{equation}
where $z=x+yi$ with $x,y\in\Bbb{R}.$ The Nevanlinna's characteristic function $T (r,f)$ and the Ahlfors-Shimizu's characteristic function $\mathring{T}(r, f )$ differ by a bounded quantity that is independent of $r\in (0, +\infty)$ (cf.\cite[p.56]{HeXiao1988}). This implies the first part of Lemma \ref{Lemma2.21}. The result of Lemma \ref{Lemma2.21} for entire functions is much subtler, which is a special case of
Clunie-Hayman \cite[Theorem 3]{Clunie1965}.
\end{remark}
\vskip 2mm
\par We also need some notions of normal families of meromorphic functions in a domain of the complex plane.
which can be found, for example in  Montel \cite{Montel1927} and Hayman \cite[pp.157-160]{Hayman1964}. For the convenience of the readers, we give the detail of the relevant notions as follows: following  Montel \cite{Montel1927}, we call a class $\mathcal{F}$ of meromorphic functions in a domain $D\subseteq\Bbb{C}$ to be normal in the domain $D,$ provided that for any given sequence $\{f_n(z)\}$ of meromorphic functions in  $\mathcal{F},$ we can find a subsequence $\{f_{n_p}(z)\}\subseteq\{f_n(z)\}$ which converges everywhere in $D$ and uniformly on compact subsets of $D$ with respect to the chordal metric on Riemann sphere. We then say that $\{f_{n_p}(z)\}$  converges locally uniformly in $D.$ An equivalent statement is that for every $z_0$ in $D$ there exists a neighbourhood $|z-z_0|<\delta$ in which $\{f_{n_p}(z)\}$ or  $\left\{\frac{1}{f_{n_p}(z)}\right\}$ converges uniformly as $p\rightarrow\infty.$ This implies that if $w_1$ and $w_2$ are two points in the $w$-plane, their distance in the chordal metric of Riemann sphere is
 \begin{equation}\nonumber
k(w_1, w_2)=\frac{|w_1-w_2|}{\sqrt{(1+|w_1|^2)(1+|w_2|^2)}}\leq |w_1-w_2|
\end{equation}
Also
 \begin{equation}\nonumber
k(w_1, w_2)=\frac{|\frac{1}{w_1}-\frac{1}{w_2}|}{\sqrt{(1+|\frac{1}{w_1}|^2)(1+|\frac{1}{w_2}|^2)}}\leq \left|\frac{1}{w_1}-\frac{1}{w_2}\right|.
\end{equation}
Thus, if either $f_n(z)\rightarrow f(z)$ or $\frac{1}{f_n(z)}\rightarrow \frac{1}{f(z)}$ uniformly in a set $E\subseteq\Bbb{C}\cup\{\infty\},$ then
$k(f_n(z), f(z))\rightarrow 0,$ uniformly in $E.$ Thus, if one of these two conditions holds uniformly in some neighbourhood of every point of $D,$
then $k(f_n(z), f(z))\rightarrow 0$ uniformly in some neighbourhood of every point of $D,$  and hence by the Heine-Borel theorem uniformly on every compact subset of $D;$ see, for example, Hayman \cite[pp.157-160]{Hayman1964}. We recall the following result from Steinmetz \cite{Steinmetz2012}:
\vskip 2mm
\par
\begin{lemma}\rm{}(\cite[Rescaling Lemma 7.1.]{Steinmetz2012})\label{Lemma2.22}.
Let $f$ and $g$ be two non-constant meromorphic functions in the extended complex plane, and let $a_1,$ $a_2,$ $a_3,$ $a_4$ be four distinct complex values in the extended complex plane such that $f$ and $g$ share $a_1,$ $a_2,$ $a_3,$ $a_4$ IM. Then, either the spherical derivatives $f^{\#}(z)$
and $g^{\#}(z)$ are uniformly bounded on $\Bbb{C},$ or else there exist
sequences $\{z_m\}_{m=1}^{+\infty}\subset\Bbb{C}$ and $\{\rho_m\}_{m=1}^{+\infty}\subset(0, +\infty)$ with $\rho_m\rightarrow 0,$ as $m\rightarrow+\infty,$ such that the sequences $\{f_m(z)\}_{m=1}^{+\infty}$ and $\{g_m(z)\}_{m=1}^{+\infty}$ with $f_m(z)$ and $g_m(z)$ being defined as $f_m(z) = f(z_m+\rho_mz)$ and $g_m(z) = g(z_m+\rho_mz)$ respectively for each $m\in\Bbb{Z}^{+},$  simultaneously tend to non-constant meromorphic functions
$\hat{f}(z)$  and $\hat{g}(z)$ in the complex plane respectively, such that $\hat{f}(z)$  and $\hat{g}(z)$  have uniformly bounded spherical derivatives on $\Bbb{C},$ and such that  $\hat{f}(z)$  and $\hat{g}(z)$ share the four complex values $a_1,$ $a_2,$ $a_3,$ $a_4$ IM.
\end{lemma}
\vskip 2mm
\par
\begin{remark}\rm{}\label{Remark2.6} From Lemma \ref{Lemma2.21} and Lemma \ref{Lemma2.22} we deduce that the orders of the non-constant meromorphic functions  $\hat{f}(z)$  and $\hat{g}(z)$ in Lemma \ref{Lemma2.22} satisfy $\rho(\hat{f})\leq 2$ and $\rho(\hat{g})\leq 2.$
\end{remark}
\vskip 2mm
\par
\maketitle
\section{Proof of Theorem \ref{Theorem1.1}}
First, before the formal beginning of the proof of Theorem \ref{Theorem1.1}, we want to point out that the lines from the next formal beginning of the proof of Theorem \ref{Theorem1.1} to the end of Case 1 of the next formal beginning of the proof of Theorem \ref{Theorem1.1} are covered essentially by the proof of  Steinmetz \cite[Theorem 6.2]{Steinmetz2012} even without any assumption on the order of growth. For the convenience of the readers, we give the detailed proof for this part on the assumption of finite order of growth as follows: first of all, we suppose,  without loss of generality, that $f$ has finite order $\rho(f)=:\rho<\infty.$ Combining this with Lemma \ref{lemma2.1}, Lemma \ref{Lemma2.2} (i) and the assumption that $f$ and $g$ are two distinct and non-constant meromorphic functions that share $a_1,$ $a_2,$ $a_3$ IM and $a_4$ CM, we deduce that $f$ and $g$ are two distinct transcendental meromorphic functions in the complex plane. From Lemma \ref{Lemma2.2}(i), Remark \ref{Remark2.3} and the assumption $\rho(f)<\infty,$ we have
\begin{equation}\label{eq3.1}
T(r,f)=T(r,g)+O(\log r), \, \, \text{as} \, \,  r\rightarrow\infty.
\end{equation}
From \eqref{eq3.1}, Theorem \ref{TheoremF} and Definition \ref{definition1.1}, we deduce that the orders of $f$ and $g$ are positive integers such that
\begin{equation}\label{eq3.2}
1\leq \rho(f)=\rho(g)=:\rho<\infty.
\end{equation}
Without loss of generality, we assume that $a_1=0,$ $a_2=1,$ $a_3=c$ and $a_4=\infty,$ and set
\begin{equation}\label{eq3.3}
\varphi:=\frac{f''}{f'}-\frac{f'}{f}-\frac{f'}{f-1}-\frac{f'}{f-c}-\frac{g''}{g'}+\frac{g'}{g}+\frac{g'}{g-1}+\frac{g'}{g-c}.
\end{equation}
By \eqref{eq3.2} and  Whittaker \cite[p. 82]{Whittaker1936} we have
\begin{equation}\label{eq3.4}
\rho(f)=\rho(f')=\rho(g)=\rho(g')<\infty.
\end{equation}
By \eqref{eq3.3}, \eqref{eq3.4} and Lemma \ref{lemma2.4} we have
\begin{equation}\label{eq3.5}
m(r,\varphi)=O(\log r), \, \, \text{as} \, \, r\rightarrow\infty.
\end{equation}
Since $f$ and $g$ share $a_{1},$ $a_{2},$ $a_{3}$ IM and $a_{4}$ CM, by simple calculating we can see that $\phi$ defined as \eqref{eq3.3} is analytic at any point $z\in\Bbb{C}$ such that $f(z)=g(z)=a_j$ for some $j=1,2,3,4,$  and so from Lemma \ref{Lemma2.3}(i) and Remark \ref{Remark2.3} we have
\begin{equation}\label{eq3.6}
N(r,\varphi)\leq N_0\left(r,\frac{1}{f'}\right)+N_0\left(r,\frac{1}{g'}\right)=O(\log r).
\end{equation}
From \eqref{eq3.5} and \eqref{eq3.6} we have
\begin{equation}\label{eq3.7}
T(r,\varphi)=m(r,\varphi)+N(r,\varphi)=O(\log r).
\end{equation}
From \eqref{eq3.7} and Lemma \ref{Lemma2.5} we see that $\phi$ is a rational function. Combining this with \eqref{eq3.3}, we deduce that \eqref{eq3.3} can be rewritten into
\begin{equation}\label{eq3.8}
\varphi=P_1+\sum\limits_{j=1}^{q_1} \frac{m_j}{z-z_j},
\end{equation}
where $P_{1}$ is reduced to zero or a non-vanishing polynomial, $q_1$ is a non-negative integer such that $\sum\limits_{j=1}^{q_1} \frac{m_j}{z-z_j}$ in the right hand side of \eqref{eq3.8} is reduced to zero when $q_1=0,$ and $m_j$ with $1\leq j\leq q_1$ and $j\in\Bbb{Z}$ is some non-zero integer when $q_1\geq 1$ and $q_1\in\Bbb{Z},$ while $z_1,$  $z_2,$ $\ldots,$ $z_{q_1}$  are those distinct points in $\Bbb{C}$ such that $f'(z_j)=0$ and $g'(z_j)(f(z_j)-a_l)(g(z_j)-a_l)\neq 0$ with $1\leq j\leq q_1$ and $j\in\Bbb{Z}$ for some $l\in\{1,2,3,4\},$ or $g'(z_j)=0$ and $f'(z_j)(f(z_j)-a_l)(g(z_j)-a_l)\neq 0$ with $1\leq j\leq q_1$ and $j\in\Bbb{Z}$ for some $l\in\{1,2,3,4\},$ or $z_j$ with $1\leq j\leq q_1$ and $j\in\Bbb{Z}$ is such a common zero of $f'$ and $g'$ in $\Bbb{C}$ with different multiplicities for $f'$ and $g',$ but $(f(z_j)-a_l)(g(z_j)-a_l)\neq 0$ with $1\leq j\leq q_1$ and $j\in\Bbb{Z}$ for all $l\in\{1,2,3,4\}$  when $q_1\geq 1$ and
$q_1\in\Bbb{Z}.$ By \eqref{eq3.3} and \eqref{eq3.8} we have
\begin{equation}\label{eq3.9}
\frac{f''}{f'}-\frac{f'}{f}-\frac{f'}{f-1}-\frac{f'}{f-c}-\frac{g''}{g'}+\frac{g'}{g}+\frac{g'}{g-1}+\frac{g'}{g-c}=P_1+\sum\limits_{j=1}^{q_1} \frac{m_j}{z-z_j}.
\end{equation}
We integrate both sides of \eqref{eq3.9}, and then we have
\begin{equation}\label{eq3.10}
\frac {f'g(g-1)(g-c)} {g'f(f-1)(f-c)}=Re^P,
\end{equation}
where $R$ is reduced to the constant $1$ when $q_1=0,$ and $R(z)=\prod\limits_{j=1}^{q_1}(z-z_j)^{m_j}$ when $q_1\geq 1,$ while $P$ is a complex constant or a non-constant polynomial such that
\begin{equation}\label{eq3.11}
P=\int_{0}^z P_1(\eta)d\eta+A,
\end{equation}
where $A$ is a finite complex constant. Next we rewrite \eqref{eq3.10} into
\begin{equation}\label{eq3.12}
\frac {f'} {f(f-1)(f-c)}=Re^P\frac {g'} {g(g-1)(g-c)}.
\end{equation}
\vskip 2mm
\par On the other hand, from \eqref{eq3.2}, Lemma \ref{Lemma2.2}(i), Lemma \ref{Lemma2.3}(ii), Lemma \ref{lemma2.4}, Remark \ref{Remark2.1}-Remark \ref{Remark2.3} and the assumption that $f$ and $g$ are non-constant meromorphic functions that share $0,$ $1,$ $c$ IM and $\infty$ CM, we have
from we have
\begin{equation}\label{eq3.13}
2T(r,f)=\overline{N}\left (r, f\right )+\overline{N}\left (r,\frac {1} {f} \right )+\overline{N} \left(r,\frac {1} {f-1} \right)+\overline{N} \left(r,\frac {1} {f-c} \right)+O(\log r),
\end{equation}
and
\begin{equation}\label{eq3.14}
N^{\ast}(r,0;f,g)+N^{\ast}(r,1;f,g)+N^{\ast}(r,c;f,g)+N_{(2}\left(r,f\right)=O(\log r), \, \, \text{as} \, \, r\rightarrow\infty.
\end{equation}
Here and in what follows, $N_{(2}\left(r,f\right)$ denotes the counting function of the multiple poles of $f$ in $|z|<r.$ From \eqref{eq3.13}, \eqref{eq3.14} and Lemma \ref{Lemma2.6} we have
\begin{equation}\label{eq3.15}
\frac {1}{7} T(r,f)\leq\sum\limits_{j=1}^{3}\sum\limits_{l=1}^{6}\overline{N}_{(1,l)}\left(r,a_j;f,g\right)
+\sum\limits_{j=1}^{3}\sum\limits_{m=2}^{6}\overline{N}_{(m,1)}\left(r,a_j;f,g\right)
+O(\log r),
\end{equation}
 as $r\rightarrow\infty.$ From \eqref{eq3.12} we consider the following two cases:
\vskip 2mm
\par {\bf Case 1.} Suppose that $P_1$ is reduced to the constant $0.$ Then, it follows from \eqref{eq3.11} that $P$ is reduced to the finite complex constant $A.$ Therefore, \eqref{eq3.12} can be rewritten into
\begin{equation}\label{eq3.16}
\frac {f'} {f(f-1)(f-c)}=Re^A\frac {g'} {g(g-1)(g-c)}.
\end{equation}
According to \eqref{eq3.12}, \eqref{eq3.15} and the assumption that $f$ and $g$ share $0,$ $1,$ $c$ IM and $\infty$ CM, we have a discussion as follows:
\vskip 2mm
\par {\bf Subcase 1.1.} Suppose that
\begin{equation}\label{eq3.17}
\overline{N}_{(1,1)}\left (r,0;f,g \right )+\overline{N}_{(1,1)}\left (r,1;f,g \right ) +\overline{N}_{(1,1)}\left (r,c;f,g \right ) \neq O(\log r), \, \, \, \text{as} \, \, r\rightarrow\infty.
\end{equation}
Here and in what follows, $\overline{N}_{(1,1)}\left (r,a;f,g \right )$ denotes the reduced counting function of the common simple zeros of $f-a$ and $g-a$ in $|z|<r$ for the value $a\in\{0,1,c\}.$ From \eqref{eq3.17} we suppose, without loss of generality, that
\begin{equation}\label{eq3.18}
\overline{N}_{(1,1)}\left (r,0;f,g \right) \neq O(\log r), \, \, \, \text{as} \, \, r\rightarrow\infty.
\end{equation}
From \eqref{eq3.18} we see that there exists
an a infinite sequence $\{z_{1,j}\}_{j=1}^{\infty}$ of the common simple zeros of $f$ and $g$ in the complex plane, such that $z_{1,j}\rightarrow\infty,$ as $j\rightarrow\infty.$ Then, it follows by \eqref{eq3.16} that $R(z_{1,j})e^A=1$ for each $j\in\Bbb{Z^{+}}.$ Combining this with the fact that $R$ is a non-vanishing rational function, we deduce that $Re^A$ is reduced to the constant $1,$ and so \eqref{eq3.16} can be rewritten into
\begin{equation}\label{eq3.19}
\frac {f'} {f(f-1)(f-c)}=\frac {g'} {g(g-1)(g-c)}.
\end{equation}
 From \eqref{eq3.19} and the assumption that $f$ and $g$ share $0,$ $1,$ $c$ IM we deduce that $f$ and $g$ share $0,$ $1,$ $c$ CM. This together with the assumption that $f$ and $g$ share $\infty$ CM, we get the assertion of Theorem \ref{Theorem1.1}.
\vskip 2mm
\par {\bf Subcase 1.2.} Suppose that there exists some positive integer $k$ satisfying $2\leq k\leq 6,$ such that
\begin{equation}\label{eq3.20}
\overline{N}_{(1,k)}\left (r,0;f,g \right )+\overline{N}_{(1,k)}\left (r,1;f,g \right ) +\overline{N}_{(1,k)}\left (r,c;f,g \right )\neq O(\log r), \, \, \text{as} \, \, r\rightarrow\infty,
\end{equation}
 where and in what follows, $\overline{N}_{(1,k)}\left (r,a;f,g \right )$ denotes the reduced counting function of those common zeros of $f-a$ and $g-a$ in $|z|<r$ that are simple zeros of $f-a,$ and are zeros of $g-a$ of multiplicity $k$ for $a\in\{0,1,c\}.$ From \eqref{eq3.20} we suppose, without loss of generality, that
 \begin{equation}\label{eq3.21}
\overline{N}_{(1,k)}\left (r,0;f,g \right )\neq O(\log r), \, \, \text{as} \, \, r\rightarrow\infty.
\end{equation}
 From \eqref{eq3.21} we see that there exists an infinite sequence $\{z_{2,j}\}_{j=1}^{\infty}$ of those common zeros of $f$ and $g$ in the complex plane  that are simple zeros of $f,$ and are zeros of $g$ of multiplicity $k.$ Combining this with \eqref{eq3.16}, we deduce
\begin{equation}\label{eq3.22}
R(z_{2,j})e^A=\frac{1}{k} \, \, \ \text{ for each } \, \, j\in\Bbb{Z^{+}}.
\end{equation}
Since $R$ is a non-vanishing rational function, we deduce from \eqref{eq3.22} that $Re^A$ is reduced to the constant $1/k,$ and so \eqref{eq3.16} can be rewritten into
\begin{equation}\label{eq3.23}
\frac {kf'} {f(f-1)(f-c)}=\frac {g'} {g(g-1)(g-c)}.
\end{equation}
 From \eqref{eq3.23}, the assumption that $f$ and $g$ share $0,$ $1,$ $c$ IM, and the supposition that $k$ is a positive integer satisfying $2\leq k\leq 6,$ we deduce that if $z_{0,j}$ is a zero of $f-a_j$ of multiplicity $\nu_{0,j}$ for $1\leq j\leq 3$ and $j\in\Bbb{Z},$ then $z_{0,j}$ is a zero of $g-a_j$ of multiplicity $k\nu_{0,j}$ for $1\leq j\leq 3$ and $j\in\Bbb{Z}.$ This implies that the functions
  \begin{equation}\label{eq3.24}
\varphi_1=:\frac {kf'} {f(f-1)}-\frac {g'} {g(g-1)}
\end{equation}
and
 \begin{equation}\label{eq3.25}
\varphi_2=:\frac {kf'} {f(f-c)}-\frac {g'} {g(g-c)}
\end{equation}
are two entire functions.  Therefore, from \eqref{eq3.2}, \eqref{eq3.24} and Lemma \ref{lemma2.4}, we deduce
  \begin{equation}\label{eq3.26}
  \begin{aligned}
T(r,\varphi_1)&=m(r,\varphi_1)=m\left(r,\frac {kf'} {f(f-1)}-\frac {g'} {g(g-1)}\right)\\
&=m\left(r,\frac {kf'} {f-1}-\frac {kf'} {f}-\frac {g'} {g-1}+\frac {g'} {g}\right)\leq m\left(r,\frac {f'} {f-1}\right)+m\left(r,\frac {f'} {f}\right)\\
&\quad+m\left(r,\frac {g'} {g-1}\right)+m\left(r,\frac {g'} {g}\right)+O(1)=O(\log r)
  \end{aligned}
\end{equation}
and
 \begin{equation}\label{eq3.27}
  \begin{aligned}
T(r,\varphi_2)&=m(r,\varphi_2)=m\left(r,\frac {kf'} {f(f-c)}-\frac {g'} {g(g-c)}\right)\\
&=m\left(r,\frac{1}{c}\left(\frac {kf'} {f-c}-\frac {kf'} {f}-\frac {g'} {g-c}+\frac {g'} {g}\right)\right)\leq m\left(r,\frac {f'} {f-c}\right)\\
&\quad +m\left(r,\frac {f'} {f}\right)+m\left(r,\frac {g'} {g-c}\right)+m\left(r,\frac {g'} {g}\right)+O(1)=O(\log r),
  \end{aligned}
\end{equation}
as $r\rightarrow\infty.$ From \eqref{eq3.26} and \eqref{eq3.27} we deduce that either $\varphi_j=0$ with $j\in\{1,2\}$ or that $\varphi_j$ with $j\in\{1,2\}$ is reduced to a non-vanishing polynomial. Suppose that $\varphi_1=0.$ Then, it follows by \eqref{eq3.24} that
  \begin{equation}\label{eq3.28}
\frac {kf'} {f(f-1)}=\frac {g'} {g(g-1)}.
\end{equation}
 By \eqref{eq3.23} and \eqref{eq3.28} we deduce $f=g.$ This contradicts the assumption that $f$ and $g$ are two distinct and non-constant meromorphic functions.
 Similarly, we get a contradiction from \eqref{eq3.23} and \eqref{eq3.25} provided that $\varphi_2=0.$ Next we suppose that $\varphi_1$ and $\varphi_2$ are non-vanishing  polynomials. Integrating both sides of \eqref{eq3.24}, we get
 \begin{equation}\label{eq3.29}
 \log \frac{g(f-1)^k}{f^k(g-1)}=\int^z_{0}\varphi_1(\eta)d\eta+c_1,
 \end{equation}
where $c_1$ is a finite complex constant. By \eqref{eq3.29} we have
\begin{equation}\label{eq3.30}
 \frac{g(f-1)^k}{f^k(g-1)}=A_1e^{\varphi_3}=:f_1\quad\text{with}\quad A_1=e^{c_1}\quad\text{and}\quad \varphi_3(z)=\int^z_{0}\varphi_1(\eta)d\eta.
 \end{equation}
Next we rewrite \eqref{eq3.25} into
 \begin{equation}\label{eq3.31}
\frac {kf'} {f-c}-\frac {kf'} {f}-\frac {g'} {g-c}+\frac {g'} {g}=c\varphi_2.
\end{equation}
Integrate both sides of \eqref{eq3.31} and we get
\begin{equation}\label{eq3.32}
 \log \frac{g(f-c)^k}{f^k(g-c)}=c\int^z_{0}\varphi_2(\eta)d\eta+c_2,
 \end{equation}
where $c_2$ is a finite complex constant. By \eqref{eq3.32} we have
\begin{equation}\label{eq3.33}
 \frac{g(f-c)^k}{f^k(g-c)}=A_2e^{\varphi_4}=:f_2\quad\text{with}\quad A_2=e^{c_2}\quad\text{and}\quad \varphi_4(z)=c\int^z_{0}\varphi_2(\eta)d\eta.
 \end{equation}
From \eqref{eq3.23} and the assumption that $f$ and $g$ share $0,$ $1,$ $\infty$ IM, we deduce that if $z_{0,j}\in\Bbb{C}$ with  $1\leq j\leq 3$ and $j\in\Bbb{Z}$ is a common zero of $f-a_j$ and $g-a_j$  with $1\leq j\leq 3$ and $j\in\Bbb{Z},$ then $ z_{0,j}$ is a zero of $g-a_j$ of multiplicity larger than or equal to $k$ with $2\leq k\leq 6$ and $k\in\Bbb{Z}$ for $1\leq j\leq 3$ and $j\in\Bbb{Z}.$ Therefore, from \eqref{eq3.2}, Lemma \ref{Lemma2.2}(ii), Remark \ref{Remark2.1}, Lemma \ref{lemma2.4} and the first fundamental theorem, we deduce
\begin{equation*}\nonumber
\begin{aligned}
2T(r,g)&=\overline{N}\left(r,g\right)+\overline{N}\left(r,\frac{1}{g}\right)+\overline{N}\left(r,\frac{1}{g-1}\right)+\overline{N}\left(r,\frac{1}{g-c}\right)+O(\log r)\\
&\leq \overline{N}\left(r,g\right)+N\left(r,\frac{1}{g'}\right)+O(\log r)\\
&=\overline{N}\left(r,g\right)+m(r,g')+N(r,g')-m\left(r,\frac{1}{g'}\right)+O(\log r)+O(1)\\
&\leq 2\overline{N}\left(r,g\right)+m(r,g)+N(r,g)-m\left(r,\frac{1}{g'}\right)+m\left(r,\frac{g'}{g}\right)\\
&\quad +O(\log r)+O(1)\leq  2\overline{N}\left(r,g\right)+T(r,g)+O(\log r),\\
\end{aligned}
\end{equation*}
and so we have
\begin{equation}\label{eq3.34}
T(r,g)\leq  2\overline{N}\left(r,g\right)+O(\log r), \, \, \text{ as} \, \, r\rightarrow\infty.
\end{equation}
Since $f$ and $g$ are transcendental meromorphic functions such that $f$ and $g$ share $\infty$ CM, we deduce by \eqref{eq3.1}, \eqref{eq3.34} and Lemma \ref{Lemma2.5} that
\begin{equation}\label{eq3.35}
\lim\limits_{r\rightarrow\infty}\frac{\overline{N}_0\left(r,\infty; f,g\right)}{T(r,f)}=\lim\limits_{r\rightarrow\infty}\frac{\overline{N}_0\left(r,\infty;f,g\right)}{T(r,g)}
=\lim\limits_{r\rightarrow\infty}\frac{\overline{N}\left(r,g\right)}{T(r,g)}\geq \frac{1}{2},
\end{equation}
where and in what follows, $\overline{N}_0\left(r,\infty;f,g\right)$ denotes the reduced counting function of the common poles of $f$ and $g$ in $|z|<r.$
\vskip 2mm
\par By \eqref{eq3.30} and \eqref{eq3.33} we deduce that if $z_{12}\in\Bbb{C}$ is a common pole of $f$ and $g,$ then $z_{12}$ is a common zero of $f_1-1$ and $f_2-1.$ Combining this with \eqref{eq3.35} and the assumption that $f$ and $g$ share $\infty$ CM, we deduce
\begin{equation}\label{eq3.36}
\begin{aligned}
&\quad \lim\limits_{r\rightarrow\infty}\frac{\overline{N}_0\left(r,1; f_1, f_2\right)}{T(r,f)}=\lim\limits_{r\rightarrow\infty}\frac{\overline{N}_0\left(r,1; f_1, f_2\right)}{T(r,g)}\geq \lim\limits_{r\rightarrow\infty}\frac{\overline{N}_0\left(r,\infty; f,g\right)}{T(r,f)}\\
&=\lim\limits_{r\rightarrow\infty}\frac{\overline{N}_0\left(r,\infty;f,g\right)}{T(r,g)}
=\lim\limits_{r\rightarrow\infty}\frac{\overline{N}\left(r,g\right)}{T(r,g)}\geq \frac{1}{2}.
\end{aligned}
\end{equation}
From \eqref{eq3.30}, \eqref{eq3.33} and the Valiron-Mokhon$'$ko lemma (cf.\cite{23}) we deduce
\begin{equation}\label{eq3.37}
\overline{N}\left(r,f_1\right)+\overline{N}\left(r,f_2\right)
+\overline{N}\left(r,\frac{1}{f_1}\right)+\overline{N}\left(r,\frac{1}{f_2}\right)=0,
\end{equation}
\begin{equation}\label{eq3.38}
\begin{aligned}
T(r,f_1)&=T\left(r, \frac{g(f-1)^k}{f^k(g-1)}\right)\leq T\left(r,\frac{g}{g-1}\right)+T\left(r,\frac{(f-1)^k}{f^k}\right)\\
        &=kT(r,f)+T(r,g)+O(1)
\end{aligned}
\end{equation}
and
\begin{equation}\label{eq3.39}
\begin{aligned}
T(r,f_2)&=T\left(r, \frac{g(f-c)^k}{f^k(g-c)}\right)\leq T\left(r,\frac{g}{g-c}\right)+T\left(r,\frac{(f-c)^k}{f^k}\right)\\
        &=kT(r,f)+T(r,g)+O(1),
\end{aligned}
\end{equation}
as $r\rightarrow\infty.$ By \eqref{eq3.38} and \eqref{eq3.39} we have
\begin{equation}\label{eq3.40}
\begin{aligned}
T(r,f_1)+T(r,f_2)\leq 2kT(r,f)+2T(r,g)+O(1), \, \, \text{as} \, \, r\rightarrow\infty.
\end{aligned}
\end{equation}
By \eqref{eq3.1}, \eqref{eq3.36} and \eqref{eq3.40} we have
\begin{equation}\label{eq3.41}
 \lim\limits_{r\rightarrow\infty}\frac{\overline{N}_0\left(r,1; f_1, f_2\right)}{T(r,f_1)+T(r,f_2)}\geq \lim\limits_{r\rightarrow\infty}\frac{\overline{N}\left(r,g\right)}{2kT(r,f)+2T(r,g)+O(1)}\geq \frac{1}{4k+4}.
\end{equation}
By  \eqref{eq3.37}, \eqref{eq3.41} and Lemma \ref{lemma2.7}, we see that there exist two integers $s$ and $t$ with $|s|+|t|>0,$ such that $f^s_1f^t_2=1.$ This together with the definitions of $f_1$ and $f_2$ in \eqref{eq3.30} and \eqref{eq3.33} respectively gives
\begin{equation}\label{eq3.42}
\left(\frac{g(f-1)^k}{f^k(g-1)}\right)^s\left(\frac{g(f-c)^k}{f^k(g-c)}\right)^t=1.
\end{equation}
From \eqref{eq3.42} we have a discussion as follows:
\vskip 2mm
\par Suppose that one of $s$ and $t$ is equal to zero, say $s=0.$ Then $t\neq 0.$ Therefore, \eqref{eq3.42} can be rewritten into
\begin{equation}\label{eq3.43}
\left(\frac{f-c}{f}\right)^{kt}=\left(\frac{g-c}{g}\right)^t.
\end{equation}
 We take the Nevanlinna$'$s characteristic functions on both sides of \eqref{eq3.43}, and then use the Valiron-Mokhon$'$ko lemma (cf.\cite{23}) to get
 \begin{equation}\label{eq3.44}
k|t|T(r,f)=|t|T(r,g)+O(1), \, \, \text{as} \, \, r\rightarrow\infty.
\end{equation}
 By \eqref{eq3.1}, \eqref{eq3.44} and the supposition that $k$ is a positive integer such that $k\geq 2,$ we have $T(r,f)=O(\log r).$ Combining this with
Lemma \ref{Lemma2.5} and the supposition that $f$ is a transcendental meromorphic function in the complex plane, we get a contradiction.
 \vskip 2mm
\par Suppose that $s+t=0.$ Then $s=-t.$ Moreover, from $|s|+|t|> 0$ we see that $s$ and $t$ are two non-zero inters. Therefore,
\eqref{eq3.42} can be rewritten into
\begin{equation}\label{eq3.45}
\left(\frac{f-c}{f-1}\right)^{kt}=\left(\frac{g-c}{g-1}\right)^t.
\end{equation}
We take the Nevanlinna$'$s characteristic functions on both sides of \eqref{eq3.45} and then use the Valiron-Mokhon$'$ko lemma (cf.\cite{23}) to get \eqref{eq3.44}.
By \eqref{eq3.1}, \eqref{eq3.44} and the supposition that $k$ is a positive integer such that $k\geq 2,$ we have $T(r,f)=O(\log r).$ Combining this with
Lemma \ref{Lemma2.5} and the supposition that $f$ is a transcendental meromorphic function in the complex plane, we get a contradiction.
 \vskip 2mm
\par Suppose that $s+t\neq 0.$ We first rewrite \eqref{eq3.42} into
\begin{equation}\label{eq3.46}
\left(\frac{(f-1)^s(f-c)^t}{f^{s+t}}\right)^{k}=\frac{(g-1)^s(g-c)^t}{g^{s+t}}.
\end{equation}
We take the Nevanlinna$'$s characteristic functions on both sides of \eqref{eq3.46} and then use the supposition $s+t\neq 0$ and the Valiron-Mokhon$'$ko lemma (cf.\cite{23}) to get
\begin{equation}\label{eq3.47}
k|s+t|T(r,f)=|s+t|T(r,g)+O(1), \, \, \text{as} \, \, r\rightarrow\infty.
\end{equation}
By \eqref{eq3.1}, \eqref{eq3.47} and the supposition that $k$ is a positive integer such that $k\geq 2,$ we have $T(r,f)=O(\log r).$ Combining this with
Lemma \ref{Lemma2.5} and the supposition that $f$ is a transcendental meromorphic function in the complex plane, we get a contradiction.
\vskip 2mm
\par {\bf Subcase 1.3.} Suppose that there exists some positive integer $k$ satisfying $2\leq k\leq 6,$ such that
\begin{equation}\label{eq3.48}
\overline{N}_{(k,1)}\left (r,0;f,g \right )+\overline{N}_{(k,1)}\left (r,1;f,g \right ) +\overline{N}_{(k,1)}\left (r,c;f,g \right ) \neq O(\log r), \, \, \text{as} \, \,  r\rightarrow\infty.
\end{equation}
where and in what follows, $\overline{N}_{(k,1)}\left (r,a;f,g \right )$ denotes the reduced counting function of those common zeros of $f-a$ and $g-a$ in $|z|<r$ that are zeros of $f-a$ of multiplicity $k,$ and are simple zeros of $g-a$ for $a\in\{0,1,c\}.$ From \eqref{eq3.48} we suppose, without loss of generality, that
\begin{equation}\label{eq3.49}
\overline{N}_{(k,1)}\left (r,0;f,g \right ) \neq O(\log r), \, \, \text{as} \, \,  r\rightarrow\infty.
\end{equation}
From \eqref{eq3.49} we see that there exists an infinite sequence $\{z_{3,j}\}_{j=1}^{\infty}$ of those common zeros of $f$ and $g$ in the complex plane
that are zeros of $f$ of multiplicity $k,$ and simple zeros of $g.$ Combining this with \eqref{eq3.16}, we deduce
  \begin{equation}\label{eq3.50}
R(z_{3,j})e^A=k \, \, \ \text{ for each } \, \, j\in\Bbb{Z^{+}}.
\end{equation}
Since $R$ is a non-vanishing rational function, and $A$ is a complex constant, we deduce from \eqref{eq3.50} that $Re^A$ is reduced to the constant $k,$ and so \eqref{eq3.16} can be rewritten into
\begin{equation}\label{eq3.51}
\frac {f'} {f(f-1)(f-c)}=\frac {kg'} {g(g-1)(g-c)}.
\end{equation}
From \eqref{eq3.51}, the assumption that $f$ and $g$ share $0,$ $1,$ $c$ IM, and the supposition that $k$ is a positive integer satisfying $2\leq k\leq 6,$ we deduce that if $\tilde{z}_{0,j}$ is a zero of $g-a_j$ of multiplicity $\tilde{\nu}_{0,j}$ with $1\leq j\leq 3$ and $j\in\Bbb{Z},$ then $\tilde{z}_{0,j}$ is a zero of $f-a_j$ of multiplicity $k\tilde{\nu}_{0,j}$ with $1\leq j\leq 3$ and $j\in\Bbb{Z}.$ This implies that the functions
\begin{equation}\label{eq3.52}
\varphi_5=:\frac {f'} {f(f-1)}-\frac {kg'} {g(g-1)}
\end{equation}
and
 \begin{equation}\label{eq3.53}
\varphi_6=:\frac {f'} {f(f-c)}-\frac {kg'} {g(g-c)}
\end{equation}
are two entire functions.  Next we use \eqref{eq3.2}, \eqref{eq3.52}, \eqref{eq3.53},  Lemma \ref{lemma2.4}, the lines of \eqref{eq3.26} and \eqref{eq3.27} in Subcase 1.2 to deduce that either $\varphi_j=0$ with $j\in\{5,6\}$ or that $\varphi_j$ with $j\in\{5,6\}$ is reduced to a non-vanishing polynomial. Suppose that one of $\varphi_5=0$ and $\varphi_6=0$ holds. Then, we use \eqref{eq3.51}-\eqref{eq3.53} and the similar reasonings as in Subcase 1.2 to get a contradiction. Next we suppose that $\varphi_5$ and $\varphi_6$ are non-vanishing polynomials. Then, we use \eqref{eq3.52}, \eqref{eq3.53}, the the similar reasonings from the line before \eqref{eq3.29} to the line of \eqref{eq3.33} to deduce
\begin{equation}\label{eq3.54}
 \frac{(f-1)g^k}{f(g-1)^k}=A_3e^{\varphi_7}=:f_3 \quad\text{with}\quad  \varphi_7(z)=\int^z_{0}\varphi_5(\eta)d\eta.
 \end{equation}
and
\begin{equation}\label{eq3.55}
 \frac{(f-c)g^k}{f(g-c)^k}=A_4e^{\varphi_8}=:f_4\quad\text{with}\quad c\varphi_8(z)=\int^z_{0}\varphi_6(\eta)d\eta,
 \end{equation}
where and in what follows, $A_3$ and $A_4$ are two finite and non-zero complex constants.
From \eqref{eq3.51} and the assumption that $f$ and $g$ share $0,$ $1,$ $c$ IM, we deduce that if $\tilde{z}_{0,j}\in\Bbb{C}$ with $1\leq j\leq 3$ and $j\in\Bbb{Z}$ is a common zero of $f-a_j$ and $g-a_j$  with $1\leq j\leq 3$ and $j\in\Bbb{Z},$ then $ \tilde{z}_{0,j}$ is a zero of $f-a_j$ of multiplicity larger than or equal to $k$ with $2\leq k\leq 6$ and $k\in\Bbb{Z}$ for $1\leq j\leq 3$ and $j\in\Bbb{Z}.$ Therefore, from \eqref{eq3.2}, Lemma \ref{Lemma2.2}(ii), Remark \ref{Remark2.1}, Lemma \ref{lemma2.4} and the first fundamental theorem, we deduce
\begin{equation}\nonumber
\begin{aligned}
2T(r,f)&=\overline{N}\left(r,f\right)+\overline{N}\left(r,\frac{1}{f}\right)+\overline{N}\left(r,\frac{1}{f-1}\right)+\overline{N}\left(r,\frac{1}{f-c}\right)+O(\log r)\\
&\leq \overline{N}\left(r,f\right)+N\left(r,\frac{1}{f'}\right)+O(\log r)\\
&=\overline{N}\left(r,f\right)+m(r,f')+N(r,f')-m\left(r,\frac{1}{f'}\right)+O(\log r)+O(1)\\
&\leq 2\overline{N}\left(r,f\right)+m(r,f)+N(r,f)-m\left(r,\frac{1}{f'}\right)+m\left(r,\frac{f'}{f}\right)\\
&\quad +O(\log r)+O(1)\leq  2\overline{N}\left(r,f\right)+T(r,f)+O(\log r),\\
\end{aligned}
\end{equation}
and so we have
\begin{equation}\label{eq3.56}
T(r,f)\leq  2\overline{N}\left(r,f\right)+O(\log r), \, \, \text{ as} \, \, r\rightarrow\infty.
\end{equation}
Since $f$ and $g$ are transcendental meromorphic functions such that $f$ and $g$ share $\infty$ CM, we deduce by \eqref{eq3.1}, \eqref{eq3.56} and Lemma \ref{Lemma2.5} that
\begin{equation}\label{eq3.57}
\lim\limits_{r\rightarrow\infty}\frac{\overline{N}_0\left(r,\infty; f,g\right)}{T(r,f)}=\lim\limits_{r\rightarrow\infty}\frac{\overline{N}_0\left(r,\infty;f,g\right)}{T(r,g)}=\lim\limits_{r\rightarrow\infty}
\frac{\overline{N}\left(r,f\right)}{T(r,f)}\geq \frac{1}{2}.
\end{equation}
\vskip 2mm
\par From \eqref{eq3.54} and \eqref{eq3.55} we deduce that if $z_{34}\in\Bbb{C}$ is a common pole of $f$ and $g,$ then $z_{34}$ is a common zero of $f_3-1$ and $f_4-1.$ Combining this with \eqref{eq3.57} and the assumption that $f$ and $g$ share $\infty$ CM, we deduce
\begin{equation}\label{eq3.58}
\begin{aligned}
&\quad \lim\limits_{r\rightarrow\infty}\frac{\overline{N}_0\left(r,1; f_3, f_4\right)}{T(r,f)}=\lim\limits_{r\rightarrow\infty}\frac{\overline{N}_0\left(r,1; f_3, f_4\right)}{T(r,g)}\geq \lim\limits_{r\rightarrow\infty}\frac{\overline{N}_0\left(r,\infty; f,g\right)}{T(r,f)}\\
&=\lim\limits_{r\rightarrow\infty}\frac{\overline{N}_0\left(r,\infty;f,g\right)}{T(r,g)}=\lim\limits_{r\rightarrow\infty}\frac{\overline{N}\left(r,f\right)}{T(r,f)}\geq \frac{1}{2}.
\end{aligned}
\end{equation}
By \eqref{eq3.54}, \eqref{eq3.55} and the Valiron-Mokhon$'$ko lemma (cf.\cite{23}) we deduce
\begin{equation}\label{eq3.59}
\overline{N}\left(r,f_3\right)+\overline{N}\left(r,f_4\right)
+\overline{N}\left(r,\frac{1}{f_3}\right)+\overline{N}\left(r,\frac{1}{f_4}\right)=0,
\end{equation}
\begin{equation}\label{eq3.60}
\begin{aligned}
T(r,f_3)&=T\left(r,  \frac{(f-1)g^k}{f(g-1)^k} \right)\leq T\left(r,\frac{f-1}{f}\right)+T\left(r,\frac{g^k}{(g-1)^k}\right)\\
        &=T(r,f)+kT(r,g)+O(1)
\end{aligned}
\end{equation}
and
\begin{equation}\label{eq3.61}
\begin{aligned}
T(r,f_4)&=T\left(r,  \frac{(f-c)g^k}{f(g-c)^k}\right)\leq T\left(r,\frac{f-c}{f}\right)+T\left(r,\frac{g^k}{(g-c)^k}\right)\\
        &=T(r,f)+kT(r,g)+O(1),
\end{aligned}
\end{equation}
as $r\rightarrow\infty.$ By \eqref{eq3.60} and \eqref{eq3.61} we have
\begin{equation}\label{eq3.62}
T(r,f_3)+T(r,f_4)\leq 2T(r,f)+2kT(r,g)+O(1), \, \, \text{as} \, \, r\rightarrow\infty.
\end{equation}
By \eqref{eq3.1}, \eqref{eq3.58},  \eqref{eq3.62}, Lemma \ref{Lemma2.5} and the obtained result that $f$ and $g$ are transcendental meromorphic functions we deduce
\begin{equation}\label{eq3.63}
 \lim\limits_{r\rightarrow\infty}\frac{\overline{N}_0\left(r,1; f_3, f_4\right)}{T(r,f_3)+T(r,f_4)}\geq \lim\limits_{r\rightarrow\infty}\frac{\overline{N}\left(r,f\right)}{2T(r,f)+2kT(r,g)}\geq \frac{1}{4(k+1)}.
\end{equation}
By  \eqref{eq3.54}, \eqref{eq3.55}, \eqref{eq3.59}, \eqref{eq3.63} and Lemma \ref{lemma2.7}, we see that there exist some two integers $s$ and $t$ with $|s|+|t|>0,$ such that $f^s_3f^t_4=1.$ This together with the definitions of $f_3$ and $f_4$ in \eqref{eq3.54} and \eqref{eq3.55} respectively gives
\begin{equation}\label{eq3.64}
\left(\frac{(f-1)g^k}{f(g-1)^k}\right)^s\left(\frac{(f-c)g^k}{f(g-c)^k}\right)^t=1.
\end{equation}
Next we use \eqref{eq3.64} and the similar reasonings from the line after \eqref{eq3.42} to the end of Subcase 1.2 to get a contradiction.
\vskip 2mm
\par {\bf Case 2.} Suppose that $P_{1}$ is a non-zero polynomial, then it follows by the definition of $P$ in \eqref{eq3.11} we see that $P$ is a non-constant polynomial. Next we let
\begin{equation}\label{eq3.65}
P(z)=p_nz^n+p_{n-1}z^{n-1}+\cdots+p_1z+p_0,
\end{equation}
where $n$ is a positive integer, and $p_n,$ $p_{n-1},$ $\ldots,$ $p_1,$ $p_0$ with $p_n\neq 0$ are finite complex constants. Next we let $p_n=|p_n|e^{i\theta_{n}}$ with $|p_n|>0$ and $\theta_n \in [0,2\pi),$ and consider the following $2n$ angles for any given positive number $\varepsilon$ satisfying $0<\varepsilon<\frac{\pi}{8n}:$
\begin{equation}\nonumber
S_j: -\frac{\theta_{n}}{n}+(2j-1)\frac{\pi}{2n}+\varepsilon<\theta<-\frac{\theta_{n}}{n}+(2j+1)\frac{\pi}{2n}-\varepsilon,
\end{equation}
where $j$ is an integer satisfying $0\leq j\leq 2n-1.$ Then, it follows from \eqref{eq3.65} and Lemma \ref{lemma2.8} that there exists a positive number $R=R(\varepsilon)$ such that for $|z|=r>R,$ we have
\begin{equation}\label{eq3.66}
\text{Re}(P(z))>|p_{n}|(1-\varepsilon)r^n\sin(n\varepsilon), \, \, \text{when} \, \, z\in S_j \, \, \text{and} \, \, j \, \,\text{is an even integer},
\end{equation}
while
\begin{equation}\label{eq3.67}
\text{Re}(P(z))<-|p_{n}|(1-\varepsilon)r^n\sin(n\varepsilon), \, \, \text{when} \, \, z\in S_j \, \, \text{and} \, \, j \, \, \text{is an odd integer}.
\end{equation}
For convenience we next let
\begin{equation}\label{eq3.68}
\Omega\left(\alpha_{j,\varepsilon},\beta_{j,\varepsilon}\right)=\left\{z\in\Bbb{C}: \alpha_{j,\varepsilon}<\arg z<\beta_{j,\varepsilon}\right\}
\end{equation}
and
\begin{equation}\label{eq3.69}
\overline{\Omega}\left(\alpha_{j,2\varepsilon},\beta_{j,2\varepsilon}\right)=\left\{z\in\Bbb{C}: \alpha_{j,2\varepsilon}\leq \arg z\leq \beta_{j,2\varepsilon}\right\}
\end{equation}
with
\begin{equation}\label{eq3.70}
\alpha_{j,\varepsilon}=-\frac{\theta_{n}}{n}+(2j-1)\frac{\pi}{2n}+\varepsilon \quad\text{and}\quad\beta_{j,\varepsilon}=-\frac{\theta_{n}}{n}+(2j+1)\frac{\pi}{2n}-\varepsilon
\end{equation}
and
\begin{equation}\label{eq3.71}
\alpha_{j,2\varepsilon}=-\frac{\theta_{n}}{n}+(2j-1)\frac{\pi}{2n}+2\varepsilon \quad\text{and}\quad\beta_{j,2\varepsilon}=-\frac{\theta_{n}}{n}+(2j+1)\frac{\pi}{2n}-2\varepsilon
\end{equation}
for each integer $j$ satisfying $0\leq j\leq 2n-1.$ Then, we have the following claim:
\vskip 2mm
\par
{\bf Claim 2I.} Based upon the assumptions of Theorem \ref{Theorem1.1} and the supposition of Case 2, there are at most finitely many common zeros of $f-a_l$ and $g-a_l$ in $\overline{\Omega}\left(\alpha_{j,2\varepsilon},\beta_{j,2\varepsilon}\right)$ for each integer $l\in\{1,2,3\}$ and each integer $j$ satisfying $0\leq j\leq 2n-1.$
\vskip 2mm
\par We prove Claim 2I: on the contrary, we suppose that Claim 2I is not valid. Then, there are infinitely many zeros of $f-a_l$ in $\overline{\Omega}\left(\alpha_{j,2\varepsilon},\beta_{j,2\varepsilon}\right)$  for each integer $l\in\{1,2,3\}$ and each integer $j$ satisfying $0\leq j\leq 2n-1.$ Next we let $\{z_{2,l,m}\}_{m=1}^{+\infty}\subset\overline{\Omega}\left(\alpha_{j,2\varepsilon},\beta_{j,2\varepsilon}\right)$ be the infinite sequence
of all the zeros of $f-a_l$ on $\overline{\Omega}\left(\alpha_{j,2\varepsilon},\beta_{j,2\varepsilon}\right)$ for $l\in\{1,2,3\}$ and the integer $j$ satisfying $0\leq j\leq 2n-1,$ where each distinct point in $\{z_{2,l,m}\}_{m=1}^{+\infty}$ is repeated as many times as its multiplicity of a zero of $f-a_l$ with $l\in\{1,2,3\},$ and the infinite sequence of $\{z_{2,l,m}\}_{m=1}^{+\infty}$ is arranged according to increasing moduli, such that
\begin{equation}\label{eq3.72}
\lim\limits_{m\rightarrow\infty}|z_{2,l,m}|=\infty
\end{equation}
and
\begin{equation}\label{eq3.73}
|z_{2,l,1}|\leq  |z_{2,l,2}|\leq \cdots\leq |z_{2,l,m}|\leq \cdots
\end{equation}
 for $l\in\{1,2,3\}.$ We consider the following two subcases:
\vskip 2mm
\par
{\bf Subcase 2.1.} Suppose that $j$ with $0\leq j \leq 2n-1$ and $j\in\Bbb{Z}$ is an even integer. Then, by \eqref{eq3.12}, \eqref{eq3.66}, \eqref{eq3.72}, \eqref{eq3.73} and \eqref{eq3.14} we deduce that there exist some positive integer $N_{l,j}$ and a large positive number $\mathcal{R}_{l,j}$ satisfying $\mathcal{R}_{l,j}\geq \mathcal{R}$ for $l\in\{1,2,3\}$ and the even integer $j$ satisfying $0\leq j\leq 2n-1,$ such that
\begin{equation}\label{eq3.74}
|z_{2,l,m}|>\mathcal{R}_{l,j}, \, \, \text{as} \, \, m\geq N_{l,j}\, \, \text{and} \, \, m\in\Bbb{Z},
\end{equation}
and such that each point $z_{2,l,m}$ satisfying \eqref{eq3.74} is not a common multiple zero of $f-a_l$ and $g-a_l$ with $l\in\{1,2,3\},$ and satisfies
\begin{equation}\label{eq3.75}
n(z_{2,l,m})=|R(z_{2,l,m})e^{\text{Re}(P(z_{2,l,m}))}|\geq e^{|p_n|(1-2\varepsilon)|z_{2,l,m}|^n\sin(n\varepsilon)}\geq 4,
\end{equation}
as $ m\geq N_{l,j}$ and $m\in\Bbb{Z}$ for $l\in\{1,2,3\}$ and the even integer $j$ satisfying $0\leq j\leq 2n-1.$ Here and in what follows, $n(z_{2,l,m})$ denotes the multiplicity of $z_{2,l,m}$ of a zero of $f-a_l$ for $l\in\{1,2,3\},$ the positive integer $m$ satisfying $m\geq N_{l,j}$ and the even integer $j$ satisfying $0\leq j\leq 2n-1.$ Therefore, from \eqref{eq3.12}, \eqref{eq3.14}-\eqref{eq3.75}  and the assumption that $f$ and $g$ share $0,$ $1$ and $c$ IM, we deduce that $z_{2,l,m}$ is also a simple zero of $g-a_l$ for $l\in\{1,2,3\},$ the positive integer $m$ satisfying $m\geq N_{l,j},$ and the even integer $j$ satisfying $0\leq j\leq 2n-1.$ Next we use Lemma \ref{lemma2.10} for $\overline{\Omega}\left(\alpha_{j,2\varepsilon},\beta_{j,2\varepsilon}\right)$ to deduce
\begin{equation}\label{eq3.76}
\begin{aligned}
S_{\alpha_{j,2\varepsilon}, \beta_{j,2\varepsilon}}(r,f)&\leq\overline{C}_{\alpha_{j,2\varepsilon}, \beta_{j,2\varepsilon}}\left(r,\frac{1}{f-a_1}\right)+\overline{C}_{\alpha_{j,2\varepsilon}, \beta_{j,2\varepsilon}}\left(r,\frac{1}{f-a_2}\right)
\\
&\quad+\overline{C}_{\alpha_{j,2\varepsilon}, \beta_{j,2\varepsilon}}\left(r,\frac{1}{f-a_3}\right)+R_{\alpha_{j,2\varepsilon}\beta_{j,2\varepsilon}}(r,f)\\
\end{aligned}
\end{equation}
for the positive number $r$ such that $r>\max\limits_{l\in\{1,2,3\}}\max\limits_{0\leq \tilde{l}\leq n-1, \, \tilde{l}\in\Bbb{Z}}\mathcal{R}_{l,2\tilde{l}}\geq\mathcal{R}.$ From  \eqref{eq3.74}, \eqref{eq3.75} and Lemma \ref{lemma2.9} we have
\begin{equation}\label{eq3.77}
\begin{aligned}
&\quad \overline{C}_{\alpha_{j,2\varepsilon}, \beta_{j,2\varepsilon}}\left(r,\frac{1}{f-a_l}\right)\\
&=\left(\overline{C}_{\alpha_{j,2\varepsilon}, \beta_{j,2\varepsilon}}\left(r,\frac{1}{f-a_l}\right)-\overline{C}_{\alpha_{j,2\varepsilon}, \beta_{j,2\varepsilon}}\left(\mathcal{R}_{l,j},\frac{1}{f-a_l}\right)\right) \\
&\quad +\overline{C}_{\alpha_{j,2\varepsilon}, \beta_{j,2\varepsilon}}\left(\mathcal{R}_{l,j},\frac{1}{f-a_l}\right)\leq \frac{1}{4}C_{\alpha_{j,2\varepsilon}, \beta_{j,2\varepsilon}}\left(r,\frac{1}{f-a_l}\right)+O(1)\\
&\leq \frac{1}{4}S_{\alpha_{j,2\varepsilon}, \beta_{j,2\varepsilon}}\left(r,\frac{1}{f-a_l}\right)+O(1)=\frac{1}{4}S_{\alpha_{j,2\varepsilon} \beta_{j,2\varepsilon}}\left(r,f\right)+O(1)
\end{aligned}
\end{equation}
for the $l\in\{1,2,3\},$ the even integer $j$ satisfying $0\leq j\leq 2n-1,$ and the positive number $r$ such that $r>\max\limits_{l\in\{1,2,3\}}\max\limits_{0\leq \tilde{l}\leq n-1, \, \tilde{l}\in\Bbb{Z}}\mathcal{R}_{l,2\tilde{l}}\geq\mathcal{R}.$ From \eqref{eq3.2} and Lemma \ref{Lemma2.12} we have
\begin{equation}\label{eq3.78}
R_{\alpha_{j,2\varepsilon}\beta_{j,2\varepsilon}}(r,f)=O(1)
\end{equation}
for the even integer $j$ satisfying $0\leq j\leq 2n-1,$ and the positive number $r$ such that $r>\max\limits_{l\in\{1,2,3\}}\max\limits_{0\leq \tilde{l}\leq n-1, \, \tilde{l}\in\Bbb{Z}}\mathcal{R}_{l,2\tilde{l}}\geq\mathcal{R}.$ From \eqref{eq3.76}-\eqref{eq3.78} we deduce
\begin{equation}\label{eq3.79}
S_{\alpha_{j,2\varepsilon}, \beta_{j,2\varepsilon}}(r,f)=O(1)
\end{equation}
for the even integer $j$ satisfying $0\leq j\leq 2n-1,$ and the positive number $r$ such that $r>\max\limits_{l\in\{1,2,3\}}\max\limits_{0\leq \tilde{l}\leq n-1, \, \tilde{l}\in\Bbb{Z}}\mathcal{R}_{l,2\tilde{l}}\geq\mathcal{R}.$
Next we use Lemma \ref{lemma2.9} and the reasonings in the proof of Zheng \cite[pp.53-54, Lemma 2.2.2]{Zheng2009} to deduce
\begin{equation}\label{eq3.80}
\begin{aligned}
&\quad S_{\alpha_{j,2\varepsilon}, \beta_{j,2\varepsilon}}\left(r,f\right)+O(1)=S_{\alpha_{j,2\varepsilon}, \beta_{j,2\varepsilon}}\left(r,\frac{1}{f-a_l}\right)\geq C_{\alpha_{j,2\varepsilon}, \beta_{j,2\varepsilon}}\left(r,\frac{1}{f-a_l}\right)\\
&\geq 2\sin\left(2\varepsilon\omega_{\alpha_{j,2\varepsilon}}\right)\sum_{|z_{2,l,m}|\in \Omega\left(\alpha_{j,2\varepsilon}, \beta_{j,2\varepsilon}\right)}\left(\frac{1}{|z_{2,l,m}|^{\omega_{j,2\varepsilon}}}-\frac{|z_{2,l,m}|^{\omega_{j,2\varepsilon}}}{r^{2\omega_{j,2\varepsilon}}}\right)+O(1)\\
&\geq 2\sin\left(2\varepsilon\omega_{j,2\varepsilon}\right)\sum_{|z_{2,l,m}|>\mathcal{R}_{l,j}}\left(\frac{1}{|z_{2,l,m}|^{\omega_{j,2\varepsilon}}}
-\frac{|z_{2,l,m}|^{\omega_{j,2\varepsilon}}}{r^{2\omega_{j,2\varepsilon}}}\right)+O(1)\\
&=2\sin\left(2\varepsilon\omega_{j,2\varepsilon}\right)\int_{1}^{r}\left(\frac{1}{t^{\omega_{j,2\varepsilon}}}-\frac{t^{\omega_{j,2\varepsilon}}}
{r^{2\omega_{j,2\varepsilon}}}\right)dn_{0}(t)+O(1)\\
&=2\omega_{j,2\varepsilon}\sin\left(2\varepsilon\omega_{j,2\varepsilon}\right)\int_{1}^{r}n_{0}(t)\left(\frac{1}{t^{1+\omega_{j,2\varepsilon}}}
+\frac{t^{\omega_{j,2\varepsilon}}-1}{r^{2\omega_{j,2\varepsilon}}}\right)dt+O(1)\\
&\geq 2\omega_{j,2\varepsilon} \sin\left(2\varepsilon\omega_{j,2\varepsilon}\right)\int_{1}^{r} \frac{1}{t^{\omega_{j,2\varepsilon}}}dN_{0}(t)+O(1)\\
&=2\omega_{j,2\varepsilon} \sin\left(2\varepsilon\omega_{j,2\varepsilon}\right)\frac{N_{0}(r)}{r^{\omega_{j,2\varepsilon}}}+2\omega^2_{j,2\varepsilon} \sin\left(2\varepsilon\omega_{j,2\varepsilon}\right)\int_{1}^{r} \frac{N_{0}(t)}{t^{1+\omega_{j,2\varepsilon}}}dt+O(1)\\
&\geq 2\omega_{j,2\varepsilon}  \sin\left(2\varepsilon\omega_{j,2\varepsilon}\right)\frac{N_{0}(r)}{r^{\omega_{j,2\varepsilon}}}+O(1)
\end{aligned}
\end{equation}
for $l\in\{1,2,3\},$ the even integer $j$ satisfying $0\leq j\leq 2n-1,$ and the positive number $r$ such that $r>\max\limits_{l\in\{1,2,3\}}\max\limits_{0\leq \tilde{l}\leq n-1, \, \tilde{l}\in\Bbb{Z}}\mathcal{R}_{l,2\tilde{l}}\geq\mathcal{R}.$ Here and in what follows,
\begin{equation}\label{eq3.81}
\omega_{j,2\varepsilon} =\frac{\pi}{\beta_{j,2\varepsilon}-\alpha_{j,2\varepsilon}}=\frac{n\pi}{\pi-4n\varepsilon}
 \end{equation}
  for the integer $j$ satisfying  $0\leq j\leq 2n-1,$ while $N_{0}(t)$ is defined as $N_{0}(t)=\int_{1}^{t}\frac{n\left(u,\Omega\left(\alpha_{j,2\varepsilon},\beta_{j,2\varepsilon}\right),\frac{1}{f-a_l}\right)}{u}du$ with  $n\left(u,\Omega\left(\alpha_{j,2\varepsilon},\beta_{j,2\varepsilon}\right),\frac{1}{f-a_l}\right)$ being the number of zeros of $f-a_l$ for $l\in\{1,2,3\}$ in $\Omega\left(\alpha_{j,2\varepsilon},\beta_{j,2\varepsilon}\right)\cap\{z:1<|z|\leq  u\}$ counted with multiplicities. From \eqref{eq3.79} and \eqref{eq3.80} we have
\begin{equation}\label{eq3.82}
 2\omega_{j,2\varepsilon}  \sin\left(2\varepsilon\omega_{j,2\varepsilon}\right)\frac{N_{0}(r)}{r^{\omega_{j,2\varepsilon}}}\leq O(1)
  \end{equation}
 for $l\in\{1,2,3\}$ and the positive number $r$ such that
 \begin{equation}\nonumber
 r>\max\limits_{l\in\{1,2,3\}}\max\limits_{0\leq \tilde{l}\leq n-1, \, \tilde{l}\in\Bbb{Z}}\mathcal{R}_{l,2\tilde{l}}\geq\mathcal{R}.
 \end{equation}
 Now we use \eqref{eq3.74} and \eqref{eq3.75} to deduce
 \begin{equation}\nonumber
 |z_{2,l,m}|>\max\limits_{l\in\{1,2,3\}}\max\limits_{0\leq \tilde{l}\leq n-1, \, \tilde{l}\in\Bbb{Z}}\mathcal{R}_{l,2\tilde{l}}\geq \mathcal{R}
 \end{equation}
  for the positive integer $m$ satisfying $m>\max\limits_{l\in\{1,2,3\}}\max\limits_{0\leq \tilde{l}\leq n-1, \, \tilde{l}\in\Bbb{Z}}N_{l,2\tilde{l}}.$ Combining this with \eqref{eq3.72}, we deduce that for any given positive number $r$ satisfying $r>\max\limits_{l\in\{1,2,3\}}\max\limits_{0\leq \tilde{l}\leq n-1, \, \tilde{l}\in\Bbb{Z}}\geq \mathcal{R},$ there exists some positive integer $N_{2,l,m}(r)$  with $l\in\{1,2,3\}$ that satisfies
    \begin{equation}\label{eq3.83}
        N_{2,l,m}(r)>\max\limits_{l\in\{1,2,3\}}\max\limits_{0\leq \tilde{l}\leq n-1, \, \tilde{l}\in\Bbb{Z}}N_{l,2\tilde{l}},
        \end{equation}
    such that for the positive integer $m$ satisfying $m>N_{2,l,m}(r)$  with $l\in\{1,2,3\},$ we have
    \begin{equation}\label{eq3.84}
       |z_{2,l,m}| \geq r>\max\limits_{l\in\{1,2,3\}}\max\limits_{0\leq \tilde{l}\leq n-1, \, \tilde{l}\in\Bbb{Z}}\mathcal{R}_{l,2\tilde{l}}\geq \mathcal{R}.
  \end{equation}
   From \eqref{eq3.84} we have
\begin{equation}\label{eq3.85}
\begin{aligned}
 &\quad N_{0}(2|z_{2,l,m}|)=\int_{1}^{2|z_{2,l,m}|}\frac{n\left(t,{\Omega}\left(\alpha_{j,2\varepsilon},\beta_{j,2\varepsilon}\right),\frac{1}{f-a_l}\right)}{t}dt\\
 &\geq \int_{|z_{2,l,m}|}^{2|z_{2,l,m}|}\frac{n\left(t,{\Omega}
\left(\alpha_{j,2\varepsilon},\beta_{j,2\varepsilon}\right),\frac{1}{f-a_l}\right)}{t}dt\\
&\geq n\left(|z_{2,l,m}|,{\Omega}\left(\alpha_{j,2\varepsilon},\beta_{j,2\varepsilon}\right),\frac{1}{f-a_l}\right)\int_{|z_{2,l,m}|}^{2|z_{2,l,m}|}\frac{1}{t}dt\\
&=n\left(|z_{2,l,m}|,{\Omega}\left(\alpha_{j,2\varepsilon},\beta_{j,2\varepsilon}\right),\frac{1}{f-a_l}\right)\log 2\geq n(z_{2,l,m})\log 2\\
 &\geq e^{|p_n|(1-2\varepsilon)|z_{2,l,m}|^n\sin(n\varepsilon)} \log 2
\end{aligned}
\end{equation}
for the positive integer $m$ satisfying $m>N_{2,l,m}(r)$ with $l\in\{1,2,3\}.$ From \eqref{eq3.82}-\eqref{eq3.85} we have
\begin{equation}\label{eq3.86}
\begin{aligned}
&\quad  2\omega_{j,2\varepsilon}
\sin\left(2\varepsilon\omega_{j,2\varepsilon}\right)\frac{e^{|p_n|(1-2\varepsilon)|z_{2,l,m}|^n\sin(n\varepsilon)} \log 2}{(2|z_{2,l,m}|)^{\omega_{j,2\varepsilon}}} \\
&\leq 2\omega_{j,2\varepsilon} \sin\left(2\varepsilon\omega_{j,2\varepsilon}\right)\frac{N_{0}(2|z_{2,l,m}|)}{(2|z_{2,l,m}|)^{\omega_{j,2\varepsilon}}}\leq O(1)
 \end{aligned}
  \end{equation}
 for the positive integer $m$ satisfying $m>N_{2,l,m}(r)$ with $l\in\{1,2,3\}.$ From \eqref{eq3.72}, \eqref{eq3.86} and \eqref{eq3.81} with the even integer $j$ satisfying $0\leq j\leq 2n-1,$ we have
\begin{equation}\label{eq3.87}
O(1)\geq  2\omega_{j,2\varepsilon}\sin\left(2\varepsilon\omega_{j,2\varepsilon}\right)\frac{e^{|p_n|(1-2\varepsilon)|z_{2,l,m}|^n\sin(n\varepsilon)} \log 2}{(2|z_{2,l,m}|)^{\omega_{j,2\varepsilon}}}\rightarrow+\infty, \, \, \text{as} \, \, m\rightarrow\infty
  \end{equation}
 for $l\in\{1,2,3\},$ the even integer $j$ satisfying $0\leq j\leq 2n-1,$ and the given positive number $\varepsilon$ satisfying $0<\varepsilon<\frac{\pi}{8n}.$ This is a contradiction.
\vskip 2mm
\par {\bf Subcase 2.2.} Suppose that $j$ with $0\leq j \leq 2n-1$ and $j\in\Bbb{Z}$ is an odd integer. Then, by \eqref{eq3.12}, \eqref{eq3.67}, \eqref{eq3.72}, \eqref{eq3.73} and \eqref{eq3.14} we deduce that there exist some positive integer $N_{l,j}$ and a large positive number $\mathcal{R}_{l,j}$ satisfying $\mathcal{R}_{l,j}\geq \mathcal{R}$ for $l\in\{1,2,3\}$ and the odd integer $j$ satisfying $0\leq j\leq 2n-1,$ such that \eqref{eq3.74} holds and each point $z_{2,l,m}$ satisfying \eqref{eq3.74} is not a common multiple zero of $f-a_l$ and  $g-a_l$ for $l\in\{1,2,3\},$ and each point $z_{2,l,m}$ satisfying \eqref{eq3.74} satisfies
\begin{equation}\label{eq3.88}
\tilde{n}(z_{2,l,m})=\frac{e^{-\text{Re}(P(z_{2,l,m}))}}{|R(z_{2,l,m})|}\geq e^{|p_n|(1-2\varepsilon)|z_{2,l,m}|^n\sin(n\varepsilon)}\geq 4,
\end{equation}
as $ m\geq N_{l,j}$ for $l\in\{1,2,3\}$ and the odd integer $j$ satisfying $0\leq j\leq 2n-1.$ Here and in what follows, $\tilde{n}(z_{2,l,m})$ denotes the multiplicity of $z_{2,l,m}$ of a zero of $g-a_l$ for $l\in\{1,2,3\},$ the positive integer $m$ satisfying $m\geq N_{l,j},$ and the odd integer $j$ satisfying $0\leq j\leq 2n-1.$ From \eqref{eq3.12}, \eqref{eq3.14}, \eqref{eq3.88} and the assumption that $f$ and $g$ share $0,$ $1$ and $c$ IM, we deduce that $z_{2,l,m}$ is also a simple zero of $f-a_l$ for $l\in\{1,2,3\},$ the positive integer $m$ satisfying $m\geq N_{l,j},$ and the odd integer $j$ satisfying $0\leq j\leq 2n-1.$ Next we use Lemma \ref{lemma2.10} for $\overline{\Omega}\left(\alpha_{j,2\varepsilon},\beta_{j,2\varepsilon}\right)$ to deduce
\begin{equation}\label{eq3.89}
\begin{aligned}
S_{\alpha_{j,2\varepsilon}, \beta_{j,2\varepsilon}}(r,g)&\leq\overline{C}_{\alpha_{j,2\varepsilon}, \beta_{j,2\varepsilon}}\left(r,\frac{1}{g-a_1}\right)+\overline{C}_{\alpha_{j,2\varepsilon}, \beta_{j,2\varepsilon}}\left(r,\frac{1}{g-a_2}\right)\\
&\quad +\overline{C}_{\alpha_{j,2\varepsilon}, \beta_{j,2\varepsilon}}\left(r,\frac{1}{g-a_3}\right)+R_{\alpha_{j,2\varepsilon}\beta_{j,2\varepsilon}}(r,g)
\end{aligned}
\end{equation}
for $r>0$ and the odd integer $j$ satisfying $0\leq j\leq 2n-1,$  From \eqref{eq3.74}, \eqref{eq3.88} and Lemma \ref{lemma2.9} we have
\begin{equation}\label{eq3.90}
\begin{aligned}
&\quad \overline{C}_{\alpha_{j,2\varepsilon}, \beta_{j,2\varepsilon}}\left(r,\frac{1}{g-a_l}\right)\\
&=\left(\overline{C}_{\alpha_{j,2\varepsilon}, \beta_{j,2\varepsilon}}\left(r,\frac{1}{g-a_l}\right)-\overline{C}_{\alpha_{j,2\varepsilon}, \beta_{j,2\varepsilon}}\left(\mathcal{R}_{l,j},\frac{1}{g-a_l}\right)\right) \\
&\quad +\overline{C}_{\alpha_{j,2\varepsilon}, \beta_{j,2\varepsilon}}\left(\mathcal{R}_{l,j},\frac{1}{g-a_l}\right)\leq \frac{1}{4}C_{\alpha_{j,2\varepsilon}, \beta_{j,2\varepsilon}}\left(r,\frac{1}{g-a_l}\right)+O(1)\\
&\leq \frac{1}{4}S_{\alpha_{j,2\varepsilon}, \beta_{j,2\varepsilon}}\left(r,\frac{1}{g-a_l}\right)+O(1)=\frac{1}{4}S_{\alpha_{j,2\varepsilon} \beta_{j,2\varepsilon}}\left(r,g\right)+O(1)
\end{aligned}
\end{equation}
for the $l\in\{1,2,3\},$ the odd integer $j$ satisfying $0\leq j\leq 2n-1,$ and the positive number $r$ such that $r>\max\limits_{l\in\{1,2,3\}}\max\limits_{0\leq \tilde{l}\leq n-1, \, \tilde{l}\in\Bbb{Z}}\mathcal{R}_{l,2\tilde{l}}\geq\mathcal{R}.$
From \eqref{eq3.2} and Lemma \ref{Lemma2.12} we have
\begin{equation}\label{eq3.91}
R_{\alpha_{j,2\varepsilon}, \beta_{j,2\varepsilon}}(r,g)=O(1)
\end{equation}
for the odd integer $j$ satisfying $0\leq j\leq 2n-1,$ and the positive number $r$ such that $r>\max\limits_{l\in\{1,2,3\}}\max\limits_{0\leq \tilde{l}\leq n-1, \, \tilde{l}\in\Bbb{Z}}\mathcal{R}_{l,2\tilde{l}}\geq \mathcal{R}.$
From \eqref{eq3.89}-\eqref{eq3.91} we deduce
\begin{equation}\label{eq3.92}
S_{\alpha_{j,2\varepsilon}, \beta_{j,2\varepsilon}}(r,g)=O(1)
\end{equation}
for the odd integer $j$ satisfying $0\leq j\leq 2n-1,$ and the positive number $r$ such that $r>\max\limits_{l\in\{1,2,3\}}\max\limits_{0\leq \tilde{l}\leq n-1, \, \tilde{l}\in\Bbb{Z}}\mathcal{R}_{l,2\tilde{l}}\geq\mathcal{R}.$ Next we use Lemma \ref{lemma2.9} and the reasonings in the proof of
Zheng \cite[pp.53-54, Lemma 2.2.2]{Zheng2009} to have
\begin{equation}\label{eq3.93}
\begin{aligned}
&\quad S_{\alpha_{j,2\varepsilon}, \beta_{j,2\varepsilon}}\left(r,g\right)+O(1)=S_{\alpha_{j,2\varepsilon}, \beta_{j,2\varepsilon}}\left(r,\frac{1}{g-a_l}\right)\geq C_{\alpha_{j,2\varepsilon}, \beta_{j,2\varepsilon}}\left(r,\frac{1}{g-a_l}\right)\\
&\geq 2\sin\left(2\varepsilon\omega_{\alpha_{j,2\varepsilon}}\right)\sum_{|z_{2,l,m}|\in \Omega\left(\alpha_{j,2\varepsilon}, \beta_{j,2\varepsilon}\right)}\left(\frac{1}{|z_{2,l,m}|^{\omega_{j,2\varepsilon}}}-\frac{|z_{2,l,m}|^{\omega_{j,2\varepsilon}}}{r^{2\omega_{j,2\varepsilon}}}\right)+O(1)\\
&\geq 2\sin\left(2\varepsilon\omega_{j,2\varepsilon}\right)\sum_{|z_{2,l,m}|>\mathcal{R}_{l,j}}\left(\frac{1}{|z_{2,l,m}|^{\omega_{j,2\varepsilon}}}
-\frac{|z_{2,l,m}|^{\omega_{j,2\varepsilon}}}{r^{2\omega_{j,2\varepsilon}}}\right)+O(1)\\
&=2\sin\left(2\varepsilon\omega_{j,2\varepsilon}\right)\int_{1}^{r}\left(\frac{1}{t^{\omega_{j,2\varepsilon}}}-\frac{t^{\omega_{j,2\varepsilon}}}
{r^{2\omega_{j,2\varepsilon}}}\right)d\tilde{n}_{0}(t)+O(1)\\
&=2\omega_{j,2\varepsilon}\sin\left(2\varepsilon\omega_{j,2\varepsilon}\right)\int_{1}^{r}\tilde{n}_{0}(t)\left(\frac{1}{t^{1+\omega_{j,2\varepsilon}}}
+\frac{t^{\omega_{j,2\varepsilon}}-1}{r^{2\omega_{j,2\varepsilon}}}\right)dt+O(1)\\
&\geq 2\omega_{j,2\varepsilon} \sin\left(2\varepsilon\omega_{j,2\varepsilon}\right)\int_{1}^{r} \frac{1}{t^{\omega_{j,2\varepsilon}}}d\tilde{N}_{0}(t)+O(1)\\
&=2\omega_{j,2\varepsilon} \sin\left(2\varepsilon\omega_{j,2\varepsilon}\right)\frac{\tilde{N}_{0}(r)}{r^{\omega_{j,2\varepsilon}}}+2\omega^2_{j,2\varepsilon} \sin\left(2\varepsilon\omega_{j,2\varepsilon}\right)\int_{1}^{r} \frac{\tilde{N}_{0}(t)}{t^{1+\omega_{j,2\varepsilon}}}dt+O(1)\\
&\geq 2\omega_{j,2\varepsilon}  \sin\left(2\varepsilon\omega_{j,2\varepsilon}\right)\frac{\tilde{N}_{0}(r)}{r^{\omega_{j,2\varepsilon}}}+O(1)
\end{aligned}
\end{equation}
for $l\in\{1,2,3\},$ the odd integer $j$ satisfying $0\leq j\leq 2n-1,$ and the positive number $r$ such that $r>\max\limits_{l\in\{1,2,3\}}\max\limits_{0\leq \tilde{l}\leq n-1, \, \tilde{l}\in\Bbb{Z}}\mathcal{R}_{l,2\tilde{l}}\geq\mathcal{R}.$ Here and in what follows, $\tilde{N}_{0}(t)$ is defined as $\tilde{N}_{0}(t)=\int_{1}^{t}\frac{n\left(u,\Omega\left(\alpha_{j,2\varepsilon},\beta_{j,2\varepsilon}\right),\frac{1}{g-a_l}\right)}{u}du$ with  $n\left(u,\Omega\left(\alpha_{j,2\varepsilon},\beta_{j,2\varepsilon}\right),\frac{1}{g-a_l}\right)$ being the number of zeros of $g-a_l$ for $l\in\{1,2,3\}$ in $\Omega\left(\alpha_{j,2\varepsilon},\beta_{j,2\varepsilon}\right)\cap\{z:1<|z|\leq  u\}$ counted with multiplicities. From \eqref{eq3.92} and \eqref{eq3.93} we have
\begin{equation}\label{eq3.94}
 2\omega_{j,2\varepsilon}  \sin\left(2\varepsilon\omega_{j,2\varepsilon}\right)\frac{\tilde{N}_{0}(r)}{r^{\omega_{j,2\varepsilon}}}\leq O(1)
  \end{equation}
 for the  odd integer $j$ satisfying $0\leq j\leq 2n-1,$ and the positive number $r$ such that $r>\max\limits_{l\in\{1,2,3\}}\max\limits_{0\leq \tilde{l}\leq n-1, \, \tilde{l}\in\Bbb{Z}}\mathcal{R}_{l,2\tilde{l}}\geq\mathcal{R}.$ Now we use \eqref{eq3.74} and \eqref{eq3.88} to deduce $|z_{2,l,m}|>\max\limits_{l\in\{1,2,3\}}\max\limits_{0\leq \tilde{l}\leq n-1, \, \tilde{l}\in\Bbb{Z}}\geq \mathcal{R}$ for the positive integer $m$ satisfying $m>\max\limits_{l\in\{1,2,3\}}\max\limits_{0\leq \tilde{l}\leq n-1, \, \tilde{l}\in\Bbb{Z}}N_{l,2\tilde{l}}.$ Combining this with \eqref{eq3.72}, we deduce that for the positive number $r$ satisfying $r>\max\limits_{l\in\{1,2,3\}}\max\limits_{0\leq \tilde{l}\leq n-1, \, \tilde{l}\in\Bbb{Z}}\geq \mathcal{R},$ there exists some large positive integer $N_{2,l,m}(r)$ with $l\in\{1,2,3\}$ that satisfies \eqref{eq3.83}, such that for the positive integer $m$ satisfying $m>N_{2,l,m}(r)$ with $l\in\{1,2,3\}$ we have \eqref{eq3.84}. From \eqref{eq3.84} we have
\begin{equation}\label{eq3.95}
\begin{aligned}
 \tilde{N}_{0}(2|z_{2,l,m}|)&=\int_{1}^{2|z_{2,l,m}|}\frac{n\left(t,{\Omega}\left(\alpha_{j,2\varepsilon},\beta_{j,2\varepsilon}\right),\frac{1}{g-a_l}\right)}{t}dt\\
\end{aligned}
\end{equation}
\begin{equation*}
\begin{aligned}
 &\geq \int_{|z_{2,l,m}|}^{2|z_{2,l,m}|}\frac{n\left(t,{\Omega}
\left(\alpha_{j,2\varepsilon},\beta_{j,2\varepsilon}\right),\frac{1}{g-a_l}\right)}{t}dt\\
&\geq n\left(|z_{2,l,m}|,{\Omega}\left(\alpha_{j,2\varepsilon},\beta_{j,2\varepsilon}\right),\frac{1}{g-a_l}\right)\int_{|z_{2,l,m}|}^{2|z_{2,l,m}|}\frac{1}{t}dt\\
&=n\left(|z_{2,l,m}|,{\Omega}\left(\alpha_{j,2\varepsilon},\beta_{j,2\varepsilon}\right),\frac{1}{g-a_l}\right)\log 2\geq \tilde{n}(z_{2,l,m})\log 2\\
 &\geq e^{|p_n|(1-2\varepsilon)|z_{2,l,m}|^n\sin(n\varepsilon)} \log 2
\end{aligned}
\end{equation*}
for $l\in\{1,2,3\}$ and the positive integer $m$ satisfying $m>N_{2,l,m}(r).$ From \eqref{eq3.83}, \eqref{eq3.84}, \eqref{eq3.94}, \eqref{eq3.95} we have
\begin{equation}\label{eq3.96}
\begin{aligned}
&\quad  2\omega_{j,2\varepsilon}
\sin\left(2\varepsilon\omega_{j,2\varepsilon}\right)\frac{e^{|p_n|(1-2\varepsilon)|z_{2,l,m}|^n\sin(n\varepsilon)} \log 2}{(2|z_{2,l,m}|)^{\omega_{j,2\varepsilon}}} \\
&\leq 2\omega_{j,2\varepsilon} \sin\left(2\varepsilon\omega_{j,2\varepsilon}\right)\frac{\tilde{N}_{0}(2|z_{2,l,m}|)}{(2|z_{2,l,m}|)^{\omega_{j,2\varepsilon}}}\leq O(1)
 \end{aligned}
  \end{equation}
 for the odd integer $j$ satisfying $0\leq j\leq 2n-1,$ and the positive integer $m$ satisfying $m>N_{2,l,m}(r)$ with $l\in\{1,2,3\}.$ From \eqref{eq3.72}, \eqref{eq3.96} and \eqref{eq3.81} with the odd integer $j$ satisfying $0\leq j\leq 2n-1,$ we deduce \eqref{eq3.87} for $l\in\{1,2,3\},$ the odd integer $j$ satisfying $0\leq j\leq 2n-1,$ and the given positive number $\varepsilon$ satisfying $0<\varepsilon<\frac{\pi}{8n}.$ From \eqref{eq3.87} we get a contradiction. This proves Claim 2I.
\vskip 2mm
\par Next we use Lemma \ref{Lemma2.16}, Lemma \ref{Lemma2.18} and Claim 2I to complete the proof of Theorem \ref{Theorem1.1}. For this purpose, we have a discussion as follows: suppose that $\limsup\limits_{r\rightarrow\infty}N(r,f)/T(r,f)>\frac{4}{5}.$ Then, it follows from Lemma 2.2 in \cite{Ishizaki1997} that there exist some real constant $\lambda >4/5$ and some set $I \subset(0, +\infty)$ that has infinite linear measure such that $N\left(r,f\right)/T(r,f)\geq \lambda$ for all $r \in I.$ Combining this with Theorem \ref{TheoremE}, we see that $f$ and $g$ share all four values $a_l,$ $a_2,$ $a_3,$
$a_4$ CM, and so  Theorem \ref{Theorem1.1} is valid. Next we suppose that $\limsup\limits_{r\rightarrow\infty}\frac{N(r,f)}{T(r,f)}\leq \frac{4}{5}.$ Then, it follows from the known equality $T(r,f)=m(r,f)+N(r,f)$ that
\begin{equation}\label{eq3.97}
\delta=:\delta(\infty,f)=\liminf\limits_{r\rightarrow\infty}\frac{m(r,f)}{T(r,f)}\geq \frac{1}{5}.
\end{equation}
Next we define the following directional arc:
\begin{equation}\label{eq3.98}
\gamma_j=\left\{z\in\Bbb{C}: z=re^{i\theta} \, \, \text{with} \, \, \theta\in  [\alpha_j, \beta_j] \, \, \text{from} \, \,  \alpha_j \, \, \text {to} \, \,  \beta_j\right\}
\end{equation}
with
\begin{equation}\label{eq3.99}
\alpha_j=-\frac{\theta_{n}}{n}+(2j-1)\frac{\pi}{2n} \, \, \text{and} \, \, \beta_j=-\frac{\theta_{n}}{n}+(2j+1)\frac{\pi}{2n}
\end{equation}
for each integer $j$ satisfying $0 \leq j\leq 2n-1.$ From \eqref{eq3.98} and \eqref{eq3.99} we have
\begin{equation}\label{eq3.100}
\begin{aligned}
m(r,f)&=\frac{1}{2\pi}\int_0^{2\pi}\log^{+}|f(re^{i\theta})|d\theta
=\frac{1}{2\pi}\sum\limits_{j=0}^{2n-1}\int_{\alpha_j}^{\beta_j}\log^{+}|f(re^{i\theta})|d\theta \quad \text{for} \quad  r>0.
\end{aligned}
\end{equation}
On the other hand, from \eqref{eq3.2}, \eqref{eq3.97} and Lemma \ref{Lemma2.15} we have
\begin{equation}\label{eq3.101}
\liminf\limits_{r_n\rightarrow\infty}\text{mes}D(r_n, \infty)\geq \min\left\{2\pi,\frac{4}{\rho}\arcsin\sqrt{\frac{\delta}{2}}\right\}.
\end{equation}
From \eqref{eq3.97}-\eqref{eq3.101} we see that there exists some integer $j$ satisfying $0\leq j\leq 2n-1,$ say $j=0,$ such that
\begin{equation}\label{eq3.102}
\text{mes}(D(r_n, \infty)\cap [\alpha_0,\beta_0])\geq \sigma_0
\end{equation}
with
\begin{equation}\label{eq3.103}
D(r_n, \infty)=\left\{\theta\in [-\pi,\pi): \log^{+}|f(r_ne^{i\theta})|>\frac{T(r_n,f)}{\log r_n}\right\}
\end{equation}
for arbitrary P\'{o}1ya
peaks $\{r_n\} \subset (0, +\infty)\setminus F $ of order $\sigma=\mu=\rho\geq 1,$ as $r_n\rightarrow\infty.$ Here and in what follows, $\sigma_0$ is some positive constant, and $F\subset (0, +\infty)$ is some set of finite logarithmic measure. Then, it follows from \eqref{eq3.102} and \eqref{eq3.103} that
\begin{equation}\label{eq3.104}
\begin{split}
&\quad \frac{1}{2\pi}\int_{\alpha_0}^{\beta_0}\log^{+}|f(r_ne^{i\theta})|d\theta
\geq \frac{1}{2\pi}\int_{[\alpha_0,\beta_0]\cap D(r_n, \infty) }\log^{+}|f(r_ne^{i\theta})|d\theta\\
&\geq \frac{\sigma_0T(r_n,f)}{2\pi\log r_n} \, \, \text{for} \, \,  \{r_n\} \subset (0, +\infty)\setminus F \, \, \text{and} \, \,  r_n\geq \mathcal{R}_0,
\end{split}
\end{equation}
where and in what follows, $\mathcal{R}_0$ is a large positive number satisfying
 \begin{equation}\nonumber
 \mathcal{R}_0\geq \max\limits_{l\in\{1,2,3\}}\max\limits_{0\leq \tilde{l}\leq n-1, \, \tilde{l}\in\Bbb{Z}}\mathcal{R}_{l,2\tilde{l}}.
 \end{equation}
Next we set the simply connected domain
\begin{equation}\label{eq3.105}
D_{0,2\varepsilon}=\{z\in\Bbb{C}:\alpha_0+2\varepsilon<\arg z<\beta_0-2\varepsilon \, \, \text{with} \, \, 0<|z|<r\}
\end{equation}
 with the positive boundary $\Gamma_{0,2\varepsilon}=:L_{0,2\varepsilon}+\gamma_{0,2\varepsilon}+L^{-}_{1,2\varepsilon},$ where $L_{0,2\varepsilon}$ and $L_{1,2\varepsilon}$ are two oriented segments such that
\begin{equation}\label{eq3.106}
L_{0,2\varepsilon}: z=te^{ i\left(-\frac{\theta_n}{n}-\frac{\pi}{2n}+2\varepsilon\right)} \, \, \text{with} \, \, t\in[0,r] \, \, \text{from} \, \, t=0 \, \, \text{to} \, \,  t=r
\end{equation}
and
\begin{equation}\label{eq3.107}
L_{1,2\varepsilon}:  z=te^{ i\left(-\frac{\theta_n}{n}+\frac{\pi}{2n}-2\varepsilon\right)}  \, \, \text{with} \, \, t\in[0,r] \, \, \text{from} \, \,  t=0 \, \, \text{to} \, \,  t=r.
\end{equation}
respectively, and $\gamma_{0,2\varepsilon}$ is a directional arc such that
\begin{equation}\label{eq3.108}
\gamma_{0,2\varepsilon}: z=re^{i\theta} \, \, \text{with} \, \, \theta\in  [\alpha_{0,2\varepsilon}, \beta_{0,2\varepsilon}] \, \, \text{from}
 \, \,  \alpha_{0,2\varepsilon} \, \, \text {to} \, \,  \beta_{0,2\varepsilon},
\end{equation}
where and in what follows, $\alpha_{0,2\varepsilon}$ and $\beta_{0,2\varepsilon}$ are defined as \eqref{eq3.71} for $j=0.$ Next we also define $G_{D_{0,2\varepsilon}}(z,a_{0,r})$ with $z\in D_{0,2\varepsilon}$ as
\begin{equation}\label{eq3.109}
 G_{D_{0,2\varepsilon}}(z,a_{0,r})=\log\frac{1}{|z-a_{0,r}|}+\omega_{D_{0,2\varepsilon}}(z,a_{0,r})
 \end {equation}
for $z\in D_{0,2\varepsilon}.$ Here and in what follows, $\varepsilon$ is a small positive number satisfying $0<\varepsilon<\frac{\pi}{8n},$ and
\begin{equation}\label{eq3.110}
a_{0,r}\in \left\{z\in D_{0,2\varepsilon}: |z|=\frac{r}{2} \, \, \text{and} \, \, \alpha_{0,2\varepsilon}+\varepsilon\leq \arg z\leq \beta_{0,2\varepsilon}-\varepsilon\right\}
\end{equation}
 is a complex number such that $f(a_{0,r})\not\in\{a_1,a_2,a_3,\infty\}.$ Then, from \eqref{eq3.105}-\eqref{eq3.110}, Remark \ref{Remark2.5X2}.3, the positive boundary $\Gamma_{0,2\varepsilon}=L_{0,2\varepsilon}+\gamma_{0,2\varepsilon}+L^{-}_{1,2\varepsilon}$ of the region $D_{0,2\varepsilon}$ and the definition of the proximity function $m(D_{0,2\varepsilon},a_{0,r},f)$ of $f$ with the center at $a_{0,r}$ for $D_{0,2\varepsilon}$ we have
\begin{equation}\label{eq3.111}
\begin{aligned}
&\quad m(D_{0,2\varepsilon},a_{0,r},f)=\frac{1}{2\pi}\int_{\Gamma_{0,2\varepsilon}}\log^{+}|f(\zeta)|\frac{\partial G_{D_{0,2\varepsilon}}(\zeta,a_{0,r})}{\partial \mathfrak{n}}ds\\
&=\frac{1}{2\pi}\int_{L^{-1}_{1,2\varepsilon}+L_{0,2\varepsilon}+\gamma_{0,2\varepsilon}}\log^{+}|f(\zeta)|\frac{\partial G_{D_{0,2\varepsilon}}(\zeta,a_{0,r})}{\partial \mathfrak{n}}ds\\
&\geq \frac{1}{2\pi}\int_{\gamma_{0,2\varepsilon}}\log^{+}|f(\zeta)|\frac{\partial G_{D_{0,2\varepsilon}}(\zeta,a_{0,r})}{\partial \mathfrak{n}}ds\\
\end{aligned}
\end{equation}
for the small positive number $\varepsilon$ satisfying $0<\varepsilon<\frac{\pi}{8n}.$ Next we also set the following transformation
\begin{equation}\label{eq3.112}
z=e^{i\alpha_{0,2\varepsilon}}{\tilde{z}}^{\frac{\beta_{0,2\varepsilon}-\alpha_{0,2\varepsilon}}{\pi}} \, \, \text{with} \, \, \tilde{z}\in \tilde{D}_0,
\end{equation}
where and in what follows, $\tilde{D}_0$ is a simply connected domain such that
\begin{equation}\label{eq3.113}
\tilde{D}_0=:\{\tilde{z}\in \Bbb{C}: 0<|\tilde{z}|<\tilde{R} \, \, \text{and} \, \, 0<\arg \tilde{z}<\pi\},
\end{equation}
where $\tilde{R}$ is a positive number satisfying
\begin{equation}\label{eq3.114}
\tilde{R}=r^{\frac{\pi}{\beta_{0,2\varepsilon}-\alpha_{0,2\varepsilon}}},
\end{equation}
and the positive boundary of the simply connected domain $\tilde{D}_0$ is $\tilde{\Gamma}_0=\tilde{L}_0+\tilde{\gamma}_0,$
where and in what follows, $\tilde{L}_0$ is an oriented segment, and $\tilde{\gamma}_0$ is a directional arc. They are defined as follows: the oriented segment $\tilde{L}_0$ is such an oriented segment that for each $\tilde{\zeta}\in\tilde{L}_0$ we have $\tilde{\zeta}=t$ with $t\in [-\tilde{R},\tilde{R}],$ and the starting point of $\tilde{L}_0$ is $\tilde{\zeta}=-\tilde{R},$ while the end point of $\tilde{L}_0$ is $\tilde{\zeta}=\tilde{R}.$
 The directional arc $\tilde{\gamma}_0$ is such a directional arc that for each $\tilde{\zeta}\in\tilde{\gamma}_0$ we have
$\tilde{\zeta}=\tilde{R}e^{i\tilde{\theta}}$ with $\tilde{\theta}\in [0,2\pi],$ and the starting point of $\tilde{\gamma}_0$ is $\tilde{\zeta}=\tilde{R}$ with $\arg \tilde{\zeta}=0,$ while the end point of $\tilde{\gamma}_0$ is $\tilde{\zeta}=\tilde{R}$ with $\arg \tilde{\zeta}=2\pi.$ In addition, from \eqref{eq3.71} with $j=0$ we have
\begin{equation}\label{eq3.115}
0<\frac{\beta_{0,2\varepsilon}-\alpha_{0,2\varepsilon}}{\pi}<\frac{1}{n}.
\end{equation}
\vskip 2mm
\par From \eqref{eq3.112} and \eqref{eq3.115} we can easily verify that the transformation of \eqref{eq3.112} maps the simply connected domain $\tilde{D}_0$ of \eqref{eq3.113} onto the simply connected domain $D_{0,2\varepsilon}$ of \eqref{eq3.105} in a one-to-one mapping manner, and maps the positive boundary $\tilde{\Gamma}_0=\tilde{L}_0+\tilde{\gamma}_0$ of $\tilde{D}_0$ into the positive boundary $\Gamma_{0,2\varepsilon}=L^{-}_{1,2\varepsilon}+L_{0,2\varepsilon}+\gamma_{0,2\varepsilon}$ of $D_{0,2\varepsilon}$ in a one-to-one mapping manner,
where the oriented segment $\tilde{L}_0$ and the directional arc $\tilde{\gamma}_0$ are mapped into $L^{-1}_{1,2\varepsilon}+L_{0,2\varepsilon}$ and $\gamma_{0,2\varepsilon}$ respectively by the transformation of \eqref{eq3.112} in a one-to-one mapping manner. Moreover, from \eqref{eq3.113}-\eqref{eq3.115}, the polar form of the Cauchy-Riemann partial differential equations and the polar form of the  sufficient conditions for an analytic function (cf.\cite{Silverman2000}), we can verify that the transformation of \eqref{eq3.112} is also an analytic transformation in the simply connected domain $\tilde{D}_0$ of \eqref{eq3.113}. Combining this with the known result that the transformation of \eqref{eq3.112} maps the simply connected domain $\tilde{D}_0$ of \eqref{eq3.113} onto the simply connected domain $D_{0,2\varepsilon}$ of \eqref{eq3.105} in a one-to-one mapping manner, we see that the univalent transformation \eqref{eq3.112} on the simply connected domain $\tilde{D}_0$ is a univalent and analytic map on the simply connected domain $\tilde{D}_0,$ and maps conformally the simply connected domain $\tilde{D}_0$ of \eqref{eq3.113} with the positive boundary $\tilde{\Gamma}_0=\tilde{L}_0+\tilde{\gamma}_0$ onto the simply connected domain $D_{0,2\varepsilon}$ of \eqref{eq3.105} with the positive boundary $\Gamma_{0,2\varepsilon}=L^{-}_{1,2\varepsilon}+L_{0,2\varepsilon}+\gamma_{0,2\varepsilon}.$ For convinience of the following discussion, we denote
 the univalent and analytic transformation \eqref{eq3.112} as
 \begin{equation}\label{eq3.116}
z=\tilde{\phi}_{\tilde{a}_{0,\tilde{r}}}(\tilde{z})=:e^{i\alpha_{0,2\varepsilon}}{\tilde{z}}^{\frac{\beta_{0,2\varepsilon}-\alpha_{0,2\varepsilon}}{\pi}}
 \, \, \text{with} \, \, \tilde{z} \in \tilde{D}_0,
\end{equation}
and there exists a unique point $\tilde{a}_{0,\tilde{r}}\in\tilde{D}_0$ satisfying
\begin{equation}\label{eq3.117}
a_{0,r}=\tilde{\phi}_{\tilde{a}_{0,\tilde{r}}}(\tilde{a}_{0,\tilde{r}}),
\end{equation}
such that the univalent and analytic mapping \eqref{eq3.116} maps conformally the simply connected domain $\tilde{D}_0$ of \eqref{eq3.113} with the positive boundary $\tilde{\Gamma}_0=\tilde{L}_0+\tilde{\gamma}_0$  onto the simply connected domain $D_{0,2\varepsilon}$ of \eqref{eq3.105} with the positive boundary $\Gamma_{0,2\varepsilon}=L^{-}_{1,2\varepsilon}+L_{0,2\varepsilon}+\gamma_{0,2\varepsilon}.$
On the other hand, from the Riemann mapping theorem (cf.\cite[p.230,Theorem 1]{Ahlfors1979}) we also see that there exists a unique univalent and analytic function on the simply connected domain $D_{0,2\varepsilon}$ of \eqref{eq3.105}, say
\begin{equation} \label{eq3.118}
w=\phi_{a_{0,r}} (z)\, \, \text {for} \, \, z\in D_{0,2\varepsilon}  \, \, \text{with} \, \,  \phi_{a_{0,r}}(a_{0,r})=0,
\end{equation}
such that the univalent and analytic transformation \eqref{eq3.118} maps conformally the simply connected domain  $D_{0,2\varepsilon}$ of \eqref{eq3.105} with the positive boundary $\Gamma_{0,2\varepsilon}=L^{-}_{1,2\varepsilon}+L_{0,2\varepsilon}+\gamma_{0,2\varepsilon}$ onto the open unit disk
$\Bbb{D}=\{w\in\Bbb{C}:|w|<1\}$ with the positive boundary $\partial\Bbb{D}: w=e^{i\Psi}$ with $\Psi\in [0,2\pi],$ where the starting point of $\partial\Bbb{D}$ is $w=1$ with $\arg w=0,$ and the end point of $\partial\Bbb{D}$ is $w=1$ with $\arg w=2\pi.$ From \eqref{eq3.116}-\eqref{eq3.118} we deduce that the composition
\begin{equation} \label{eq3.119}
w=\Phi_{\tilde{a}_{0,\tilde{r}}}(\tilde{z})=:\phi_{a_{0,r}}\circ\tilde{\phi}_{\tilde{a}_{0,\tilde{r}}}(\tilde{z})=\phi_{a_{0,r}} \left(e^{i\alpha_{0,2\varepsilon}}{\tilde{z}}^{\frac{\beta_{0,2\varepsilon}-\alpha_{0,2\varepsilon}}{\pi}}\right)
\end{equation}
 of the transformation \eqref{eq3.118} and the transformation \eqref{eq3.116} with $\Phi_{\tilde{a}_{0,\tilde{r}}}\left(\tilde{a}_{0,\tilde{r}}\right)=\phi_{a_{0,r}} (a_{0,r})=0$ for $\tilde{z}\in \tilde{D}_0$
  is a univalent and analytic transformation of $\tilde{D}_{0}$ onto the open unit disk $\Bbb{D}=\{w\in\Bbb{C}:|w|<1\},$ such that it maps the positive boundary $\tilde{\Gamma}_0=\tilde{L}_0+\tilde{\gamma}_0$ of $\tilde{D}_0$ into the positive boundary $\partial\Bbb{D}$ of the open unit disk $\Bbb{D}$ mentioned above in a one-to-one mapping manner. Noting that the transformation of \eqref{eq3.118} is a univalent and analytic mapping of $D_{0,2\varepsilon}$ onto the open unit disk $\Bbb{D}=\{w\in\Bbb{C}:|w|<1\}.$ Combining this with the theorem on the boundary behavior of conformal mappings, we see that the function $\Phi_{\tilde{a}_{0,\tilde{r}}}(\tilde{z})$
\eqref{eq3.119} is continuous on $\tilde{D}_0,$ and $|\Phi_{\tilde{a}_{0,\tilde{r}}}(\tilde{\zeta})|=1$ for $\tilde{\zeta}\in \tilde{\Gamma}_0= \partial\tilde{D}_0.$ Next we use the similar reasoning of proving \eqref{eq2.7} in Remark \ref{Remark2.5X2} to have
 \begin{equation}\label{eq3.120}
 G_{D_{0,2\varepsilon}}(z,a_{0,r})=-\log|\phi_{a_{0,r}} (z)| \, \, \text {for each} \, \, z\in \overline{D}_{0,2\varepsilon} \setminus\{a_{0,r}\},
 \end{equation}
\begin{equation}\label{eq3.121}
\frac{\phi'_{a_{0,r}}(\zeta)}{\phi_{a_{0,r}}(\zeta)}d\zeta=i\frac{\partial G_{D_{0,2\varepsilon}}(\zeta,a_{0,r})}{\partial \mathfrak{n}}ds \, \, \text{for each} \, \, \zeta\in \Gamma_0=\partial D_{0,2\varepsilon},
\end{equation}
\begin{equation}\label{eq3.122}
G_{\tilde{D}_0}\left(\tilde{z},\tilde{a}_{0,\tilde{r}}\right)=-\log|\Phi_{\tilde{a}_{0,\tilde{r}}}(\tilde{z})| \, \, \text {for each} \, \, \tilde{z}\in \overline{D}_0 \setminus\{\tilde{a}_{0,\tilde{r}}\}
 \end{equation}
and
\begin{equation}\label{eq3.123}
\frac{\Phi'_{\tilde{a}_{0,\tilde{r}}}(\tilde{\zeta})}{\Phi_{\tilde{a}_{0,\tilde{r}}}(\tilde{\zeta})}d\tilde{\zeta}
=i\frac{\partial G_{\tilde{D}_0}\left(\tilde{\zeta},\tilde{a}_{0,\tilde{r}}\right)}{\partial \mathfrak{\tilde{n}}}d\tilde{s}
\end{equation}
for each point $\tilde{\zeta}=\tilde{R}e^{i\tilde{\theta}}$ with $\tilde{\theta}=\arg\tilde{\zeta}$ and $0<\tilde{\theta}< \pi$ on the semi-circle
$\partial \tilde{D}_{0}=\tilde{\Gamma}_0=\{\tilde{\zeta}\in\Bbb{C}:\tilde{\zeta}=\tilde{R}e^{i\tilde{\theta}} \, \, \text{with} \, \, 0<\tilde{\theta}<\pi\}.$
\vskip 2mm
\par Since the transformation \eqref{eq3.116} is a univalent and analytic transformation of $D_{0,2\varepsilon}$ onto $\tilde{D}_0$ such that \eqref{eq3.117} holds, we have from \eqref{eq3.116}, \eqref{eq3.118}-\eqref{eq3.120}, \eqref{eq3.122} and Lemma \ref{Lemma2.17} that the Green's functions $G_{\tilde{D}_0}(\tilde{z},\tilde{a}_{0,\tilde{r}})$ and $G_{D_{0,2\varepsilon}}(z,a_{0,r})$ for $\tilde{D}_0$ and $D_{0,2\varepsilon}$ with singularities $\tilde{a}_{0,\tilde{r}}$ and $a_{0,r}$ respectively satisfy the relation
\begin{equation}\label{eq3.124}
\begin{aligned}
G_{\tilde{D}_0}\left(\tilde{z},\tilde{a}_{0,\tilde{r}}\right)&=G_{D_{0,2\varepsilon}}\left(\tilde{\phi}_{\tilde{a}_{0,\tilde{r}}}(\tilde{z}), a_{0,r}\right) =G_{D_{0,2\varepsilon}}\left(e^{i\alpha_{0,2\varepsilon}}{\tilde{z}}^{\frac{\beta_{0,2\varepsilon}-\alpha_{0,2\varepsilon}}{\pi}}, a_{0,r}\right)\\
&=G_{D_{0,2\varepsilon}}\left(z, a_{0,r}\right)=-\log\left|\phi_{a_{0,r}} \left(e^{i\alpha_{0,2\varepsilon}}{\tilde{z}}^{\frac{\beta_{0,2\varepsilon}-\alpha_{0,2\varepsilon}}{\pi}}\right)\right|
\end{aligned}
\end{equation}
for each $\tilde{z} \in \overline{\tilde{D}}_0\setminus\{\tilde{a}_{0,\tilde{r}}\}.$
 Next we use \eqref{eq3.117}, \eqref{eq3.119}-\eqref{eq3.124} and the invariance of first order differential form to deduce
 \begin{equation}\label{eq3.125}
\frac{\partial G_{D_{0,2\varepsilon}}(\zeta,a_{0,r})}{\partial \mathfrak{n}}ds=\frac{\partial G_{\tilde{D}_0}\left(\tilde{\zeta},\tilde{a}_{0,\tilde{r}}\right)}{\partial \mathfrak{\tilde{n}}}d\tilde{s}
\end{equation}
with $\zeta=e^{i\alpha_{0,2\varepsilon}}{\tilde{\zeta}}^{\frac{\beta_{0,2\varepsilon}-\alpha_{0,2\varepsilon}}{\pi}}$ for each point $\tilde{\zeta}=\tilde{R}e^{i\tilde{\theta}}$ with $\tilde{\theta}=\arg\tilde{\zeta}$ and $0<\tilde{\theta}< \pi$ on the semi-circle
$\partial \tilde{D}_{0}=\tilde{\Gamma}_0=\{\tilde{\zeta}\in\Bbb{C}:\tilde{\zeta}=\tilde{R}e^{i\tilde{\theta}} \, \, \text{with} \, \, 0<\tilde{\theta}<\pi\}.$
From \eqref{eq3.110}, \eqref{eq3.113}, \eqref{eq3.114}, \eqref{eq3.123}, \eqref{eq2.12} in Remark \ref{Remark2.5X2}.1 and Goldberg-Ostrovskii \cite[p.306]{Goldberg1970} we have
\begin{equation}\label{eq3.126}
\begin{aligned}
&\quad \frac{\partial G_{\tilde{D}_0}\left(\tilde{\zeta},\tilde{a}_{0,\tilde{r}}\right)}{\partial \mathfrak{\tilde{n}}}d\tilde{s}=-i\left(\log \Phi_{\tilde{a}_{0,\tilde{r}}}(\tilde{\zeta})\right)'d\tilde{s}=\\
&-i\left(\log\frac{\tilde{R}^2-\overline{\tilde{a}_{0,\tilde{r}}}\tilde{\zeta}}{\tilde{R}(\tilde{\zeta}-\tilde{a}_{0,\tilde{r}})}
 \cdot\frac{\tilde{R}(\tilde{\zeta}-\overline{\tilde{a}_{0,\tilde{r}}})}{\tilde{R}^2-\tilde{a}_{0,\tilde{r}}\tilde{\zeta}} \right)'d\tilde{s}
  =\left(\frac{\tilde{R}^2-|\tilde{a}_{0,\tilde{r}}|^2}{|\tilde{\zeta}-\tilde{a}_{0,\tilde{r}}|^2}-\frac{\tilde{R}^2-|\tilde{a}_{0,\tilde{r}}|^2}{|\tilde{\zeta}-\overline{a}|^2}\right)d\tilde{\theta}\\
 &=\frac{4(\tilde{R}^2-|\tilde{a}_{0,\tilde{r}}|^2)\tilde{R}|\tilde{a}_{0,\tilde{r}}|\sin\tilde{\theta}_0 \sin\tilde{\theta}}{(\tilde{R}^2+|\tilde{a}_{0,\tilde{r}}|^2-2\tilde{R}|\tilde{a}_{0,\tilde{r}}|
 \cos(\tilde{\theta}_0-\tilde{\theta}))((\tilde{R}^2+|\tilde{a}_{0,\tilde{r}}|^2 -2\tilde{R}|\tilde{a}_{0,\tilde{r}}|\cos(\tilde{\theta}_0+\tilde{\theta}))}d\tilde{\theta}\\
&\geq \frac{4(\tilde{R}^2-\tilde{r}^2)\tilde{R}\tilde{r}\sin\tilde{\theta}}{(\tilde{R}+\tilde{r})^4}\left(\sin\frac{\pi \varepsilon}{\beta_{0,2\varepsilon}-\alpha_{0,2\varepsilon}}\right)d\tilde{\theta}
 \end{aligned}
  \end{equation}
for each point $\tilde{\zeta}=\tilde{R}e^{i\tilde{\theta}}$ with $\tilde{\theta}=\arg\tilde{\zeta}$ and $0<\tilde{\theta}< \pi$ on the semi-circle
$\tilde{\Gamma}_0=\{\tilde{\zeta}\in\Bbb{C}:\tilde{\zeta}=\tilde{R}e^{i\tilde{\theta}} \, \, \text{with} \, \, 0<\tilde{\theta}<\pi\}$
 and the point
\begin{equation}\label{eq3.127}
\tilde{a}_{0,\tilde{r}}=\tilde{r}e^{i\tilde{\theta}_0 } \, \, \text{with} \, \, \tilde{\theta}_0=\arg \tilde{a}_{0,\tilde{r}} \, \, \text{and} \, \,
\, \, \frac{\pi \varepsilon}{\beta_{0,2\varepsilon}-\alpha_{0,2\varepsilon}}\leq \tilde{\theta}_0\leq \pi-\frac{\pi \varepsilon}{\beta_{0,2\varepsilon}-\alpha_{0,2\varepsilon}}.
\end{equation}
Here and in what follows
\begin{equation}\label{eq3.128}
\tilde{r}=\left(\frac{r}{2}\right)^{\frac{\pi}{\beta_{0,2\varepsilon}-\alpha_{0,2\varepsilon}}}
=\left(\frac{1}{2}\right)^{\frac{\pi}{\beta_{0,2\varepsilon}-\alpha_{0,2\varepsilon}}}\tilde{R} \quad  \text{for} \, \, r\in (0, +\infty)\setminus  F.
\end{equation}
From \eqref{eq3.104}, \eqref{eq3.111}, \eqref{eq3.117}, \eqref{eq3.125}, \eqref{eq3.126}-\eqref{eq3.128}, the supposition $f(a_{0,r})\not\in\{a_1,a_2,a_3,\infty\},$ the definitions of the proximity function and the Nevanlinna characteristic function of $f$ with the center at $a_{0,r}$ for $D_{0,2\varepsilon}$ for $r\in (0, +\infty)\setminus  F$ and $r\geq \mathcal{R}_0$ we deduce
\begin{equation}\nonumber
\begin{aligned}
  &\quad T(D_{0,2\varepsilon},a_{0,r},f)\geq m(D_{0,2\varepsilon},a_{0,r},f)\geq \frac{1}{2\pi}\int_{\gamma_{0,2\varepsilon}}\log^{+}|f(\zeta)|\frac{\partial G_{D_{0,2\varepsilon}}(\zeta,a_{0,r})}{\partial \mathfrak{n}}ds\\
&= \frac{1}{2\pi}\int_{\tilde{\gamma}_0}\log^{+}|f\left(\tilde{\phi}_{\tilde{a}_{0,\tilde{r}}}(\tilde{\zeta})\right)|\frac{\partial G_{\tilde{D}_0}(\tilde{\zeta},\tilde{a}_{0,\tilde{r}})}{\partial \mathfrak{\tilde{n}}}d\tilde{s}\\
&\geq \frac{1}{2\pi}\frac{4(\tilde{R}^2-\tilde{r}^2)\tilde{R}\tilde{r}}{(\tilde{R}+\tilde{r})^4}\left(\sin\frac{\pi \varepsilon}{\beta_{0,2\varepsilon}-\alpha_{0,2\varepsilon}}\right)
\int_{0}^{\pi}
 \log^{+}\left|f\left(\tilde{\phi}_{\tilde{a}_{0,\tilde{r}}}\left(\tilde{R}e^{i\tilde{\theta}}\right)\right)\right|\sin\tilde{\theta}d\tilde{\theta}\\
 &\geq \frac{1}{2\pi}\frac{4(\tilde{R}^2-\tilde{r}^2)\tilde{R}\tilde{r}}{(\tilde{R}+\tilde{r})^4}\left(\sin\frac{\pi \varepsilon}{\beta_{0,2\varepsilon}-\alpha_{0,2\varepsilon}}\right)
\int_{\varepsilon}^{\pi-\varepsilon}
 \log^{+}\left|f\left(\tilde{\phi}_{\tilde{a}_{0,\tilde{r}}}\left(\tilde{R}e^{i\tilde{\theta}}\right)\right)\right|\sin\tilde{\theta}d\tilde{\theta}\\
 \end{aligned}
\end{equation}
\begin{equation*}
\begin{aligned}
  &\geq \frac{1}{2\pi}\frac{4(\tilde{R}^2-\tilde{r}^2)\tilde{R}\tilde{r}\sin\varepsilon}{(\tilde{R}+\tilde{r})^4}\left(\sin\frac{\pi \varepsilon}{\beta_{0,2\varepsilon}-\alpha_{0,2\varepsilon}}\right)
\int_{\varepsilon}^{\pi-\varepsilon}
 \log^{+}\left|f\left(\tilde{\phi}_{\tilde{a}_{0,\tilde{r}}}\left(\tilde{R}e^{i\tilde{\theta}}\right)\right)\right|d\tilde{\theta}\\
 &= \frac{1}{2\pi}\frac{4(\tilde{R}^2-\tilde{r}^2)\tilde{R}\tilde{r}\sin\varepsilon}{(\tilde{R}+\tilde{r})^4}\left(\sin\frac{\pi \varepsilon}{\beta_{0,2\varepsilon}-\alpha_{0,2\varepsilon}}\right)\int_{\alpha_{0,2\varepsilon}+\varepsilon_0(\varepsilon)}
 ^{\beta_{0,2\varepsilon}-\varepsilon_0(\varepsilon)}
 \log^{+}\left|f\left(re^{i\theta}\right)\right|d\theta,\\
\end{aligned}
\end{equation*}
i.e.,
\begin{equation} \label{eq3.129}
\begin{aligned}
 &\quad T(D_{0,2\varepsilon},a_{0,r},f)\\
   &\geq \frac{1}{2\pi}\frac{4(\tilde{R}^2-\tilde{r}^2)\tilde{R}\tilde{r}\sin\varepsilon}{(\tilde{R}+\tilde{r})^4}\left(\sin\frac{\pi \varepsilon}{\beta_{0,2\varepsilon}-\alpha_{0,2\varepsilon}}\right)\int_{\alpha_{0,2\varepsilon}+\varepsilon_0(\varepsilon)}
 ^{\beta_{0,2\varepsilon}-\varepsilon_0(\varepsilon)}
 \log^{+}\left|f\left(re^{i\theta}\right)\right|d\theta
\end{aligned}
\end{equation}
for the small positive number $\varepsilon$ satisfying $0<\varepsilon<\frac{\pi}{8n}$ and the positive number $r\in(0, +\infty)\setminus  F$ and $r\geq \mathcal{R}_0.$ Here and in what follows,
\begin{equation}\label{eq3.130}
\varepsilon_0(\varepsilon)=\frac{\varepsilon}{\omega_{0,2\varepsilon}}=\frac{(\beta_{0,2\varepsilon}-\alpha_{0,2\varepsilon})\varepsilon}{\pi}.
\end{equation}
Next we denote by $S_{\alpha_{0,2\varepsilon+\varepsilon_0(\varepsilon)}, \beta_{0,2\varepsilon+\varepsilon_0(\varepsilon)}}(r,f)$
the Nevanlinna$'$s angular characteristic function on the angular domain
\begin{equation}\label{eq3.131}
\overline{\Omega}(\alpha_{0,2\varepsilon+\varepsilon_0(\varepsilon)},\beta_{0,2\varepsilon+\varepsilon_0(\varepsilon)})
=\left\{z\in\Bbb{C}:\alpha_{0,2\varepsilon+\varepsilon_0(\varepsilon)}\leq \arg z\leq \beta_{0,2\varepsilon+\varepsilon_0(\varepsilon)}\right\}
\end{equation}
 with
  \begin{equation}\label{eq3.132}
 \alpha_{0,2\varepsilon+\varepsilon_0(\varepsilon)}=\alpha_0+2\varepsilon+\varepsilon_0(\varepsilon) \, \, \text{and} \, \, \beta_{0,2\varepsilon+\varepsilon_0(\varepsilon)}= \beta_0-2\varepsilon-\varepsilon_0(\varepsilon).
 \end{equation}
 Next we use \eqref{eq3.131}, \eqref{eq3.132} and the line of the reasoning for proof of \eqref{eq3.79} in Subcase 2.1 of the proof of  Claim 2I to deduce
\begin{equation}\label{eq3.133}
S_{\alpha_{0,2\varepsilon+\varepsilon_0(\varepsilon)}, \beta_{0,2
\varepsilon+\varepsilon_0(\varepsilon)}}(r,f)=O(1)
\end{equation}
for the positive number $r$ such that $r\geq \mathcal{R}_0,$ say. From \eqref{eq3.104}, \eqref{eq3.131}-\eqref{eq3.133} and Lemma \ref{Lemma2.13}, we deduce that there exists some positive constant $c_1$ and there exists some set $F\subset (0, +\infty)$ of finite logarithmic measure, such that when $r\not\in F$ and $r\rightarrow\infty,$ the relation
\begin{equation}\label{eq3.134}
\begin{aligned}
\log |f(re^{i\theta})|
&=c_1r^{\omega_{0,2\varepsilon+\varepsilon_0(\varepsilon)}}
\sin\left(\omega_{0,2\varepsilon+\varepsilon_0(\varepsilon)}\left(\theta-\alpha_{0,2\varepsilon+\varepsilon_0(\varepsilon)}\right)\right)
+o\left(r^{\omega_{0,2\varepsilon+\varepsilon_0(\varepsilon)}}\right)\\
\end{aligned}
\end{equation}
holds uniformly in $\theta \in \left\{ \theta\in [0,2\pi):\alpha_{0,2\varepsilon+\varepsilon_0(\varepsilon)}\leq\theta\leq \beta_{0,2\varepsilon+\varepsilon_0(\varepsilon)}\right\}.$ Here and in what follows
\begin{equation}\label{eq3.135}
\omega_{0,2\varepsilon+\varepsilon_0(\varepsilon)}
=\frac{\pi}{\beta_{0,2\varepsilon+\varepsilon_0(\varepsilon)}-\alpha_{0,2\varepsilon+\varepsilon_0(\varepsilon)}}
=\frac{\pi}{\beta_{0,2\varepsilon}-\alpha_{0,2\varepsilon}-2\varepsilon_0(\varepsilon)}.
\end{equation}
Similarly, from Lemma \ref{Lemma2.13} and \eqref{eq3.79} for $j=0$ we deduce that there exists some positive constant $c_0$ and there exists some set $F\subset (0, +\infty)$ of finite logarithmic measure, such that the relation
\begin{equation}\label{eq3.136}
\begin{aligned}
\log |f(re^{i\theta})|
&=c_0r^{\omega_{0,2\varepsilon}}
\sin\left(\omega_{0,2\varepsilon}\left(\theta-\alpha_{0,2\varepsilon}\right)\right)
+o\left(r^{\omega_{0,2\varepsilon}}\right)\\
\end{aligned}
\end{equation}
holds uniformly in $\theta \in \left\{ \theta\in [0,2\pi):\alpha_{0,2\varepsilon}\leq\theta\leq \beta_{0,2\varepsilon}\right\},$ when $r\not\in F$ and $r\rightarrow\infty.$ Here and in what follows, $\omega_{0,2\varepsilon}$ is defined as \eqref{eq3.81} for $j=0.$ From \eqref{eq3.110}, \eqref{eq3.136} and the supposition $f(a_{0,r})\not\in\{a_1,a_2,a_3,\infty\}$ we deduce
\begin{equation}\label{eq3.137}
\frac{c_0}{2}\sin\frac{\pi\varepsilon}{\beta_{0,2\varepsilon}-\alpha_{0,2\varepsilon}}\left(\frac{r}{2}\right)^{\omega_{0,2\varepsilon}}\leq \log |f(a_{0,r})|
\leq 2c_0\left(\frac{r}{2}\right)^{\omega_{0,2\varepsilon}},
\end{equation}
 as $r/2\not\in F,$ $f(a_{0,r})\not\in\{a_1,a_2,a_3,\infty\}$ and $r\rightarrow\infty.$ On the other hand, from \eqref{eq3.129}, \eqref{eq3.132}, \eqref{eq3.134} and \eqref{eq3.135} we have
\begin{equation}\label{eq3.138}
\begin{aligned}
 & \quad T(D_{0,2\varepsilon},a_{0,r},f)\\
  &\geq \frac{1}{2\pi}\frac{4(\tilde{R}^2-\tilde{r}^2)\tilde{R}\tilde{r}\sin\varepsilon}{(\tilde{R}+\tilde{r})^4}\left(\sin\frac{\pi \varepsilon}{\beta_{0,2\varepsilon}-\alpha_{0,2\varepsilon}}\right)\int_{\alpha_{0,2\varepsilon}+\varepsilon_0(\varepsilon)}
 ^{\beta_{0,2\varepsilon}-\varepsilon_0(\varepsilon)}
 \log^{+}\left|f\left(re^{i\theta}\right)\right|d\theta\\
  &\geq  \frac{1}{2\pi}\frac{4(\tilde{R}^2-\tilde{r}^2)\tilde{R}\tilde{r}\sin\varepsilon}{(\tilde{R}+\tilde{r})^4}\left(\sin\frac{\pi \varepsilon}{\beta_{0,2\varepsilon}-\alpha_{0,2\varepsilon}}\right)\int_{\alpha_{0,2\varepsilon}+\varepsilon_0(\varepsilon)+\varepsilon}
 ^{\beta_{0,2\varepsilon}-\varepsilon_0(\varepsilon)-\varepsilon}
 \log^{+}\left|f\left(re^{i\theta}\right)\right|d\theta\\
 &\geq  \frac{1}{2\pi}\frac{4(\tilde{R}^2-\tilde{r}^2)\tilde{R}\tilde{r}\sin\varepsilon}{(\tilde{R}+\tilde{r})^4}\left(\sin\frac{\pi \varepsilon}{\beta_{0,2\varepsilon}-\alpha_{0,2\varepsilon}}\right)\left(\sin
 \frac{\pi\varepsilon}{\beta_{0,2\varepsilon}-\alpha_{0,2\varepsilon}-2\varepsilon_0(\varepsilon)}\right)\\
 &\quad \times\int_{\alpha_{0,2\varepsilon}+\varepsilon_0(\varepsilon)+\varepsilon} ^{\beta_{0,2\varepsilon}-\varepsilon_0(\varepsilon)-\varepsilon}\frac{c_1r^{\omega_{0,2\varepsilon+\varepsilon_0(\varepsilon)}}}{2} d\theta\\
 &=\frac{1}{2\pi}\frac{4(\tilde{R}^2-\tilde{r}^2)\tilde{R}\tilde{r}\sin\varepsilon}{(\tilde{R}+\tilde{r})^4}\left(\sin\frac{\pi \varepsilon}{\beta_{0,2\varepsilon}-\alpha_{0,2\varepsilon}}\right)\left(\sin
 \frac{\pi\varepsilon}{\beta_{0,2\varepsilon}-\alpha_{0,2\varepsilon}-2\varepsilon_0(\varepsilon)}\right)\\
 &\quad \times\frac{\beta_{0,2\varepsilon}-\alpha_{0,2\varepsilon}-2\varepsilon_0(\varepsilon)-2\varepsilon}{2}
 c_1r^{\omega_{0,2\varepsilon+\varepsilon_0(\varepsilon)}},
\end{aligned}
\end{equation}
 as $r/2\not\in F,$ $r\not\in F,$ $f(a_{0,r})\not\in\{a_1,a_2,a_3,\infty\}$ and $r\rightarrow\infty.$ Last, from Lemma \ref{Lemma2.16} we have
\begin{equation}\label{eq3.139}
 T\left(D_{0,2\varepsilon},a_{0,r},f\right)
 \leq \sum\limits_{l=1}^{3}N\left(D_{0,2\varepsilon},a_{0,r},\frac{1}{f-a_l}\right)+S(D_{0,2\varepsilon},a_{0,r},f),
\end{equation}
where
\begin{equation} \label{eq3.140}
\begin{aligned}
S(D_{0,2\varepsilon},a_{0,r},f)&=m\left(D_{0,2\varepsilon},a_{0,r},\frac{f'}{f}\right)
+\sum\limits_{l=1}^{3}m\left(D_{0,2\varepsilon},a_{0,r},\frac{f'}{f-a_l}\right)\\
&\quad +3\left(\log^{+}\frac{6}{\delta_0}+\log^{+}\frac{\delta_0}{6}+\log2\right)+\log 6-\log|f'(a_{0,r})| \\
&\quad+\sum\limits_{l=1}^{3} \left(\log|f(a_{0,r})-a_l|+\varepsilon(a_l,D_{0,2\varepsilon})\right)
\end{aligned}
\end{equation}
with $\delta_0=\min\limits_{1\leq j< k \leq 3}|a_j-a_k|$ and $\varepsilon(a_l,D_{0,2\varepsilon})\leq \log^{+}|a_l|+\log 2$ for $l\in\{1,2,3\}.$ Noting that $G_{\tilde{D}_0}\left(\tilde{z},\tilde{a}_{0,\tilde{r}}\right)$ denotes the Green function for the half-disc
\begin{equation}\nonumber
\tilde{D}_0=\left\{z\in\Bbb{C}: |z|<\tilde{R} \, \,  \text{and} \, \, \text{Im}(\tilde{z})>0\right\},
\end{equation}
we use the formula of the Green function \eqref{eq2.11} for the half-disc
\begin{equation}\nonumber
D=\left\{z\in\Bbb{C}: |z|<R \, \,  \text{and} \, \, \text{Im}(z)>0\right\}
\end{equation}
in Remark \ref{Remark2.5X2}.1 to have
\begin{equation}\label{eq3.141}
\begin{split}
&\quad G_{\tilde{D}_0}\left(\tilde{z},\tilde{a}_{0,\tilde{r}}\right)=\log\left|\frac{\tilde{R}^2
-\overline{\tilde{a}_{0,\tilde{r}}}\tilde{z}}{\tilde{R}(\tilde{z}-\tilde{a}_{0,\tilde{r}})}
 \cdot\frac{\tilde{R}(\tilde{z}-\overline{\tilde{a}_{0,\tilde{r}}})}{\tilde{R}^2-\tilde{a}_{0,\tilde{r}}\tilde{z}} \right|\\
 &=\log\frac{1}{\left|\tilde{z}-\tilde{a}_{0,\tilde{r}}\right|} +\omega_{\tilde{D}_0}\left(\tilde{z},\tilde{a}_{0,\tilde{r}}\right)
  \end{split}
\end{equation}
with
\begin{equation}\label{eq3.142}
\omega_{\tilde{D}_0}\left(\tilde{z},\tilde{a}_{0,\tilde{r}}\right)=:\log\left|\frac{\left(\tilde{R}^2
-\overline{\tilde{a}_{0,\tilde{r}}}\tilde{z}\right)(\tilde{z}-\overline{\tilde{a}_{0,\tilde{r}}})}{\tilde{R}^2-\tilde{a}_{0,\tilde{r}}\tilde{z}}
\right|
\end{equation}
for each $\tilde{z}\in\tilde{D}_0\cup\partial{\tilde{D}_0}=\left\{z\in\Bbb{C}: |\tilde{z}|\leq \tilde{R} \, \,  \text{and} \, \, \text{Im}(\tilde{z})\geq 0\right\}\setminus\{\tilde{a}_{0,\tilde{r}}\}.$ From \eqref{eq3.112}, \eqref{eq3.124} and \eqref{eq3.141} we have
\begin{equation}\label{eq3.143}
G_{D_{0,2\varepsilon}}\left(\tilde{\phi}_{\tilde{a}_{0,\tilde{r}}}(\tilde{z}),a_{0,r}\right)
=G_{\tilde{D}_0}\left(\tilde{z},\tilde{a}_{0,\tilde{r}}\right)
=\log\frac{1}{\left|\tilde{z}-\tilde{a}_{0,\tilde{r}}\right|} +\omega_{\tilde{D}_0}\left(\tilde{z},\tilde{a}_{0,\tilde{r}}\right)
\end{equation}
for each $\tilde{z}\in\tilde{D}_0\cup\partial{\tilde{D}_0}\setminus\{ \tilde{a}_{0,\tilde{r}}\}$ with $z=e^{i\alpha_{0,2\varepsilon}}{\tilde{z}}^{\frac{\beta_{0,2\varepsilon}-\alpha_{0,2\varepsilon}}{\pi}}.$
From Claim 2I we see that there exist finitely many distinct zeros of $f-b_l$ with $l\in\{1,2,3\}$ in $D_{0,2\varepsilon},$ say
$b_{l,1}, b_{l,2}, \ldots,  b_{l,N_l} \, \,  \text{with} \, \, l\in\{1,2,3\},$
such that
\begin{equation}\nonumber
|b_{l,1}|<|b_{l,2}|<\ldots<|b_{l,N_l}| \, \,  \text{with} \, \, l\in\{1,2,3\},
\end{equation}
where $N_l=N_l(r,\varepsilon)$ is some positive integer that depends only upon the positive numbers $r$ and $\varepsilon.$ Next we suppose that the multiplicity of $ b_{l,j}$ of a zero of $f-a_l$ is equal to some positive integer $m_{l,j}$ for $1\leq j\leq N_l,$ $j\in\Bbb{Z}^{+}$ and $l\in\{1,2,3\}.$ Then, from \eqref{eq3.143}, \eqref{eq3.128}, the reasoning in Goldberg-Ostrovskii \cite[p.306]{Goldberg1970}, the supposition  $f(a_{0,r})\not\in\{a_1,a_2,a_3,\infty\}$ and the definition \eqref{eq2.4} of the counting function for the meromorphic function $f$ with the center at $a_{0,r}$ for $D_{0,2\varepsilon},$ we have
\begin{equation}\label{eq3.144}
\begin{aligned}
&\quad N\left(D_{0,2\varepsilon},a_{0,r}, \frac{1}{f-a_l}\right)\\
&=\sum\limits_{j=1}^{N_l} m_{l,j}G_{D_{0,\varepsilon}}\left(b_{l,j},a_{0,r}\right) +n(0,a_{0,r},f)\omega_{D_{0,\varepsilon}}\left(a_{0,r},a_{0,r}\right)\\
&=\sum\limits_{j=1}^{N_l} m_{l,j}G_{D_{0,\varepsilon}}\left(b_{l,j},a_{0,r}\right)=\sum\limits_{j=1}^{N_l} m_{l,j}G_{\tilde{D}_0}\left(\tilde{b}_{l,j},\tilde{a}_{0,\tilde{r}}\right)\\
\end{aligned}
\end{equation}
\begin{equation*}
\begin{aligned}
&=\sum\limits_{j=1}^{N_l} m_{l,j}\log\left|\frac{\tilde{R}^2-\overline{\tilde{a}_{0,\tilde{r}}}\tilde{b}_{l,j}}
{\tilde{R}(\tilde{b}_{l,j}-\tilde{a}_{0,\tilde{r}})}\cdot\frac{\tilde{R}(\tilde{b}_{l,j}-\overline{\tilde{a}_{0,\tilde{r}}})}{\tilde{R}^2-\tilde{a}_{0,\tilde{r}}\tilde{b}_{l,j}} \right|=O(1) \, \, \text{for} \, \, l\in\{1,2,3\},
\end{aligned}
\end{equation*}
as
$\tilde{r}=\left(\frac{r}{2}\right)^{\frac{\pi}{\beta_{0,2\varepsilon}-\alpha_{0,2\varepsilon}}}
=\left(\frac{1}{2}\right)^{\frac{\pi}{\beta_{0,2\varepsilon}-\alpha_{0,2\varepsilon}}}\tilde{R}\rightarrow\infty.$ Here and in what follows,
\begin{equation}\label{eq3.145}
b_{l,j}=\tilde{\phi}_{\tilde{a}_{0,\tilde{r}}}\left(\tilde{b}_{l,j}\right)
=e^{i\alpha_{0,2\varepsilon}}{\tilde{b}_{l,j}}^{\frac{\beta_{0,2\varepsilon}-\alpha_{0,2\varepsilon}}{\pi}}
\end{equation}
for $1\leq j\leq N_l,$ $j\in\Bbb{Z}^{+}$ and $l\in\{1,2,3\}.$ On the other hand, from Lemma \ref{Lemma2.18} we see that there exists some  set
$F\subset (1, +\infty)$ of finite logarithmic measure such that
\begin{equation} \label{eq3.146}
\left|\frac{f'(z)}{f(z)}\right|\leq |z|^{\rho-1+\varepsilon} \, \, \text{and} \, \, \left|\frac{f'(z)}{f(z)-a_l}\right|\leq |z|^{\rho-1+\varepsilon},
\end{equation}
 as $|z|\not\in F\cup [0,1]$ and $|z|\rightarrow\infty.$ From the formula (2.1.1) in Zheng \cite[p.27]{Zheng2009}, we have
\begin{equation} \label{eq3.147}
\frac{1}{2\pi}\int_{\partial{D_{0,2\varepsilon}}}\frac{\partial {G}_{D_{0,2\varepsilon}}(\zeta, z )}{\partial{\mathfrak{n}}}ds=1 \, \, \text{for each } \,  z\in D_{0,2\varepsilon}.
\end{equation}
From  \eqref{eq3.146}, \eqref{eq3.147} and the definitions of the proximity functions and the Nevanlinna characteristic functions of $\frac{f'}{f}$  and $\frac{f'}{f-a_l}$ for $l\in\{1,2,3\}$ with the center at $a_{0,r}$ for $D_{0,2\varepsilon},$ we have
\begin{equation} \label{eq3.148}
m\left(D_{0,2\varepsilon},a_{0,r},\frac{f'}{f}\right)+\sum\limits_{l=1}^{3}m\left(D_{0,2\varepsilon},a_{0,r},\frac{f'}{f-a_l}\right)
\leq 4(\rho-1+\varepsilon)\log r,
\end{equation}
as $r\not\in F\cup [0,1]$ and $r\rightarrow\infty.$  From \eqref{eq3.110} and \eqref{eq3.137} we deduce
\begin{equation}\label{eq3.149}
\begin{aligned}
&\quad \sum\limits_{l=1}^{3} \left(\log|f(a_{0,r})-a_l|+\varepsilon(a_l,D_{0,2\varepsilon})\right)\\
&\leq \sum\limits_{l=1}^{3} \left(\log^{+} (|f(a_{0,r})|+|a_l|)+\varepsilon(a_l,D_{0,2\varepsilon})\right)\leq 7c_0\left(\frac{r}{2}\right)^{\omega_{0,2\varepsilon}},
\end{aligned}
\end{equation}
 as $r/2\not\in F\cup [0,1],$ $f(a_{0,r})\not\in\{a_1,a_2,a_3,\infty\}$ and $r\rightarrow\infty.$ From Remark \ref{Remark2.3}, Lemma \ref{Lemma2.3} (i), the assumption $\rho(f)=\rho<\infty$ and the assumption that $f$ and $g$ share $a_1,$ $a_2,$ $a_3$ IM and $a_4$ CM we have
 \begin{equation}\nonumber
 N_0\left(r,\frac{1}{f'}\right)+N_0\left(r,\frac{1}{g'}\right)=O(\log r),
 \end{equation}
 which implies that $f$ and $g$ have finitely many critical values including $a_1,$ $a_2$ and $a_3$ possibly in the complex plane. Combining this with Lemma \ref{Lemma2.20} and the assumption $\rho(f)=\rho<\infty,$ we deduce that the set of finite singular values of the inverse function $f^{-1}$  is bounded. Therefore, from Lemma \ref{Lemma2.19} we see that there exist constants $L_0 > 0$ and $M _0> 0$ such that if $|z| > L_0$ and $|f(z)| > M_0$ then
\begin{equation}\label{eq3.150}
\left|z\frac{f'(z)}{f(z)}\right|\geq \frac{\log \left|f(z)/M_0\right|}{C_0},
\end{equation}
where  $C_0$ is a positive absolute constant, in particular independent of $f,$ $L_0$ and $M_0.$ From \eqref{eq3.110}, \eqref{eq3.137}, \eqref{eq3.150} and the supposition $f(a_{0,r})\not\in\{a_1,a_2,a_3,\infty\}$ we see that there exists some set $F\subset (1, +\infty)$ of finite logarithmic measure such that
\begin{equation}\label{eq3.151}
\left|\frac{1}{f'(a_{0,r})}\right|\leq \frac{c_0|a_{0,r}|}{\log\left|f(a_{0,r})\right|-\log M_0}\cdot\frac{1}{|f(a_{0,r})|} \, \, \text{with} \, \, l\in\{1,2,3\},
\end{equation}
 as $|a_{0,r}|=r/2\not\in F\cup [0,1],$ $f(a_{0,r})\not\in\{a_1,a_2,a_3,\infty\}$ and $r\rightarrow\infty.$ From \eqref{eq3.115} and \eqref{eq3.81} for $j=0$ we have $1<\omega_{0,2\varepsilon}<\infty.$  Combining this with \eqref{eq3.110}, \eqref{eq3.137} and \eqref{eq3.151} we deduce
\begin{equation}\label{eq3.152}
\begin{aligned}
&\quad\log^{+}\left|\frac{1}{f'(a_{0,r})}\right|\leq \log^{+}\frac{c_0|a_{0,r}|}{\left|\log\left|f(a_{0,r})\right|-\log M_0\right|}+\log^{+}\left|\frac{1}{f(a_{0,r})}\right|\\
&\leq\log^{+}\frac{\frac{c_0r}{2}}{\left|\frac{c_0}{2}\sin\frac{\pi\varepsilon}{\beta_{0,2\varepsilon}-\alpha_{0,2\varepsilon}}\left(\frac{r}{2}\right)^{\omega_{0,2\varepsilon}}-\log M_0\right|}=0,
\end{aligned}
\end{equation}
 as $|a_{0,r}|=r/2\not\in F\cup [0,1],$ $f(a_{0,r})\not\in\{a_1,a_2,a_3,\infty\}$ and $r\rightarrow\infty.$ From \eqref{eq3.110}, \eqref{eq3.137}, \eqref{eq3.152}, Lemma \ref{Lemma2.18} and the supposition $f(a_r)\not\in\{a_1,a_2,a_3,\infty\}$
we see that there exists some set $F\subset (1, +\infty)$ of finite logarithmic measure, such that
\begin{equation}\label{eq3.153}
\begin{aligned}
&\quad \left|\log|f'(a_{0,r})|\right|=\log^{+}|f'(a_{0,r})|+\log^{+}\left|\frac{1}{f'(a_{0,r})}\right|=\log^{+}|f'(a_{0,r})| \\
&\leq \log^{+}|f(a_{0,r})|+\log^{+}\left|\frac{f'(a_{0,r})}{f(a_{0,r})}\right|\leq 2c_0\left(\frac{r}{2}\right)^{\omega_{0,2\varepsilon}}+\log^{+}\left(\frac{r}{2}\right)^{\rho-1+\varepsilon}\\
&\leq 3c_0\left(\frac{r}{2}\right)^{\omega_{0,2\varepsilon}},
\end{aligned}
\end{equation}
as $|a_{0,r}|=r/2\not\in F\cup [0,1],$ $f(a_r)\not\in\{a_1,a_2,a_3,\infty\}$ and $r\rightarrow\infty.$ From \eqref{eq3.140}, \eqref{eq3.148}, \eqref{eq3.149}, \eqref{eq3.153} we have
\begin{equation} \label{eq3.154}
\begin{aligned}
S(D_{0,2\varepsilon},a_{0,r},f)&=4(\rho-1+\varepsilon)\log r+  3c_0\left(\frac{r}{2}\right)^{\omega_{0,2\varepsilon}}+7c_0\left(\frac{r}{2}\right)^{\omega_{0,2\varepsilon}}+O(1)\\
&\leq 11c_0\left(\frac{r}{2}\right)^{\omega_{0,2\varepsilon}},
\end{aligned}
\end{equation}
 as $|a_{0,r}|=r/2\not\in F\cup [0,1],$ $
 r\not\in F\cup [0,1],$ $f(a_{0,r})\not\in\{a_1,a_2,a_3,\infty\}$ and $r\rightarrow\infty.$ From \eqref{eq3.139}, \eqref{eq3.140}, \eqref{eq3.144}, \eqref{eq3.145}, \eqref{eq3.154} we have
\begin{equation} \label{eq3.155}
T(D_{0,2\varepsilon},a_{0,r},f)\leq 11c_0\left(\frac{r}{2}\right)^{\omega_{0,2\varepsilon}},
\end{equation}
as $|a_{0,r}|=r/2\not\in F\cup [0,1],$ $r\not\in F\cup [0,1],$ $f(a_{0,r})\not\in\{a_1,a_2,a_3,\infty\}$ and as $r\rightarrow\infty.$
From \eqref{eq3.128}, \eqref{eq3.138} and \eqref{eq3.155} we have
\begin{equation}\label{eq3.156}
\begin{aligned}
&\quad \frac{1}{2\pi}\frac{4(\tilde{R}^2-\tilde{r}^2)\tilde{R}\tilde{r}\sin\varepsilon}{(\tilde{R}+\tilde{r})^4}\left(\sin\frac{\pi \varepsilon}{\beta_{0,2\varepsilon}-\alpha_{0,2\varepsilon}}\right)\left(\sin
 \frac{\pi\varepsilon}{\beta_{0,2\varepsilon}-\alpha_{0,2\varepsilon}-2\varepsilon_0(\varepsilon)}\right)\\
 &\quad \times\frac{\beta_{0,2\varepsilon}-\alpha_{0,2\varepsilon}-2\varepsilon_0(\varepsilon)-2\varepsilon}{2}
 c_1r^{\omega_{0,2\varepsilon+\varepsilon_0(\varepsilon)}} \leq 11c_0\left(\frac{r}{2}\right)^{\omega_{0,2\varepsilon}},
\end{aligned}
\end{equation}
 as $|a_{0,r}|=r/2\not\in F\cup [0,1],$ $r\not\in F\cup [0,1],$ $f(a_{0,r})\not\in\{a_1,a_2,a_3,\infty\}$ and $r\rightarrow\infty.$ From \eqref{eq3.130}, \eqref{eq3.135}, \eqref{eq3.71} and \eqref{eq3.81} for $j=0$  we see that
 \begin{equation}\label{eq3.157}
  \omega_{0,2\varepsilon}<\omega_{0,2\varepsilon+\varepsilon_0(\varepsilon)}
 \end{equation}
 for any given small positive number $\varepsilon$ satisfying $0<\varepsilon<\frac{\pi}{8n}.$
 From \eqref{eq3.156}, \eqref{eq3.157} and \eqref{eq3.128} we derive a contradiction, when $|a_{0,r}|=r/2\not\in F\cup [0,1],$ $r\not\in F\cup [0,1],$ $f(a_{0,r})\not\in\{a_1,a_2,a_3,\infty\}$ and $r\rightarrow\infty.$ This completes the proof of Theorem \ref{Theorem1.1}.
\vskip 2mm
\par
\section{On a question of Gary G. Gundersen concerning the nonexistence of two distinct and non-constant meromorphic functions sharing three distinct values DM and a fourth value CM }
\vskip 2mm
\par In 1979, G.G.Gundersen proposed the following two questions in Gundersen \cite{Gundersen1979}:
\vskip 2mm
\par
\begin{question}\rm{}(\cite[p.458]{Gundersen1979}) \label{question4.1}
Do there exist two distinct and non-constant meromorphic functions that share two distinct values CM and share the other two distinct values DM, where the four values are four distinct values in the extended complex plane?
\end{question}
\vskip 2mm
\par
\begin{question}\rm{}(\cite[p.458]{Gundersen1979})\label{question4.2}
Do there exist two distinct and non-constant meromorphic functions that share one value CM and share the other three distinct values DM, where the four values are four distinct values in the extended complex plane?
\end{question}
\vskip 2mm
\par In 1983, Gundersen\cite[Theorem 1]{Gundersen1983} proved that if two distinct and non-constant meromorphic functions share two values CM and share two other values IM, where the four values are four distinct values in the extended complex plane, then the two meromorphic functions share all four values CM. This implies that Gundersen\cite[Theorem 1]{Gundersen1983} gave a negative answer to Question \ref{question4.1}. But Question \ref{question4.2} is still open by now. The following question is a special case of Question \ref{question4.2}:
\vskip 2mm
\par
\begin{question}\rm{}\label{question4.3}
Do there exist two distinct and non-constant entire functions that share three distinct finite values DM?
\end{question}
In this direction, we recall the following result in Mues \cite{Mues1989F1} that completely resolved Question \ref{question4.3}:
\vskip 2mm
\par
\begin{theorem}\rm{}(\cite[pp.109-117]{Mues1989F1})\label{TheoremI}
There do not exist two distinct and non-constant entire functions that share three distinct finite values $a_1,$ $a_2,$ $a_3$ DM.
\end{theorem}
\vskip 2mm
\par For the existence of three distinct and non-constant meromorphic functions sharing four distinct values IM, where the four distinct shared values are in the extended complex plane that are neither CM shared values nor DM shared values, we refer to Steinmetz \cite[Theorem 2]{Steinmetz1988}.
For the existence of two distinct and non-constant meromorphic functions sharing four distinct values DM, where the four distinct shared values are in the extended complex plane, we refer to  Reinders \cite[Theorem 1]{Reinders1992}.
\vskip 2mm
\par From Theorem \ref{Theorem1.1} we get the following result that gives a negative answer to Question \ref{question4.2} for the finite order meromorphic functions:
\vskip 2mm
\par
\begin{theorem}\rm{}\label{Theorem4.1}
There do not exist two distinct and non-constant finite order meromorphic functions that share $a_1,$ $a_2,$ $a_3$ DM and share $a_4$ CM, where
$a_1,$ $a_2,$ $a_3,$ $a_4$ are four distinct complex values in the extended complex plane.
\end{theorem}
\vskip 2mm
\par
From Example \ref{Example1.1} above and the following example we see that the number of the three distinct DM shared values and the number of one CM shared value in Theorem \ref{Theorem4.1}, in a sense, are best possible:
\vskip 2mm
\par
\begin{example} \rm{}(\cite[p.94]{Steinmetz1988}) \, Let $\mathcal{P}$ denote the Weierstrass $\mathcal{P}$-function with a pair of primitive periods $2\omega-\omega'$ and $\omega+\omega'.$ Then
\begin{equation}\nonumber
F(z)=\frac{\left(\mathcal{P}(z)-\mathcal{P}(\omega/2)\right)\left(\mathcal{P}(z)-\mathcal{P}(3\omega/2)\right)^2}{\mathcal{P}(z)-\mathcal{P}(\omega)}, \, \, G(z)=F(z+\omega)
\end{equation}
 and $H(z)=F(z+\omega')$ share the value $0,$ $a_1,$ $a_2,$ $\infty$ IM, where $a_1$ and $a_2$ are two distinct complex numbers such that $a_1/a_2$ with $a_1/a_2\neq -1$ is a cube root of unity, while $\omega$ and $\omega'$ are two complex numbers such that $2\omega-\omega'$ and $\omega+\omega'$ are linearly independent on $\Bbb{R}.$ That is to say, for each pair of real numbers $\lambda_1$ and $\lambda_2$ such that $|\lambda_1|+|\lambda_2|\neq 0,$ we have $\lambda_1(2\omega-\omega')+\lambda_2(\omega+\omega')\neq 0.$  Moreover, we can verify that the four values $0,$ $a_1,$ $a_2,$ $\infty$ are neither CM shared values nor DM shared values of $F(z),$ $G(z)$ and $H(z).$
\end{example}
\vskip 2mm
\par
\section {Concluding remarks}
\vskip 2mm
\par From Question \ref{Question1.1} and Theorem \ref{Theorem1.1} we propose the following conjecture:
\vskip 2mm
\par
\begin{conjecture}\rm{}\label{Conjecture6.1}
Suppose that $f$ and $g$ are two distinct transcendental meromorphic functions of infinite order, if $f$ and $g$ share $a_1,$ $a_2,$ $a_3$ IM and $a_4$ CM, where $a_1,$ $a_2,$ $a_3,$ $a_4$ are four distinct complex values in the extended complex plane, then $f$ and $g$ share the four values $a_1,$ $a_2,$ $a_3,$ $a_4$ CM.
\end{conjecture}
\vskip 2mm
\par From Question \ref{question4.2} amd Theorem \ref{Theorem4.1} we propose the following conjecture:
\vskip 2mm
\par
\begin{conjecture}\rm{}\label{Conjecture6.2}
There do not exist two distinct and non-constant infinite order meromorphic functions that share $a_1,$ $a_2,$ $a_3$ DM and share $a_4$ CM, where
$a_1,$ $a_2,$ $a_3,$ $a_4$ are four distinct complex values in the extended complex plane.
\end{conjecture}
\vskip 2mm
\par
\def\cprime{$'$}
\providecommand{\bysame}{\leavevmode\hbox to3em{\hrulefill}\thinspace}
\providecommand{\MR}{\relax\ifhmode\unskip\space\fi MR }
\providecommand{\MRhref}[2]{%
  \href{http://www.ams.org/mathscinet-getitem?mr=#1}{#2}
}
\providecommand{\href}[2]{#2}

\end{document}